\documentclass[a4paper,a4wide,11pt]{article}

\usepackage[usenames,dvipsnames]{pstricks}
\usepackage{epsfig}
\usepackage{pst-grad} 
\usepackage{pst-plot} 

\usepackage{pgf,tikz}
\usepackage{verbatim}
\usepackage{a4,amsmath,amssymb,amsthm,euscript}
\usepackage{subfigure}
\usetikzlibrary{arrows}

\newtheorem{theorem}{Theorem}[section]
\newtheorem{proposition}[theorem]{Proposition}

\newtheorem{definition}[theorem]{Definition}
\newtheorem{lemma}[theorem]{Lemma}
\newtheorem{corollary}[theorem]{Corollary}

\newtheorem{problem}[theorem]{Problem}

\theoremstyle{remark}
\newtheorem{remark}[theorem]{Remark}

\definecolor{ag1}{rgb}{0.93, 0.57, 0.13}
\definecolor{or}{rgb}{0.93, 0.57, 0.13}
\definecolor{ag}{rgb}{0.5, 0.0, 0.13}
\definecolor{ag2}{rgb}{0.66, 0.66, 0.66}

\newcommand\N{\mathbb{N}}
\newcommand\R{\mathbb{R}}

\newcommand{\St}{{\mathcal St}}

\newcommand{\Card}{\#}
\newcommand\C{\mathcal{C}}
\newcommand{\eps}{\varepsilon}
\newcommand{\HH}{{\mathcal H}}
\renewcommand{\H}{\HH^1}

\newcommand{\forget}[1]{}

\def\dist{\textup{dist}\,}
\def\diam{\textup{diam}\,}

\def\Lip{\textup{Lip}}
\def\Om{\Omega}

\makeatletter
\renewcommand{\thefigure}{\ifnum \c@section>\z@ \thesection.\fi
 \@arabic\c@figure}
\@addtoreset{figure}{section}
\makeatother

\begin{document}

\title{Regularity for the optimal compliance problem with length penalization}
\author{Antonin Chambolle\footnote{CMAP, Ecole Polytechnique, CNRS, 91128 Palaiseau Cedex, France}\and Jimmy Lamboley\footnote{Universit\'e Paris-Dauphine, PSL Research University, UMR 7534, CEREMADE, Place du Mar\'echal de Lattre de Tassigny, 75775 Paris, France}\and Antoine Lemenant\footnote{Universit\'e Paris Diderot - LJLL CNRS UMR 7598 - Paris 7
75205 Paris Cedex 13, France}\and Eugene Stepanov\footnote{St.Petersburg Branch
of the Steklov Mathematical Institute of the Russian Academy of Sciences,
Fontanka 27,
191023 St.Petersburg,%
Russia
and
Department of Mathematical Physics, Faculty of Mathematics and Mechanics,
St. Petersburg State University, Universitet\-skij pr.~28, Old Peterhof,
198504 St.Peters\-burg, Russia
and ITMO University, Russia
}
\thanks{
This work was partially supported by the projects ANR-12-BS01-0007 OPTIFORM and ANR-12-BS01-0014-01 GEOMETRYA  financed by the French Agence Nationale de la Recherche (ANR). The authors acknowledge the support of the project MACRO (Mod\`eles d'Approximation Continue de R\'eseaux Optimaux), funded by the Programme Gaspard Monge pour l'Optimisation of EDF and the Fondation Math\'ematiques Jacques Hadamard.
The fourth author also acknowledges the support of the St.Petersburg State University grants \#6.38.670.2013 and
\#6.38.223.2014, 
the Russian government grant NSh-1771.2014.1,
the RFBR grant \#14-01-00534 and
the project
2010A2TFX2 ``Calcolo delle variazioni''
 of the
Italian Ministry of Research. 
}}

\date{\today}


\maketitle
\begin{abstract}

We study the regularity and topological structure of a compact connected set $S$ minimizing the ``compliance" functional with a length penalization. The compliance is, here, the work of the force applied to a membrane which is attached along the set $S$.

This shape optimization problem, which can be interpreted as that of finding the best location for attaching a membrane subject to a given external force, can be seen as an elliptic PDE version of the minimal average distance problem.

We prove that minimizers in the given region consist of a finite number of smooth curves which meet only at triple points with angles of 120 degrees, contain no loops, and possibly touch the boundary of the region only tangentially. The proof uses, among other ingredients,  some  tools from the theory of free discontinuity problems (monotonicty formula, flatness improving estimates, blow-up limits), but   adapted to the  specific problem of min-max type studied here, which constitutes a notable difference with the classical setting and  may be useful also for similar other problems.

\end{abstract}

\newpage

\tableofcontents

\newpage

\section{Introduction and main results}

Consider a bounded open set $\Omega\subset \R^2$.
Let $\Sigma\subset\overline{\Omega}$ a closed set, and
for $f\in L^2(\Omega)$ and $u\in H_0^1(\Omega\setminus\Sigma)$ denote
\[
E(u):=\frac{1}{2} \int_{\Omega\setminus\Sigma} |\nabla u|^2\,dx -\int_\Omega fu\,dx.
\]
Letting $u_\Sigma\in H_0^1(\Omega\setminus\Sigma)$ stand for the unique minimizer of $E(\cdot)$ over $H_0^1(\Omega\setminus\Sigma)$,  it is classical that
\begin{equation}\label{eq_DeltSigf1}
-\Delta u_\Sigma=f
\end{equation}
in the weak sense in $\Omega\setminus\Sigma$.
From physical point of view we can think of $\Omega$ as a membrane, and $\Sigma$ as
the ``glue line'' attaching it to some fixed base
(i.e.\ preventing the displacement). Then $u_\Sigma$
is the displacement of the membrane with the boundary fixed
subject to force field $f$.
The compliance of the membrane can then be defined
as
\[
\C(\Sigma):=-E(u_\Sigma)= \frac{1}{2} \int_{\Omega\setminus\Sigma} |\nabla u_\Sigma|^2\,dx= \frac{1}{2} \int_\Omega fu_\Sigma\,dx,
\]
the equivalence of the two above expressions for $\C$ being due to~\eqref{eq_DeltSigf1}.
It can be seen as a measure of how much the membrane resists to the force
field $f$, and is proportional to the work of the force.
In this paper we study the following problem.

\begin{problem}\label{pb_compl_pen1}
Given $\lambda>0$, find a set $\Sigma \subset \overline{\Om}$ minimizing the functional $\mathcal{F}_\lambda$ defined by
\[
\mathcal{F}_\lambda(\Sigma'):= \C(\Sigma')+\lambda\H(\Sigma')=-E(u_{\Sigma'})+\lambda \H(\Sigma')
\]
among all closed connected sets $\Sigma'\subset\overline{\Om}$.
\end{problem}

A similar problem, previously introduced and studied in the literature (see
e.g.~\cite{BuSan07,ButMaiSte09,Til12}) is that of minimizing the compliance $\C$ over
closed connected sets $\Sigma$ subject to the constraint on the length $\H(\Sigma)\leq \ell$ instead of the length penalization. In this paper we however concentrate our attention exclusively on the penalized Problem~\ref{pb_compl_pen1}, whose physical interpretation
is as follows: we are trying to find the best location $\Sigma$ for the glue to put on a membrane  in order to minimize the compliance of the latter, subject to the force $f$, while the penalization by $\lambda \H$ takes into account, for instance, the quantity (or cost) of the glue.

The existence of minimizers of Problem~\ref{pb_compl_pen1}
follows in a more or less standard way from Blaschke, \v{S}ver\'ak, and Go\l ab theorems, see Proposition~\ref{existence}  (or~\cite{BuSan07}).

A   problem similar to  Problem~\ref{pb_compl_pen1}  has been studied in~\cite{Zucco}, where the compliance $\mathcal{C}(\Sigma)$
has been replaced by $\lambda_1(\Omega\setminus \Sigma)$, the first  eigenvalue of the Laplace operator  with Dirichlet condition on
$\Sigma \cup \partial \Omega$, in a constrained version $\H(\Sigma)\leq \ell$.  In~\cite{Zucco}, the asymptotic behavior of solutions to this problem
 when $\ell$ goes to $+\infty$ has been identified.
On the other hand,    the regularity of minimizers is not known for this problem.  Since $\lambda_1(\Omega\setminus \Sigma)$ is yet an energy of elliptic type, we believe that the technics introduced in this paper
for the compliance could also serve to study its eigenvalue version, though
do not pursue this analysis here.

Another version of shape optimization problem similar to Problem~\ref{pb_compl_pen1} may be obtained substituting the standard Laplacian by the
$p$-Laplace operator.
In this case, letting $p\to \infty$, one obtains~\cite[Theorem~3]{BuSan07} in
the limit the so-called
average distance   minimization problem of purely geometric nature
\[
\min_{\Sigma \subset \overline{\Omega} \text{ closed connected }} \int_{\Omega} \dist(x,\Sigma) d\mu + \lambda \H(\Sigma).
\]
 This problem is related to a Monge-Kantorovitch problem with a
``free Dirichlet region'' which was introduced as a model
for an  optimal urban traffic network~\cite{ButOudSte02,ButSte03}.
 The topological and geometric properties of minimizers of this problem (mostly in its constrained version but also in the easier penalized one)
were studied by several authors  (see~\cite{lreview1} for a review on this problem, and~\cite{S1,S2,S3,S4,l1,st,t000} for related results on this and similar problems).

 Problem~\ref{pb_compl_pen1} is much different from the average distance minimization problem, and  to certain extent
is  closer to the Mumford-Shah problem.
We will indeed prove at the end of this paper that it 
is in a sense dual to the Mumford-Shah problem, and a
certain amount of tools used in this paper
(monotonicity formula, scheme of proof for the $C^1$ result) is inspired by
the arguments developed for the Mumford-Shah functional (we refer to~\cite{afp, d,lreview2} for  reviews on the latter). Here,  we show that the minimizers are locally smoooth inside $\Omega$.
This is in sharp contrast with the average distance problem:
 in fact, for the latter Slep\v{c}ev~\cite{slepcev} (see also~\cite{LuSlep13}) 
has shown that minimizers are, in general, not $C^1$ and using his construction one may find minimizers
with infinite and possibly not closed set of points with lack of regularity~\cite{Lu_RendPadova16}.

\subsubsection*{Main results}

The following theorem sums up the main  results of this paper.

\begin{theorem}\label{TheMainTheorem}
Let $\Omega\subset \R^2$ be a  $C^1$ domain   (i.e.\ an open bounded connected set with the boundary locally, up to rotation, the graph of a $C^1$ function),  and $\lambda\in(0,\infty)$.  If $f\in L^p(\Om)$ with $p>2$, then every solution $\Sigma$ of
Problem~\ref{pb_compl_pen1} has the following properties.

\medskip
\noindent {\sc Part I: qualitative properties.}
\begin{itemize}
\item[(i)] $\Sigma$ contains no loops (homeomorphic images of $S^1$), hence $\R^2\setminus\Sigma$ is connected.
\item[(ii)] $\Sigma$ is Ahlfors regular.
\item[(iii)] $\Sigma$ is a chord-arc set, i.e.\ it satisfies
\begin{eqnarray}
d_{\Sigma}(x,y)\leq C|x-y|,   \label{chordarc0}
\end{eqnarray}
for some constant $C>0$, where $d_\Sigma$ denotes
the geodesic distance in $\Sigma$.
\end{itemize}

\noindent {\sc Part II: regularity.}
\begin{itemize}
\item[(iv)] $\Sigma$ consists of a finite number of embedded curves, possibly intersecting at ``triple points'' where the curves are meeting by $3$ at $120$ degrees
angles. In particular, it has finite number of endpoints and finite number of
branching points, all of which are triple points as described.
\item[(v)] Such curves are locally $C^{1,\alpha}$ regular for some $\alpha\in (0,1)$ inside $\Omega$ and possibly touch $\partial \Omega$ only  tangentially.
Moreover, if $\Omega$ is a convex domain, then they are locally $C^{1,\alpha}$ in $\overline{\Omega}$.
\item[(vi)] If, further,
$f\in H^1(\Om)\cap L^p(\Om)$   with $p>2$,
then the curves of $\Sigma$ are locally $C^{2,\alpha}$ inside $\Omega$ for $\alpha=1-2/p$, and
if $f$, moreover, is locally $C^{k,\beta}$ regular for some $k\in\N$ and $\beta\in(0,1)$ then the latter curves are locally $C^{k+3,\beta}$.
\end{itemize}
\end{theorem}

We emphasize that~$(i)$ and~($iv$)-($v$) are probably the most interesting items of this result, while~($ii$) and~($iii$) might at first glance look technical.
However, besides being interesting on their own, more remarkably, the results of Part~I are  
needed to prove Part~II.
Let us stress moreover that several of the announced results still hold
without assuming  $C^1$ regularity of $\partial \Omega$. For instance, for a generic bounded open $\Om\subset \R^2$ we show in fact that
every minimizer $\Sigma$ of Problem~\ref{pb_compl_pen1} satisfies~$(i)$ and~$(ii)$, i.e.\
has no loops and is Ahlfors regular,  and for every compact $\Omega'\Subset \Omega$, the estimate~\eqref{chordarc0} still holds for $\{x,y\}\subset \Sigma\cap \Omega'$   (with the constant $C$  depending on $\Omega'$),
has a finite number of branching points and endpoints in $\Sigma \cap \Omega'$,
all the branching points are triple points where the curves are meeting by $3$ at $120$ degrees
angles,
while
the curves composing $\Sigma\cap \Omega'$ are  locally $C^{1,\alpha}$ regular.

With the developed technique at hand, it would be
also possible to get some finer results for just Lipschitz domains,
as, for instance, the fact that the optimal set $\Sigma$ will never touch a convex corner of $\partial \Omega$, but we do not pursue this analysis in details.

We finally mention that the $C^{1,\alpha}$ regularity up to
the tip of an endpoint, or up to  the branching point for a triple point, is not proven in this paper. We only know that any blow-up sequence is converging to the  global minimizer of corresponding type at the limit. However, concerning the triple points, it is quite likely that one could adapt our $\varepsilon$-regularity result (Theorem \ref{C1reg}) to obtain a similar statement,
as David did in \cite[Section 53]{d} for  Mumford-Shah minimizers.


Theorem~\ref{TheMainTheorem} is the concatenation of several results contained in the paper. Namely,~$(i)$ is Theorem~\ref{th_compl_noloop1},~$(ii)$ is Theorem~\ref{th_compl_ahl1}, ~$(iii)$ is Proposition~\ref{chordarc}.
Assertions~$(iv)$ and~$(v)$ are given by Theorem~\ref{main}, in particular,
the finite number of curves comes from Theorem~\ref{finite}, the characterization of branching by $120$ degrees comes from the classification of blow-up limits (Proposition~\ref{corr}) together with the uniqueness of the type of the blow-up
(Proposition~\ref{typeS}).
The property~$(v)$ for a generic (that is, not necessarily convex)
$\Omega$, i.e.\
the local $C^{1,\alpha}$ regularity of the curves has two ingredients, namely,
the classification of blow-up limits
(Proposition~\ref{corr}) which together with Proposition~\ref{typeS}  says that any point which is not a triple point neither an endpoint is necessarily a ``flat'' point, i.e.\ blows-up as a line), and then Theorem~\ref{C1reg} which says that $\Sigma$ is   $C^{1,\alpha}$ around any ``flat'' points.
The fact that $\Sigma$ touches $\partial \Omega$ tangentially is a consequence of Theorem~\ref{global2} and Proposition~\ref{classification1}. The case of $\Omega$ convex in~$(v)$ is Remark~\ref{rem_regconvOm1}. Finally,~$(vi)$ is Proposition~\ref{prop_furtherReg1}.

\subsubsection*{Basic techniques and background idea}


  The  proof of the  local $C^{1,\alpha}$ regularity result is contained in Sections~\ref{sectionEpsReg} and~\ref{dual}. The main strategy  
is inspired by the regularity theory for the Mumford-Shah functional in dimension $2$, more precisely by the approach of Bonnet~\cite{b} and David~\cite{d}. 
The rough idea that first comes in mind is
to show that every minimizer $\Sigma$ of Problem~\ref{pb_compl_pen1} is an \emph{almost minimizer} for the length, i.e.\ to prove that for any competitor $\Sigma'$ satisfying $\Sigma\triangle \Sigma'\subset B_r$, where $B_r$ stands for a ball of radius $r>0$, one has
\[\H(\Sigma \cap B_r)\leq \H(\Sigma'\cap B_r) +  Cr^{1+\alpha},\]
and then
to apply the regularity theory for almost minimal sets.

In our situation, the error term $Cr^{1+\alpha}$ may only come from the energy (or, equivalently, compliance) part
of the functional, namely, we need to prove an estimate of type
\[|\mathcal{C}(\Sigma)-\mathcal{C}(\Sigma')|\leq Cr^{1+\alpha}, \quad \text{ if  } \quad  \Sigma\triangle \Sigma'\subset B_r.\]
To obtain the latter we have to overcome several difficulties  which are due to
essential differences between Problem~\ref{pb_compl_pen1} and  more classical free boundary or free discontinuity problems. The  first one is a
substantially nonlocal behavior of the compliance functional, in the sense that
small perturbations of $\Sigma$ affects the potential $u_{\Sigma}$ in the
whole $\Omega$. This can be overcome by a simple cut-off argument (Lemma~\ref{lm_compl_loc1}). It shows that if $\Sigma'$ is a competitor for $\Sigma$ in $B_r$, then the defect of minimality is controlled by  the inequality
\[|\mathcal{C}(\Sigma)-\mathcal{C}(\Sigma')|\leq C\int_{B_r}|\nabla u_{\Sigma'}|^2 \, dx.\]
Here, the right hand side depends on the competitor $\Sigma'$,
which leads us to
introduce the quantity
\[  \omega_\Sigma(x,r) := \max_{\Sigma' \text{ connected } ;  \Sigma' \Delta \Sigma \subset \overline{B}_r(x) } \left(\frac{1}{r} \int_{B_r(x)} |\nabla u_{\Sigma'} |^2 \, dx \right),
 \]
 and we arrive to the estimate
  \[|\mathcal{C}(\Sigma)-\mathcal{C}(\Sigma')|\leq Cr \omega_\Sigma(x,r).\]
To obtain $C^{1,\alpha}$ regularity, one has to prove that $\omega_\Sigma(x,r)\leq Cr^{\alpha}$ for some $\alpha>0$. This is done via a classical monotonicity formula, which is a version of the one of Bonnet~\cite{b}, but adapted for Dirichlet boundary conditions.
It implies a suitable decay for $\omega_\Sigma(x,r)$, provided that
$\Sigma$ is flat enough in the neighborhood of $x$.
On the other hand the flatness decays suitably fast provided that
$\omega_\Sigma$ is small enough. In other words, we cannot apply directly the theory of almost minimal sets but we have to reproduce some of its arguments, with the additional difficulty that we have to control both the quantity $\omega_\Sigma$ and flatness of $\Sigma$ at the same time. All this leads to an ``$\varepsilon$-regularity" result  (Corollary~\ref{TheCor}). 
At the end we can bootstrap all the estimates and so follows the $C^{1,\alpha}$ regularity.


To get the  full regularity result (i.e.\ to show that every minimizer $\Sigma$ consists of finite number of smooth injective curves),
we perform  a blow-up analysis. Again, the main difficulty comes from the nonlocal behavior of the compliance functional. To bypass it,
we consider the dual formulation of Problem~\ref{pb_compl_pen1} (Proposition~\ref{duality1}), proving that Problem~\ref{pb_compl_pen1} is equivalent   to the following one
\begin{align*}
\min_{(\sigma, \Sigma) \in \mathcal{B}} &\frac{1}{2}\int_{\Omega }|\sigma|^2 \, dx +  \H(\Sigma), 
\quad\text{where }\\
\mathcal{B} &:=\{(\sigma,\Sigma) \; : \; \Sigma\subset \overline{\Omega} \text{ compact connected},  \sigma \in L^2(\Omega,\R^2),\, {\rm div}\, \sigma = f \text{ in  } \mathcal{D}'(\Omega \setminus \Sigma) \}.
\end{align*}
The latter problem is now localizable. As a matter of a fact,
it turns out that this dual problem is very close to the Mumford-Shah problem, and more remarkably, the dual formulation of the minimizing problem satisfied by the blow-up limits is exactly the same as the  characterization of
global minimizers for  Mumford-Shah problem in~\cite{b}.
This gives a possibility to characterize completely the possible
blow-up limits and conclude the proof of full regularity.
This is done in Section~\ref{dual}. Note however, that the blow-up
at the boundary  of $\partial \Omega$ is much different than what usually happens in the Mumford-Shah problem. Here the blow-up limits at the boundary are tangent to the boundary, whereas for the Mumford-Shah functional
they are transversal.


Let us furthermore mention that we first prove several other qualitative properties on minimizers, like Ahlfors-regularity, absence of loops, chord-arc estimate (Part I of Theorem~\ref{TheMainTheorem}). It is worth mentioning that, curiously enough, all of theses properties are needed to characterize the blow-up limits.
More precisely, one of the key ingredients that leads to the classification of the global minimizers for the Mumford-Shah functional
is the fact that the blow-up limits are connected in the limit, as required by Bonnet~\cite{b} to get the classification of blow-up limits. Although here
we deal with a priori connected sets, in general, the blow-up of a connected set may not be connected in the limit, and this is why one needs to use the optimality. To this aim, we prove that every minimizer $\Sigma$ of  Problem~\ref{pb_compl_pen1} is a chord-arc set (Proposition~\ref{chordarc}).
It is not difficult to see then that the  blow-up limit of a chord-arc set is connected.
A similar chord-arc estimate has been  already proven for each connected component of a Mumford-Shah minimizer, but the proof  presented  here  for the compliance minimizers is much different:
here we need to preserve connectedness on competitors.
Our strategy is as follows. Let us look closer at the rough idea of
proving  that a minimizer does  not contain loops.
The argument is by cutting the loop at a
flat point by a piece of set of size $r$,
and estimate the loss in the compliance term in terms of $r^{1+\alpha}$, provided that we have found a suitable point where to cut.
The proof of~\eqref{chordarc0} is a sort of quantitative version of this argument: if $x$ and $y$ are very close to each others and the curve connecting them in $\Sigma$ is going far away, then one can cut this curve at some place where it is flat, and add the little segment connecting $x$ to $y$ to preserve connectedness. To perform it, one needs the radius of the ball where $\Sigma$ is flat to be controlled from below, uniformly with respect to the distance from $x$ to $y$. This is obtained from the uniform rectifiability of $\Sigma$, which follows from the Ahlfors-regularity of the latter (Theorem~\ref{th_compl_ahl1}).

\section{Preliminaries}

\subsection{Notation}

We introduce the following notation.

\begin{itemize}
\item For $ \{x,y\}\subset \R^2$, $|x|$ denotes the Euclidean norm and $d(x,y):=|x-y|$ the Euclidean distance, $x \cdot y$ the usual inner product.
\item (${\bf e}_1, {\bf e}_2)$ denotes the canonical basis of $\R^2$.
\item For $D\subset \R^2$, $\mathbf{1}_D$ is the characteristic function of $D$, $D^c:=\R^2\setminus D$, $\overline{D}$ and $\partial D$ are the closure and the topological boundary of $D$ respectively,
     $\HH^1(D)$ is the $1$-dimensional Hausdorff measure of $D$, $\diam\, D$ the diameter of $D$, and $\dist(x,D):=\inf\{d(x,y)\,:\, y\in D\}$ whenever $x\in \R^2$.
\item  We let
$B_{r}(x):=\{y\in \R^2, d(x,y)<r\}$ stand for the open ball of radius  $r\in[0,\infty)$ and center $x\in \R^2$; if $\alpha>0$ then we use also the notation
$\alpha B_r(x):= B_{\alpha r}(x)$.
\item $d_H(A,B)$ is the Hausdorff distance between the sets
$A$ and $B$ defined by
$$d_H(A,B)=\max(\sup_{x\in A} \dist(x,B), \sup_{x\in B}\dist(x,A)).$$
\end{itemize}

Let now $\Om\subset \R^2$.
\begin{itemize}
\item If $\Omega$ is measurable, then $L^p(\Omega)$ stands for the Lebesgue space of $p$-integrable real functions for $p\in [1, +\infty)$ and measurable essentially bounded functions for $p=+\infty$, $\|\cdot\|_p$ starnding for
its standard norm and $p'$ for the conjugate exponent $1/p+1/p'=1$;  $L^p(\Omega,\R^2)$ is the respective Lebesgue space of 
functions with values in $\R^2$; $L^p_{loc}(\Omega)$ is the space of functions $u$ such that $u\in L^p(K)$ for all compact $K\subset \Omega$. The convergence in $L^p_{loc}(\Omega)$ means the convergence for the norm $\|\cdot\|_p$ on every compact $K\subset \Omega$.
\end{itemize}

If $\Omega$ is open, then
\begin{itemize}
\item $\mathcal{K}(\Omega)$ denotes the set of all compact and connected sets $\Sigma \subset \overline{\Omega}$;
\item $C^\infty_0(\Omega)$ stands for the class of infinitely differentiable functions with
compact support in $\Omega$;
\item $C^k(\Omega)$ (resp.\ $C^{k,\alpha}(\Omega)$) stands for the class of $k$ times differentiable functions (resp.\  $k$ times differentiable with $\alpha$-H\"{o}lder $k$-th derivative) in $\Omega$, where $k\in \N$, $\alpha\in (0,1)$;
\item $\mathcal{D}'(\Omega)$ stands for the usual space of distributions
 in $\Omega$;
\item $\Lip(\overline{\Om})$  stands for the class of all Lipschitz maps
$f\colon \overline{\Om}\to \R$;
\item $H^1(\Omega)$ is the standard Sobolev space of functions $u\in L^2(\Omega)$ having
  (distributional) derivative in $L^2(\Omega)$; $H^1_{loc}(\Omega)$ is its local version akin to $L^p_{loc}(\Omega)$;
\item $H_0^1(\Omega)$ stands for the usual Sobolev space defined by the closure
of $C^\infty_0(\Omega)$ for the norm $\|u\|_{H_0^1(\Omega)}:= \int_\Omega |\nabla u|^2\,dx$;
if necessary, the functions in $H_0^1(\Omega)$ will be silently assumed to be extended to the whole
$\R^n$ by zero over $\Omega^c$. This gives a natural embedding of $H_0^1(\Omega)$ into $H^1(\R^n)$;
\item  Given $\Omega'\subset \Omega$ and $\Sigma \in \mathcal{K}(\Omega)$ we will also denote by
$$H^1_{0,\Sigma}(\Omega' \setminus \Sigma)=\bigcup_{g \in H^1_0(\Omega \setminus \Sigma)}
\left(g+H^1_0(\Om'\setminus \Sigma)\right).$$
\end{itemize}

An open, bounded, and connected $\Omega \subset \R^2$ will be called a
$C^1$ domain, if $\partial \Omega$ is locally up to rotation
the graph a $C^1$ function.

\subsection{Estimate of $u_{\Sigma}$}


Let $\Om\subset \R^2$ be an open set.
When $\Sigma\subset \overline{\Om}$ is closed, then
$\Om\setminus\Sigma$ is open, and
it is well-known that if $f\in L^2(\Om)$, then there is a unique function $u_{\Sigma}$ that minimizes $E$ over $H^1_{0}(\Om\setminus\Sigma)$.
The dependence of $u$ and hence of compliance on $\Sigma$ is related to the fact that a $1$-dimensional set in $\R^2$ has   a non-trivial capacity: in fact,
if $\Sigma$ has zero Hausdorff dimension, then $\C(\Sigma)=\C(\emptyset)$.


In this paper, we need the following estimate, which is a direct consequence of~\cite[Theorem~8.16]{gt}.

\begin{proposition}\label{boundus} If $f \in L^p(\Omega)$ for some $p>2$ and $\Sigma$ is a closed subset of $\overline{\Om}$, then $u_\Sigma$ is bounded with
\begin{equation}\label{eq_boundus1}
\|u_\Sigma\|_\infty \leq C\|f\|_p.
\end{equation}
for some $C=C(p, |\Omega|)>0$.
\end{proposition}

\begin{remark}\label{rem_boundus2} Under conditions of the above Proposition~\ref{boundus} we have further that
\begin{equation}\label{eq_boundus2}
\begin{aligned}
\|\nabla u_\Sigma\|_2^2 & = 2\mathcal{C}(u_\Sigma)\leq 2 \mathcal{C}(\emptyset) = 2\int_\Om |\nabla u_\emptyset|^2\,dx 
\\
& =  2\int_\Om u_\emptyset f\,dx \leq C(p, |\Omega|)\|f\|_p |\Om|^{1/p'}\quad\mbox{by H\"{o}lder inequality and~\eqref{eq_boundus1}.}
\end{aligned}
\end{equation}
\end{remark}

\subsection{Existence}

Here we state a particular case of \v{S}ver\'ak's theorem, which will be used several times in this paper.

\begin{theorem}[\v{S}ver\'ak~\cite{sverak}]\label{th_Sverak1}
Let $\Om$ an open bounded set in $\R^2$, and $f\in L^2(\Om)$. Let $(\Sigma_{n})_{n}$ be a sequence of connected set in $\Om$, converging to $\Sigma \subset \overline{\Omega}$ in the Hausdorff distance. Then
\[u_{\Sigma_{n}}\mathop{\longrightarrow}_{n\to\infty}u_{\Sigma}\;\;\;\textrm{ strongly in }H^1(\Om).\]
\end{theorem}

\begin{remark} \label{mosco}Notice that, a byproduct of Sver\'ak's theorem is the so called Mosco-convergence of $H^1_0(\Omega \setminus \Sigma_n)$ that will be used later in the paper. More precisely, for
every $u \in H^1(\Omega\setminus \Sigma)$ there exists a sequence $\{u_n\}$
such that $u_n\in H^1_0(\Omega \setminus \Sigma_n)$  for all $n\in \N$ and $u_n\to u$ strongly in $H^1(\Omega)$, and moreover  any sequence of $\{v_n\}$ satisfying $v_n\in H^1_0(\Omega \setminus \Sigma_n)$  for all $n\in \N$ and $v_n \rightharpoonup v$ weakly in $H^1(\Omega)$ satisfies   $v\in H^1_0(\Omega \setminus \Sigma)$ (see~\cite[Proposition 3.5.4]{hp}).
\end{remark}

The following simple assertion gives existence of solutions to the penalized optimal compliance problem.

\begin{proposition}\label{existence} Problem~\ref{pb_compl_pen1} admits a minimizer.
\end{proposition}

\begin{proof} Let $\Sigma_n$ be a minimizing sequence for Problem~\ref{pb_compl_pen1}. From Blaschke's selection principle~\cite[Theorem~6.1]{afp}
there exists a subsequence converging for the Hausdorff distance to some closed $\Sigma\subset \overline{\Om}$. By Theorem~\ref{th_Sverak1}
 $u_{\Sigma_n}$ converges strongly in $H^1(\Om)$ to $u_{\Sigma}$ and   Go\l ab's theorem~\cite[Theorem~3.18]{falconer} gives the lower-semicontinuity of $\H$, which implies that $\Sigma$ must be a minimizer.
\end{proof}

\begin{remark} In~\cite{BuSan07} the assumption $f\geq 0$ is added in the statement of existence but it is actually unnecessary.
\end{remark}

\begin{remark} If we drop the connectedness assumption on $\Sigma$, then the existence of minimizer for  $\mathcal{F}_\lambda$ fails. Indeed, it is not difficult to construct an example of (highly disconnected) set $\Sigma\subset \Omega$ with $\H(\Sigma)$ arbitrary small which spreads into $\Omega$ so that  $u_{\Sigma}$ has very small energy, leading to
\[\inf_{\Sigma \subset \Omega, closed} \mathcal{F}_\lambda(\Sigma)=0.\]
\end{remark}

\section{Estimates of the variations of the compliance}\label{sect:variationcompliance}

  Our goal in this section is to obtain an estimate of the variation of the compliance between $\Sigma$ and $\Sigma'$ when $\Sigma\Delta\Sigma'$ is localized, without assuming any regularity for $\Sigma$ and $\Sigma'$.

\subsection{Localization lemma}

We first prove the following localization lemma.

\begin{lemma}\label{lm_compl_loc1}
Let $f\in L^2(\Omega)$, $\Sigma$ and $\Sigma'$ be closed subsets of $\overline{\Omega}$, and let $x_{0} \in \R^2$. We consider $0<r_{0}<r_1$ and assume
$\Sigma'\Delta\Sigma\subset B_{r_0}(x_0)$.
Then  for
every $\varphi\in \Lip(\R^2)$ such that 
$\varphi=1$ over $B_{r_1}^c(x_0)$,
$\varphi=0$ over $B_{r_0}(x_0)$, and $\|\varphi\|_{\infty}\leq 1$ on $\R^2$,
one has
\begin{align*}
E(u_{\Sigma})- E(u_{\Sigma'})
\leq  & \int_{B_{r_1}(x_0)} u_{\Sigma'} f(1-\varphi)\, dx + \int_{B_{r_1}(x_0)} u_{\Sigma'}^2 |\nabla \varphi|^2\, dx  + \int_{ B_{r_1}(x_0)} u_{\Sigma'} \varphi\nabla u_{\Sigma'}\cdot\nabla\varphi\, dx
\end{align*}
\end{lemma}

\begin{proof}
Since $u_{\Sigma'}\varphi\in H_0^1(\Omega\setminus\Sigma)$ and $u_\Sigma$ is a minimizer of $E$ over $H^1_{0}(\Om\setminus\Sigma)$, we have
\[
E(u_{\Sigma})\leq E(u_{\Sigma'}\varphi),
\]
and hence,
\begin{eqnarray*}
E(u_{\Sigma}) -E(u_{\Sigma'})&\leq&
E(u_{\Sigma'}\varphi) -E(u_{\Sigma'})
\leq  \frac 1 2 \int_\Omega  |\nabla(u_{\Sigma'}\varphi)|^2\, dx - \int_\Omega u_{\Sigma'}\varphi f\, dx\\
&& \qquad \qquad- \frac 1 2 \int_\Omega  |\nabla u_{\Sigma'}|^2\, dx  + \int_\Omega u_{\Sigma'} f\, dx \\
&  = &\frac 1 2 \int_\Omega  |\nabla u_{\Sigma'}|^2\varphi^2\, dx  +
 \frac 1 2 \int_\Omega  |\nabla \varphi|^2u_{\Sigma'}^2\, dx +
 \int_\Omega u_{\Sigma'} \varphi\nabla u_{\Sigma'}\cdot\nabla\varphi\, dx
 - \int_\Omega u_{\Sigma'}\varphi f\, dx\\
&&\qquad\qquad- \frac 1 2 \int_\Omega  |\nabla u_{\Sigma'}|^2\, dx  + \int_\Omega u_{\Sigma'} f\, dx \\
&=&\frac 1 2 \int_\Omega  |\nabla u_{\Sigma'}|^2\underbrace{(\varphi^2-1)}_{\leq 0}\, dx  +
 \frac 1 2 \int_\Omega  |\nabla \varphi|^2u_{\Sigma'}^2\, dx  + \int_\Omega u_{\Sigma'} (1-\varphi)f\, dx\\
&&\qquad\qquad+
 \int_\Omega u_{\Sigma'} \varphi\nabla u_{\Sigma'}\cdot\nabla\varphi\, dx
\end{eqnarray*}
Therefore,
recalling that $\varphi=1$  over $B_{r_1}^c(x_0)$, we conclude the proof.
\end{proof}

\begin{remark}
Note that if  $\Sigma'\subset\Sigma$  then
\[
0\leq E(u_{\Sigma})-E(u_{\Sigma'}) ,
\]
since $H_0^1(\Omega\setminus\Sigma) \subset H_0^1(\Omega\setminus\Sigma')$.
\end{remark}

\subsection{Monotonicity formula and decay of energy}

The monotonicity of energy will be one of our main tool in all the sequel. We start with a  general statement which says that the maximum length for $\partial B_r \setminus \Sigma$ gives the power of decay for the function $u_\Sigma$.   In what follows we assume that $u_\Sigma$ is extended by $0$ outside $\Omega$.

\begin{lemma}\label{lm_compl_monot1}
Let  $\Sigma\subset\overline{\Om}$  be a closed 
 set, $f\in L^p(\Omega)$, where $p>2$,
and
$ x_0\in{\overline{\Omega}} $.
Let $0  \leq   r_0< r_1$, and $\gamma \in [\gamma_{\Sigma}(x_0, r_{0},r_{1}),2\pi]\setminus\{0\}$, where we denote
\begin{equation}\label{eq:defgamma}
\gamma_{\Sigma}(x_0,r_{0},r_{1}):=\sup \left\{\frac{\H(S)}{r}\colon r\in (r_0,r_1)\textrm{ and } S  \textrm{ connected component of } \partial B_r(x_{0})\setminus (\Sigma\cup\partial\Omega) \right\},
\end{equation}
We assume
\begin{equation}\label{eq:geomass}
(\Sigma\cup\partial\Om)\cap\partial B_{r}(x_{0})\neq\emptyset\text{ for all $r\in[r_{0},r_{1}]$},
\end{equation}
and finally we suppose that $1/p'>\pi/\gamma$. Then the function

\begin{align*}
 r\in [r_0,r_1]\mapsto \frac{1}{r^{\alpha}}\int_{B_{r}(x_0)} |\nabla u_{\Sigma}|^2\, dx + Cr^{\frac{2}{p'}-\alpha},
\end{align*}
is nondecreasing, where $\alpha:=2\pi/\gamma$ and
$C=C(|\Omega|, p, \|f\|_p, \gamma)>0$.
\end{lemma}

\begin{remark}\label{remYeah}
In particular, setting $\gamma:=2\pi$ in the above Lemma, we get that
the function
\begin{align*}
 r\in [r_{0},r_1]\mapsto \frac{1}{r}\int_{B_{r}(x_0)} |\nabla u_{\Sigma}|^2\, dx + Cr^{\frac{2}{p'}-1},
\end{align*}
is nondecreasing, where   $C=C(|\Omega|, p, \|f\|_p)>0$, once $x_0$ satisfies~\eqref{eq:geomass}.
\end{remark}


\begin{proof}
We denote $u=u_{\Sigma}$ and extend $u$ by zero outside $\Omega$, {so that $u\in H^1(\R^2)$}. One has,   for a.e. $r \in (r_{0},r_{1}) $,
\begin{equation}\label{eq_compl_monot1}
\begin{aligned}
G(r) &:=\int_{B_{r}(x_0)} |\nabla u|^2\, dx \\
& = \int_{\partial B_{r}(x_0)} u\frac{\partial u}{\partial\nu}\,d\H
+  \int_{B_{r}(x_0)} f u\, dx \quad \mbox{ by Lemma~\ref{lm_compl_intparts1}}\\
&
\leq \int_{\partial B_{r}(x_0)} \frac \delta {2} u^2 \, d\H
 +\int_{\partial B_{r}(x_0)} \frac 1 {2\delta} \left(\frac{\partial u}{\partial\nu}\right)^2 \,
d\H
+  \int_{B_{r}(x_0)} f u\, dx
\end{aligned}
\end{equation}
for every $\delta>0$.
By Poincar\'{e} inequality (Lemma~\ref{lm_compl_poincare1})  that can be applied
because of~\eqref{eq:geomass},
 we get
\[
\int_{\partial B_{r}(x_0)} u^2 \,d\H=\int_{\partial B_{r}(x_0)\setminus (\Sigma\cup\partial\Omega)} u^2 \,d\H
\leq    \left(\frac{\gamma r}{\pi}\right)^2  \int_{\partial B_{r}(x_0)\setminus (\Sigma\cup\partial\Omega)} \left|\nabla_\tau u\right|^2\, d\H,
\]

 Notice indeed that the trace of $u$ on $\partial B_r$ belongs to $H^1_0(\partial B_r \setminus \Sigma \cup \partial \Omega) $ for a.e. $r$.  Thus choosing $ \delta:= \pi / (\gamma r)$, we get
\begin{equation}\label{eq_compl_monot2}
\begin{aligned}
G(r)& \leq
\frac{1}{2}\frac{\gamma r}{\pi}  \int_{\partial B_{r}(x_0)}
\left(\left|\nabla_\tau u\right|^2 + \left(\frac{\partial u}{\partial\nu}\right)^2\right) \,d\H
+  \int_{B_{r}(x_0)} f u\, dx\\
&=
\frac{\gamma}{2\pi} r \int_{\partial B_{r}(x_0)}
|\nabla u|^2 \,d\H
+  \int_{B_{r}(x_0)} f u\, dx\\
&= \frac{\gamma}{2\pi} r G'(r)
+  \int_{B_{r}(x_0)} f u\, dx.
\end{aligned}
\end{equation}
 The function $r\mapsto G(r)$ is indeed absolutely continuous, and its derivative is a.e. equal to $r\mapsto \int_{\partial B_{r}(x_0)}
|\nabla u|^2 \,d\H$. 

Recalling that
 $u\in L^\infty(\Omega)$  by Proposition~\ref{boundus}
and using the H\"{o}lder inequality together with~\eqref{eq_boundus1}
to estimate
\[
\left|\int_{B_{r}(x_0)} f u\, dx\right| \leq C  r^{2/p'}  
\]
with some constant $C>0$ depending only on  $\|f\|_p$, $p$, $|\Omega|$,
we get
\[
G(r)\leq \frac{\gamma}{2\pi} r G'(r)
+    C r^{2/p'}.
\]
This gives that
\[r\mapsto \frac{G(r)}{r^\alpha} + C_0 r^{\frac{2}{p'}-\alpha}\]
is non decreasing, with
\[C_0:= \frac{C\alpha}{\frac{2}{p'}-\alpha}  \quad  \text{ and }\quad \alpha=\frac{2\pi}{\gamma}\]
 (notice that $\frac{2}{p'}-\alpha>0$ under our assumptions).
Indeed,
\begin{eqnarray}
\frac{d}{dr}\left(\frac{G(r)}{r^\alpha} + C_0 r^{\frac{2}{p'}-\alpha}\right)&=& \frac{G'(r)r^\alpha - G(r)\alpha r^{\alpha-1}}{r^{2\alpha}} + C\alpha r^{\frac{2}{p'}-\alpha -1} \notag \\
&=&  \frac{G'(r)r^\alpha - G(r)\alpha r^{\alpha-1}+ C\alpha r^{\frac{2}{p'}+\alpha -1}}{r^{2\alpha}} \notag \\
&=&  \alpha \frac{\frac{1}{\alpha}G'(r)r - G(r) + Cr^{\frac{2}{p'}}}{r^{\alpha+1}} \geq 0. \notag
\end{eqnarray}
\end{proof}

In particular, for $C^1$ domains we obtain the following useful decay at the boundary.

\begin{lemma} \label{decayboundary} Assume that  $\Omega$ is a $C^1$ domain,  let $\Sigma \subset \overline{\Omega}$ be a closed set and $p>2$. Then there exists an $r_0=r_0(\partial \Omega,p)>0$ and a $\nu=\nu(p)>0$   such that
\begin{align*}
 \int_{B_{r}(x)} |\nabla u_{\Sigma}|^2\, dx  \leq  C_1\left(\frac{r}{r_0}\right)^{1+\nu} + C_2r^{\frac{2}{p'}} \quad \text{for all } x \in \partial \Omega \text{ and for all }r\in(0,r_0),
\end{align*}
with positive constants $C_1$ and $C_2$ depending only on $|\Omega|$, $p$, $\|f\|_p$.
\end{lemma}

\begin{proof} Since $p>2$, one has  $1/p'>1/2$, so that one can find $\gamma_p<2\pi$ such that
$$\frac{1}{p'}>\frac{\pi}{\gamma_p}>\frac{1}{2}.$$
Now since $\partial \Omega$ is locally $C^1$ and compact,
there exists a threshold $r_0$ depending on $\partial \Omega$ (and also on $p$) such that $\partial \Omega$ is flat enough in all the balls
$B_r(x)$ for $r\leq r_0$, uniformly in $x \in \partial \Omega$, in other words, such that
$\gamma_{\Sigma}(x,0,r_0)\leq \gamma_p$ for all $x \in \partial \Omega$,  so that Lemma~\ref{lm_compl_monot1} applies and for all $x \in \partial \Omega$
and $r\in(0,r_0)$ giving
\begin{eqnarray*}
 \frac{1}{r^\alpha}\int_{B_{r}(x)} |\nabla u_{\Sigma}|^2\, dx + Cr^{\frac{2}{p'}-\alpha}&\leq&  \frac{1}{r_0^\alpha}\int_{B_{r_0}(x)} |\nabla u_{\Sigma}|^2\, dx + Cr_0^{\frac{2}{p'}-\alpha}\\
 &\leq & \frac{1}{r_0^\alpha} \|\nabla u_\Sigma\|_2^2  + Cr_0^{\frac{2}{p'}-\alpha},
 \end{eqnarray*}
where  $C=C(|\Omega|, p, \|f\|_p)$,  
and $\alpha=2\pi/\gamma_p$. We conclude multiplying by $r^\alpha$, defining $\nu:=\alpha-1=(2\pi/\gamma_p)-1$ and estimating $\|\nabla u_\Sigma\|_2$ by Remark~\ref{rem_boundus2}.
 \end{proof}

\subsection{Estimate of $|\C(\Sigma)-\C(\Sigma')|$}\label{ssect:estimate}

We collect now all the necessary estimates on minimizers of $E$ over $H_0^1(\Omega\setminus\Sigma)$ in the following statement.

\begin{lemma}\label{lm_compl_estabove1a}
Let  $\Sigma\subset\overline{\Om}$  be a closed  connected
 set, $f\in L^p(\Omega)$, where $p>2$,
and
$x_0\in \R^2$.
Let
$r_0,r_1$ satisfy
\begin{equation}\label{eq:geomass2}
\Sigma\cap B_{r_{0}}(x_{0})\neq\emptyset, \;\;\;\;\;\;\Sigma\setminus B_{r_{1}}(x_{0})\neq\emptyset.
\end{equation}
Then  for any  $r\in [r_0,r_1/2]$, for any  $\varphi\in \Lip(\R^2)$ such that $\|\varphi\|_\infty\leq 1$ and $\varphi=1$ over $B_{2r}^c(x_0)$,
$\varphi=0$ over $B_{r}(x_0)$ and
$\|\nabla\varphi\|_{\infty}\leq 1/r$,
the following assertions hold.
\begin{itemize}
  \item[(i)] There exists $C>0$ depending only on $|\Omega|$, $p$, $\|f\|_p$ such that:
 \begin{equation}\label{eq_compl_loc2a1}
    \left|\int_{B_{2r}(x_0)} {u_{\Sigma}} f(1-\varphi)\, dx\right|\leq  C r^{\frac{2}{p'}}.
\end{equation}
  \item[(ii)]
There exists $C>0$ universal such that:
  \begin{equation}\label{eq_compl_loc2a2}
\begin{aligned}
\left|\int_{B_{2r}(x_0)} {u_{\Sigma}}^2 |\nabla \varphi|^2\, dx \right| & \leq   C  \int_{B_{2r}(x_0)\setminus B_r{(x_{0})}} |\nabla {u_{\Sigma}} |^2 \, dx.
\end{aligned}
\end{equation}
  \item[(iii)] There exists $C>0$ universal such that:
  \begin{equation}\label{eq_compl_loc2c2}
\begin{aligned}
\left|\int_{B_{2r}(x_0)} {u_{\Sigma}} \varphi\nabla u_{\Sigma}\cdot\nabla\varphi\, dx\right| &\leq
C \int_{B_{2r}(x_0)} |\nabla u_{\Sigma}|^2\, dx.
\end{aligned}
\end{equation}
\end{itemize}
\end{lemma}

\begin{proof} The estimate~\eqref{eq_compl_loc2a1} is directly coming from H\"older inequality together with Proposition~\ref{boundus}.

Next, { we can use Poincar\'e inequality (Lemma~\ref{lm_compl_poincare2}), since $\Sigma\cap\overline{B}_{r}(x_{0})\neq\emptyset$ and $\Sigma\setminus{B}_{2r}(x_{0})\neq\emptyset.$
We obtain}
\begin{eqnarray}
\int_{B_{2r}(x_0)\setminus B_{r}{(x_{0})}} u_{\Sigma}^2 \, dx \leq C_P r^2 \int_{B_{2r}(x_0)\setminus B_r{(x_{0})}} |\nabla u_{\Sigma} |^2 \, dx  .\label{decay3}
\end{eqnarray}
{where $C_P$ is a universal constant, and in particular does not depend on the geometry of $\Sigma$}.  Using that $|\nabla \varphi| \leq \frac{1}{r} . {\bf 1}_{B_{2r}\setminus B_r} $, we directly get~\eqref{eq_compl_loc2a2}.

It remains to prove~($iii$). To this aim, we estimate
\begin{equation}\label{eq_compl_loc2c}
\begin{aligned}
\left|\int_{B_{2r}(x_0) } u_{\Sigma} \varphi\nabla u_{\Sigma}\cdot\nabla\varphi\, dx\right| &\leq
\frac{1}{r} \int_{B_{2r}(x_0)\setminus B_r(x_0)} |u_{\Sigma}|\cdot |\nabla u_{\Sigma}|\, dx \\
& \leq \frac{1}{r}\left(\int_{B_{2r}(x_0)\setminus B_{r}(x_0)} u_{\Sigma}^2\, dx\right)^{\frac{1}{2}} \left(\int_{B_{2r}(x_0)} |\nabla u_{\Sigma}|^2\, dx\right)^{\frac{1}{2}},
\end{aligned}
\end{equation}
the latter by H\"{o}lder inequality. As in~\eqref{decay3}, from Poincar\'{e} inequality in the
annulus $B_{2r}(x_0)\setminus B_r(x_0)$ (Lemma~\ref{lm_compl_poincare2})
we get
\begin{align*}
    \int_{B_{2r}(x_0)\setminus B_r(x_0)} u_{\Sigma}^2 \, dx  & \leq
    C_P r^2 \int_{B_{2r}(x_0)} |\nabla u_{\Sigma}|^2\,dx,
   \end{align*}
so that~\eqref{eq_compl_loc2c} becomes
\begin{equation}\label{eq_compl_loc2c1}
\begin{aligned}
\left|\int_{B_{2r}(x_0)} u_{\Sigma} \varphi\nabla u_{\Sigma}\cdot\nabla\varphi\, dx\right| &\leq
\sqrt{C_P} \left(\int_{B_{2r}(x_0)} |\nabla u_{\Sigma}|^2\, dx\right),
\end{aligned}
\end{equation}
as desired.
\end{proof}

\begin{proposition}\label{prop_compl_estabove1}
Let  $\Sigma\subset\overline{\Om}$  be a closed   connected
 set, $f\in L^p(\Omega)$, where $p>2$,
and
$x_0\in{\Sigma}$.
Let  $r_0,r_1$ be satisfying
 \[  0<2r_0<r_1<  \min(1, \diam(\Sigma)/2),\]
 and $\gamma\in[\gamma_{\Sigma}(x_0,r_{0},r_{1}),2\pi]\setminus\{0\}$ (see~\eqref{eq:defgamma} for a definition of $\gamma_{\Sigma}$).
Then for any $r\in [r_0,r_1/2]$, and for any closed connected set $\Sigma' \subset \overline{\Omega}$ satisfying $\Sigma \Delta \Sigma'\subset B_r(x_0)$ we have
\begin{equation}\label{eq_compl_estabove1mainN}
{|\C(\Sigma')-\C(\Sigma)|}=|E(u_{\Sigma'})- E(u_{\Sigma})|\leq C  \left(\frac{r}{ r_1}\right)^{\frac{2\pi}{\gamma}}  + C r^{\frac{2}{p'}}
\end{equation}
where $C>0$ depends on $|\Omega|$, $\|f\|_p$, $\gamma$, $p$.
\end{proposition}

\begin{proof}
There exists $\varphi\in \Lip(\R^2)$
as in the statement of Lemma~\ref{lm_compl_estabove1a}, i.e.\
such that $\|\varphi\|_\infty\leq 1$, $\varphi=1$ over $B_{2r}^c(x_0)$,
$\varphi=0$ over $B_{r}(x_0)$ and
$\|\nabla\varphi\|_{\infty}\leq 1/r$.
From Lemma~\ref{lm_compl_loc1} (applied with $r$ and $2r$ instead of $r_0$ and $r_1$ respectively) 
we have
\begin{equation}\label{eq_EminE'}
\begin{aligned}
E(u_{\Sigma'})- E(u_{\Sigma})
\leq  & \int_{B_{2r}(x_0)} u_{\Sigma'} f(1-\varphi)\, dx + \int_{B_{2r}(x_0)\setminus (\Sigma'\cup\partial\Omega)} u_{\Sigma'}^2 |\nabla \varphi|^2\, dx \\
& + \int_{B_{2r}(x_0)\setminus (\Sigma'\cup\partial\Omega)} u_{\Sigma'} \varphi\nabla u_{\Sigma'}\cdot\nabla\varphi\, dx,
\end{aligned}
\end{equation}
and, interchanging $\Sigma$ with $\Sigma'$,
\begin{equation}\label{eq_E'minE}
\begin{aligned}
E(u_{\Sigma})- E(u_{\Sigma'})
\leq  & \int_{B_{2r}(x_0)} u_{\Sigma} f(1-\varphi)\, dx + \int_{B_{2r}(x_0)\setminus (\Sigma\cup\partial\Omega)} u_{\Sigma}^2 |\nabla \varphi|^2\, dx \\
& + \int_{B_{2r}(x_0)\setminus (\Sigma\cup\partial\Omega)} u_{\Sigma} \varphi\nabla u_{\Sigma}\cdot\nabla\varphi\, dx.
\end{aligned}
\end{equation}

Observe that we can apply Lemma~\ref{lm_compl_monot1} and Lemma~\ref{lm_compl_estabove1a} for both $\Sigma$ and $\Sigma'$. For $\Sigma$, we use the connectedness of $\Sigma$ and $r_{1}<\diam(\Sigma)/2$ to obtain~\eqref{eq:geomass} and~\eqref{eq:geomass2}.
For $\Sigma'$ we apply Lemma~\ref{lm_compl_monot1} and Lemma~\ref{lm_compl_estabove1a}
with $r$ instead of $r_0$: indeed, $\Sigma \Delta \Sigma'\subset B_r(x_0)$ implies that~\eqref{eq:geomass} and~\eqref{eq:geomass2}
are also valid for $\Sigma'$ for $[r,r_{1}]$ instead of $[r_{0},r_{1}]$,
and we notice that
\begin{equation}
\gamma_{\Sigma'}(x_0,r,r_{1})=\gamma_{\Sigma}(x_0,r,r_{1})\leq \gamma_{\Sigma}(x_0,r_{0},r_{1})\leq\gamma.
\end{equation}
Therefore,
\begin{equation*}\label{decay1}
\begin{aligned}
\frac{1}{r^{\alpha}}\int_{B_{r}(x_0)} |\nabla u_\Sigma|^2\, dx  & \leq  \frac{1}{{r_1}^{\alpha}}\int_{B_{{r_1}}(x_0)} |\nabla u_\Sigma|^2\, dx + C{r_1}^{\frac{2}{p'}-\alpha} \\
& \leq  \frac{1}{{r_1}^{\alpha}}\|\nabla u_\Sigma\|_2^2 + C{r_1}^{\frac{2}{p'}-\alpha}
\end{aligned}
\end{equation*}
with $\alpha=2\pi/\gamma$ and $C=C(|\Omega|, p, \|f\|_p, \gamma)$, and
estimating $\|\nabla u_\Sigma\|_2$ by Remark~\ref{rem_boundus2}, we get
\begin{eqnarray}
\int_{B_{r}(x_0)} |\nabla u_\Sigma|^2\, dx  \leq  \left(\frac{r}{{r_1}}\right)^\alpha (C_1 + C{r_1}^{\frac{2}{p'}}) \leq C_2 \left(\frac{r}{{r_1}}\right)^\alpha\label{decay2}
\end{eqnarray}
where $C_1=C_1(|\Omega|,\|f\|_p,p,\gamma)$, $C_2:=C_1+C$ (we used $r_1\leq 1$ in the last inequality).
Since the same holds true also for $u_{\Sigma'}$,
we get
\begin{eqnarray}
\int_{B_{r}(x_0)} |\nabla u_{\Sigma'}|^2\, dx  \leq C_2 \left(\frac{r}{{r_1}}\right)^\alpha.\label{decay2a}
\end{eqnarray}
Applying to~\eqref{eq_EminE'} and~\eqref{eq_E'minE} the estimates of Lemma~\ref{lm_compl_estabove1a} for $u_\Sigma$ and $u_{\Sigma'}$
respectively (combined with~\eqref{decay2} and~\eqref{decay2a} respectively),
we arrive at~\eqref{eq_compl_estabove1mainN},
concluding the proof.
\end{proof}

\section{First qualitative properties}

In this section, we prove geometrical properties of the minimizers for Problem~\ref{pb_compl_pen1}. 

\subsection{Absence of loops}\label{ssect:noloop}

The following result holds true.

\begin{theorem}\label{th_compl_noloop1}
Let $f\in L^p(\Omega)$, $p>2$.
Then every solution $\Sigma$ of the penalized Problem~\ref{pb_compl_pen1} contains no closed curves
(homeomorphic images of $S^1$), hence $\R^2\setminus\Sigma$ is connected.
\end{theorem}


{  The idea of the proof of this result is to see that if $\Sigma$ has a loop, then we can cut a small piece of this loop and remain connected; this cut decreases the length and increases the energy. If we choose this cut properly, namely where the curve is ``flat'', then the energy estimate from Section~\ref{ssect:estimate} shows that these two variations are not of the same order, which leads to a contradiction, see the proof below. Before this proof, we give a geometric lemma asserting that such a cut can be done:

\begin{lemma}\label{lem:cut+flat}
Let $\Sigma$ be a closed connected set in $\R^2$, containing a simple closed curve $\Gamma$ (a homeomorphic image of $S^1$) and such that
$\H(\Sigma)<\infty$. Then $\H$-a.e.\
point $x\in \Gamma$ is such that
\begin{itemize}
\item ``noncut'': there is a sequence of (relatively) open sets $D_n\subset\Sigma$ satisfying
       \begin{itemize}
        \item[(i)]   $x\in D_n$ for all sufficiently large $n$;
        \item[(ii)] $\Sigma\setminus D_n$ are connected for all
                   $n$;
        \item[(iii)]  $\diam D_n\searrow 0$ as $n\to \infty$;
        \item[(iv)]  $D_n$ are connected for all
                   $n$.
       \end{itemize}
\item ``flatness'':
 there exists a ``tangent'' line $P$ to $\Sigma$ at $x$ in the sense that
$x\in  P$ and
\[
        \lim_{r\to 0^+}\beta_{\Sigma,P}(x,r)=0\quad\mbox{where }
           \beta_{\Sigma,P}(x,r)
         :=
        \sup_{y\in\Sigma\cap B_r(x)} \frac{\dist(y,P)}{r}.
\]
\end{itemize}
\end{lemma}

\begin{proof}
First, we apply~\cite[Lemma~5.6]{PaoSte13-steiner}, stating that
under the hypotheses on $\Sigma$,
$\H$-a.e.\ point $x\in \Gamma$ is a noncut point for $\Sigma$ (i.e. a point such that $\Sigma\setminus\{x\}$ is connected).
Then~\cite[Lemma~6.1]{ButSte03} affirms that
for every noncut point there are connected neighborhoods that
can be cut leaving the set connected, i.e.~ $(i)$-$(iv)$ are satisfied for a suitable sequence $D_{n}$.
The second requirement relies on the standard fact that closed connected sets with finite length are rectifiable (see, e.g.,~\cite[Proposition~3.4]{MirPaoSte04} or~\cite[Proposition~2.2]{Fil}).
\end{proof}

\begin{proof}[Proof of Theorem~\ref{th_compl_noloop1}]
Assume by contradiction that for some $\lambda>0$ a minimizer $\Sigma$ of $\mathcal{F}_\lambda$ over closed connected subsets of $\overline{\Om}$ contains a closed curve $\Gamma\subset \Sigma$. From Lemma~\ref{lem:cut+flat}, we have that there exists a point $x\in \Gamma$ which is a noncut point and such that $\Sigma$ is differentiable at $x$. Therefore there exist the sets
$D_n\subset\Sigma$   and a straight line $P$ as in Lemma~\ref{lem:cut+flat}. We denote $r_n:=\diam D_n$ so that $D_n\subset \Sigma\cap B_{r_n}(x)$.
The flatness  of $\Sigma$ at  $x$ implies that 
for some given $\varepsilon>0$, there exists an $r_{1}>0$
such that
\[
\arcsin \beta_{\Sigma,P}(x,r) \leq \varepsilon\pi/2
\quad\text{for all }r\in(0,r_{1}].
\]
Thus  every
connected $S\subset \partial B_r(x)\setminus\Sigma$ satisfies
\[
\H(S)\leq (\pi + 2\arcsin \beta_{\Sigma,P}(x,r))r=(1+\varepsilon) \pi r, \qquad\qquad r\in (0,r_{1}],
\]
and hence using Proposition~\ref{prop_compl_estabove1} with $\gamma:=(1+\varepsilon)\pi $ and $\Sigma'=\Sigma\setminus D_n$ when $r_n\leq R$,
we get
\begin{align*}
|E(u_{\Sigma\setminus D_n})- E(u_{\Sigma})| &\leq Cr_n^{\frac{2\pi}{\gamma}}+Cr_n^{\frac{2}{p'}}.
\end{align*}
But since
\[
\H(\Sigma\setminus D_n)= \H(\Sigma)-\H(D_n)\leq \H(\Sigma)-r_n,
\]
we get
\[
\mathcal{F}_\lambda(\Sigma\setminus D_n)\leq \mathcal{F}_\lambda(\Sigma) - \lambda r_n + Cr_n^{\frac{2\pi}{\gamma}}+Cr_n^{\frac{2}{p'}}.
\]
Recalling that $2/p'>1$ and
choosing $\varepsilon <1/2$   (of course, this affects the value of $C$ and
of $r_{1}$), then one has
\[
\mathcal{F}_\lambda(\Sigma\setminus D_n)\leq \mathcal{F}_\lambda(\Sigma) - \lambda r_n + o(r_n)
\]
as $r_n\to 0^+$, which for sufficiently small $r_n$
contradicts the minimality of $\Sigma$ and hence
proving that $\Sigma$ does not contain closed curves.
About the last assertion in Theorem~\ref{th_compl_noloop1}, we use theorem~II.5 of~\cite[\S~61]{Kur},
stating that if $D\subset \R^2$ is a bounded connected set with
locally connected boundary, then
there is a simple closed curve
$S\subset \partial D$. Then, if $\R^2\setminus\Sigma$ were
disconnected, there would exist a bounded
connected component $D$ of $\R^2\setminus\Sigma$
such that $\partial D\subset \Sigma$, and hence $\Sigma$
would contain a simple closed curve, contrary to what we proved just before.
\end{proof}

\subsection{Ahlfors regularity}

It is not difficult to show that the minimizers of Problem~\ref{pb_compl_pen1}  are
Ahlfors regular under quite nonrestrictive conditions on the data. Recall that a set $\Sigma\subset \R^2$
is called Ahlfors regular, if
there exist  some constants $c>0$, $r_0>0$ and $C>0$
such that for every
$r\in (0,r_0)$ and for every $x\in \Sigma$ one has
\begin{equation}\label{eq_M2Ahldef}
cr\leq \H(\Sigma\cap B_r(x))\leq Cr,
\end{equation}
a singleton being considered Ahlfors regular by definition.
The constants $C$ and $c$ will be further referred to as \emph{upper and lower Ahlfors regularity constants} of $\Sigma$ respectively.
It is known
that Ahlfors regularity of a closed connected set $\Sigma$ implies
{\em uniform rectifiability} (a kind of ``quantitative rectifiability'' which is somewhat stronger than
the classical rectifiability used in
geometric measure theory)
of $\Sigma$, which
provides several nice analytical properties of $\Sigma$, see for example~\cite{DavSemm93}, and will be used later several times.


If $\Sigma$ is closed and connected, then
the lower estimate in~\eqref{eq_M2Ahldef} is trivial: in fact, for
all $r <\diam \Sigma/2$ one has
$\Sigma\cap \partial B_r(x)\neq \emptyset$, and hence
\[
\H(\Sigma\cap B_r(x))\geq r,
\]
when $x\in \Sigma$. Hence, for such $\Sigma$ the proof
of Ahlfors regularity
reduces to verifying that
for every $x\in \Sigma$
and for all $r\in (0,r_0)$  with some $r_0>0$ independent of $x$
one has
\begin{equation}\label{eq_M2Ahldef2}
\frac{\H(\Sigma\cap B_r(x))}{r}\leq C.
\end{equation}

\begin{theorem}\label{th_compl_ahl1}
Let $f\in L^p(\Omega)$, $p>2$. 
Then every solution $\Sigma$ of Problem~\ref{pb_compl_pen1} is Ahlfors regular.
\end{theorem}

\begin{remark}
Looking closer at the proof of the above Theorem~\ref{th_compl_ahl1}, we observe that the upper Ahlfors regularity constant  may possibly depend only on $|\Omega|$, $\|f\|_p$, $p$. 
\end{remark}

\begin{proof}
We show that for every $\lambda>0$, every minimizer $\Sigma$ of $\mathcal{F}_\lambda$
among closed connected subsets of $\overline{\Om}$
satisfies~\eqref{eq_M2Ahldef2} with some $C> 2\pi$ and $r_0\in (0,\diam \Sigma/4)$.

Let $x\in \Sigma$ be arbitrary and set $R:=\diam \Sigma/4$, so that $\partial B_{2r}(x)\cap\Sigma\neq\emptyset$ for all $r\in (0, R]$. For every $r\in (0,R]$ we set $\Sigma_r:=(\Sigma\setminus B_r(x))\cup (\partial B_r(x)\cap\overline{\Omega})$.
Clearly, $\Sigma_r\subset\overline{\Om}$ is still closed, connected and
\begin{equation}\label{eq_compl_ahl1}
\H(\Sigma_r)\leq \H(\Sigma)-\H(\Sigma  \cap B_r(x)) + 2\pi r.
\end{equation}
Using Proposition~\ref{prop_compl_estabove1} with $\gamma=2\pi$ and $\Sigma_r$  instead of $\Sigma'$,
we get
\begin{align*}
|E(u_{\Sigma_r})- E(u_{\Sigma})| &\leq Cr +Cr^{2/p'}\leq 2Cr.
\end{align*}
for all $r\in (0, r_0]$ for   $r_0= \min\{R,1\}$ and $C>0$ independent on $x\in \Sigma$ (recall that $2/p'\geq 1$).
But in view of~\eqref{eq_compl_ahl1}
we get
\[
\mathcal{F}_\lambda(\Sigma_r)\leq \mathcal{F}_\lambda(\Sigma) - \lambda \H(\Sigma \cap B_r(x)) + 2\pi \lambda r + 2Cr.
\]
Therefore the optimality of $\Sigma$ implies
\begin{equation}\label{eq_compl_ahl2}
-\lambda \H(\Sigma  \cap  B_r(x)) + 2\pi \lambda r + 2Cr \geq 0,
\end{equation}
for all $r\in(0,r_{0}]$,
Clearly, then~\eqref{eq_compl_ahl2} gives~\eqref{eq_M2Ahldef2} and hence Ahlfors regularity of $\Sigma$.
\end{proof}

\section{Flatness and small energy implies $C^{1}$ regularity}

 \label{sectionEpsReg}

In this section we find sufficient conditions in a ball $B_{r}(x)$ which imply that $\Sigma\cap B_{r/2}(x)$ is a nice $C^1$ curve. The strategy vaguely follows the approach of Guy David in~\cite{d} where a similar work is done for Mumford-Shah minimizers, but adapted with the specificities of our problem which is a  Dirichlet problem of min-max type instead of a Neumann problem of min-min type in~\cite{d}.
The connectedness constraint makes also a big difference with the Mumford-Shah problem. The general tools are  a decay of energy provided that $\Sigma$ stays ``flat'' thanks to the monotonicity formula, and on the other hand, a control on the flatness when the energy is small. We then conclude by bootstraping both estimates.

In this section and all the next ones, we will always assume $\lambda=1$ for simplicity. Of course this is not restrictive regarding to the regularity theory.



\subsection{Control of the energy when $\Sigma$ is flat}

For any $x \in   \Omega $ and $r>0$ such that $B_{r}(x) \subset \Omega$ we denote by $\beta_\Sigma(x,r)$ the flatness of $\Sigma$ in $B_{r}(x)$ defined through
\[\beta_\Sigma(x,r):=\inf_{ P \ni x} \frac{1}{r}d_H(\Sigma\cap B_r(x), P\cap B_r(x)),\]
where $d_H$ is the Hausdorff distance and where the infimum is taken over all affine lines $P$  passing through $x$. 
Notice that the $\inf$ above is attained, i.e.\ is actually a minimum.

We shall also need a variant where the proximity to affine lines is checked only in an annulus and not in the whole ball. Precisely, for $0<s<r$ we denote
\[\beta_\Sigma(x,r,s):=\inf_{ P\ni x} \frac{1}{r}d_H(\Sigma\cap B_r(x)\setminus B_s(x), P\cap B_r(x)\setminus B_s(x)),\]
where again the $\inf$ is taken over all affine lines $P$ passing through $x$.

Observe that for all $0<s<r$, it directly comes from the definition that
\[\beta_\Sigma(x,r,s)\leq \beta_\Sigma(x,r),\]
and the main point about $\beta_\Sigma(x,r,s)$ is the following obvious fact
\[
\beta_\Sigma(x,r)\leq \varepsilon  \text{ and } \Sigma \Delta \Sigma'  \subset B_s(x) \text{ imply }
\beta_{\Sigma'}(x,r,s)\leq \varepsilon.
\]
Observe also that for all $a\in (0,1)$ we have
\begin{equation}\label{eq_beta_ar1}
\beta_\Sigma(x,ar)\leq \frac{1}{a}\beta_\Sigma(x,r).
\end{equation}

Our first aim is to use the monotonicity formula to control the energy of $u_\Sigma$ around points where $\beta_{\Sigma}$ is small: this is done in Proposition~\ref{prop_compl_monot4} below.  For this purpose we seek for a lemma similar to Lemma~\ref{lm_compl_monot1}, but where the assumption~\eqref{eq:geomass} is replaced by an assumption on $\beta_\Sigma(x,r)$ only. This leads to new difficulties, because even if $\Sigma$ is connected and flat around a point $x\in\Sigma$, it may happen that $\Sigma\cap B_{r}(x)$ is not connected (see Figure \ref{fig:flat1}).
We state the lemma when { $\gamma=\pi+\pi/3$ (so that $2\pi/\gamma=3/2$); we could have chosen any $\gamma=\pi+\eps$ with $\eps$ small (leading to $2\pi/\gamma$ as close as 2 as required), though this specific choice of $\gamma$ is sufficient for our purpose, since our main goal is to allow an exponent strictly bigger than 1.}
Moreover, it is more convenient to work on harmonic functions, and then compare $u_\Sigma$ with its harmonic replacement to obtain a similar statement for $u_\Sigma$. Therefore the following lemmas are stated first in this more convenient framework of harmonic functions.

We
call then \emph{Dirichlet minimizer}  $u\in g+H^1_0(D\setminus \Sigma)$ any  function that minimizes the Dirichlet energy
$w\mapsto \int_{D} |\nabla w |^2\, dx$ over $g+H^1_0(D\setminus \Sigma)$.

\begin{lemma}\label{lm_compl_monot13}
Let  $\Sigma\subset\overline{\Om}$  be a closed and  connected  set,  $g\in H^1_{0}(\Om\setminus\Sigma)$, $x_{0}\in\Sigma$,
$2r_1 \varepsilon<r_{0}<r_1 < \diam(\Sigma)/2$
and
\begin{eqnarray}
\beta_\Sigma(x_0,r_1,r_{0}) \leq \varepsilon \label{betaestimm3}
\end{eqnarray}
for some $0<\varepsilon\leq 1/10$.
Then for every Dirichlet minimizer $u\in g+H_0^1(B_{r_1}(x_{0})\setminus\Sigma)$ and for every
$r\in [ 
r_{0}, r_1 ]$ one has
\[
\int_{B_{r}} |\nabla u|^2\, dx \leq 
4\left(\frac{r}{r_1}\right)^{\frac{3}{2}} \int_{B_{r_1}} |\nabla u|^2 \, dx .
\]
\end{lemma}

To get a proof of this result, we need first the following two lemmas: the first one is a version of Lemma~\ref{lm_compl_monot1} for harmonic functions, the second one is a purely geometric statement.

\begin{lemma}\label{lm_compl_monot12}
Let  $\Sigma\subset\overline{\Om}$  be a closed set, $g\in H^1_{0}(\Om\setminus\Sigma)$, $x_{0}\in \overline{\Omega}$. Assume   that~\eqref{eq:geomass} is valid
 and $\gamma \in[\gamma_{\Sigma}(x_0,r_{0},r_{1}),2\pi]\setminus\{0\}$ (see~\eqref{eq:defgamma} for a definition of $\gamma_{\Sigma}$).
Then for every Dirichlet minimizer $u\in g+H^1_0(B_{r_1}(x_{0}) \setminus\Sigma)$
one has that
the function
\begin{align*}
 r\in [r_0,r_1] \mapsto r^{-\frac{2\pi}{\gamma}}\int_{B_{r}(x_0)} |\nabla u|^2\, dx ,
\end{align*}
is nondecreasing.
\end{lemma}
The proof  of this lemma is exactly  the same as the one of Lemma~\ref{lm_compl_monot1} (with $f=0$) so we omit it.

As we said before, to replace the assumption~\eqref{eq:geomass} 
by another one relying only on the flatness $\beta_\Sigma(x_0,r_1,r_{0})$, we face the difficulty that it may happen, even if $\Sigma$ is connected, that $\partial B_{r}(x_{0})\cap\Sigma=\emptyset$ for some $r\in(r_{0}, r_{1})$.
 To handle this difficulty,
 we establish the following topological fact.

\begin{lemma} \label{topologicalfact}
Let $\Sigma\subset \R^2$ be a closed and arcwise connected set. 
Assume that $x_0 \in \Sigma$ and $2r_1 \varepsilon<r_{0}<r_1 < \diam(\Sigma)/2$ are such that
\begin{eqnarray}
\beta_\Sigma(x_0,r_1,r_{0}) \leq \varepsilon \label{betaestimm}
\end{eqnarray}
for some $0<\varepsilon \leq 1/10$.
Let $A\subset { [r_0 , r_1]}$ be the set of all $r$ for which
\[
\sup \left\{\frac{\H(S)}{r}\colon S \textrm{ connected component of }  \partial B_r(x_0)\setminus \Sigma \right\} > \pi+2\arcsin\Big(\frac{\varepsilon r_1}{r}\Big). \]
Then $A$ is an     interval  of length less than $2\varepsilon r_1$.
\end{lemma}
\begin{figure}[htbp]
\begin{center}
\includegraphics*[width=0.8\textwidth]{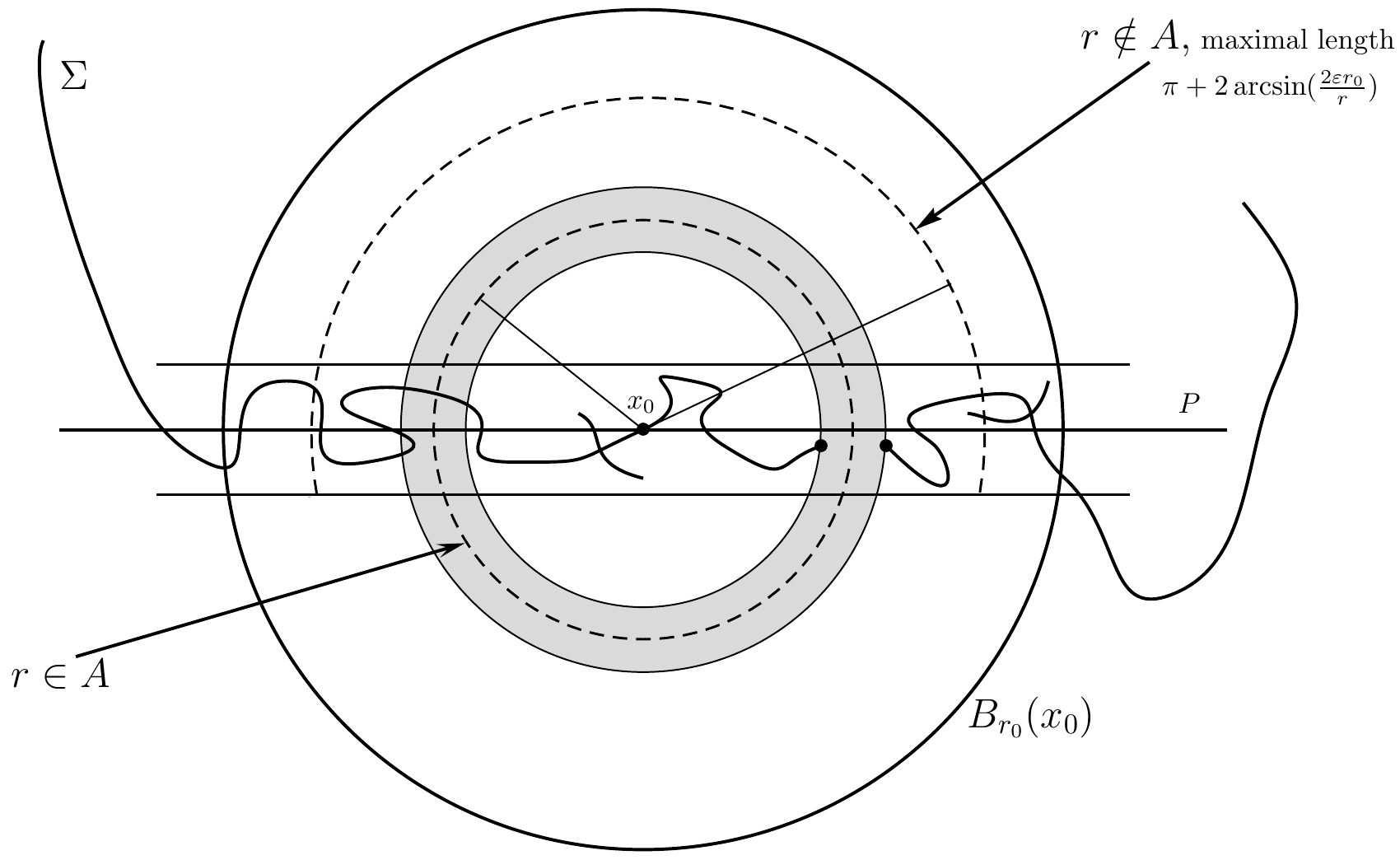}\label{fig:flat1}
\end{center}
\caption{  Situation in the proof of Lemma~\ref{topologicalfact}: $\beta_{\Sigma}(x_{0},r_{0},r_1)\leq\eps$}
\end{figure}%

\begin{remark} Let us recall that any compact and connected set $\Sigma\subset \R^2$ satisfying $\mathcal{H}^1(\Sigma)<+\infty$ is automatically arcwise connected (see for instance~\cite[Corollary~30.2, p.~186]{d}).
\end{remark}

\begin{proof}  In this proof, every ball is centered at $x_{0}$. The proof relies on the fact that $\Sigma \cap B_{r_1} \setminus B_{r_{0}}$ is localized in a strip that meets any sphere  $\partial B_{r}$ for $r\in [r_{0}, r_1]$ by two sides. Let $P_0$ be the line that realizes the infimum in the definition  of $\beta_\Sigma(x_0,r_1,r_{0})$. Let us identify $P_0\cap B_{r_{1}}$ with the segment $[-r_1,r_1]$.  By~\eqref{betaestimm} we know that for each point  $t$ such that $r_0\leq |t| \leq r_1$, there exists a point $z(t) \in \Sigma\cap B_{r_1}\setminus B_{r_0}$  which is $\varepsilon r_1$ close to $t$. The   point $z(t)$ connects a point outside $B_{r_{1}}$ via a curve, that must escape the ball $B_{r_1}$ either on the right or on the left.  Let $t^+$ be  the minimum of $t>0$ for which the curve escape from the right, and $t^-$ be the maximum of $t>0$ for which the corresponding curve escapes from the left. Then the length of the  interval $(t^-,t^+)$ must be smaller than $2\varepsilon r_1$   otherwise we would have a hole of so big size that   would contradict~\eqref{betaestimm}. And every sphere $\partial B_t$ for $t \in [r_{0}, r_1] \setminus [t^-,t^+]$ necessarily meet a point of $\Sigma$ in the $\varepsilon r_1$-strip around $P_0$. We deduce that $A$ is a interval contained in $[t^-,t^+]$ (see Figure~\ref{fig:flat1}).
\end{proof}

Now we are in position to prove Lemma~\ref{lm_compl_monot13}:

\begin{proof}[Proof of Lemma~\ref{lm_compl_monot13}.] In this proof, every ball is centered at $x_{0}$. Observe first that  $r\in [2\varepsilon r_1, r_1 ]\mapsto\pi +2\arcsin\left(\frac{\varepsilon r_1}{r}\right)$ is decreasing, and is therefore always smaller than $\pi + \frac{\pi}{3}$, achieved for $r=2\varepsilon r_1$. Let 
$A:=(t_1,t_2)$
be the set of bad radii given by Lemma~\ref{topologicalfact}.
When  $r \in [t_2,r_1]$, we apply Lemma~\ref{lm_compl_monot12} with $\gamma = \pi +\frac{\pi}{3}$ thus obtaining directly
\begin{eqnarray}
\int_{B_{r}} |\nabla u|^2\, dx \leq \left(\frac{r}{r_1}\right)^{\frac{3}{2}} \int_{B_{r_1}} |\nabla u|^2 \, dx . \label{dec1}
\end{eqnarray}
If $r \in (t_1,t_2)$. Then since~\eqref{dec1} holds for $t_2$ we can simply write
\begin{eqnarray}
\int_{B_{r}} |\nabla u|^2\, dx \leq \int_{B_{t_2}} |\nabla u|^2\, dx &\leq&  \left(\frac{t_2}{r_1}\right)^{1+\frac{1}{2}}
 \int_{B_{r_1}} |\nabla u|^2 \, dx \quad \text{by Lemma~\ref{lm_compl_monot12}}
\notag \\
&\leq & \left(\frac{r+\varepsilon r_1}{r_1}\right)^{\frac{3}{2}} \int_{B_{r_1}} |\nabla u|^2 \, dx \notag \\
&\leq & \left(\frac{2r}{r_1}\right)^{\frac{3}{2}} \int_{B_{r_1}} |\nabla u|^2 \, dx 
\leq
4\left(\frac{r}{r_1}\right)^{\frac{3}{2}} \int_{B_{r_1}} |\nabla u|^2 \, dx.
\end{eqnarray}
Finally, if  
$r \in [r_0, t_1]$,
we apply Lemma~\ref{lm_compl_monot12} on 
$ [r_0, t_1]$
which yields
\begin{eqnarray}
\int_{B_{r}} |\nabla u|^2\, dx &\leq& \left(\frac{r}{t_1}\right)^{\frac{3}{2}} \int_{B_{t_1}} |\nabla u|^2 \, dx
\leq
\left(\frac{r}{t_1}\right)^{\frac{3}{2}} \int_{B_{t_2}} |\nabla u|^2 \, dx . \notag
\end{eqnarray}
Combining this estimate with~\eqref{dec1}  applied with $r:=t_2$, we get
\begin{eqnarray}
\int_{B_{r}} |\nabla u|^2\, dx \leq  \left(\frac{r}{t_1}\right)^{\frac{3}{2}} \left( \frac{t_2}{r_1} \right)^{\frac{3}{2}} \int_{B_{r_1}} |\nabla u|^2 \, dx  \leq4 \left( \frac{r}{r_1} \right)^{\frac{3}{2}} \int_{B_{r_1}} |\nabla u|^2 \, dx,  \notag
\end{eqnarray}
because $t_2/t_1\leq 2$, concluding the proof.
\end{proof}

We are now ready to state a useful decay result on $u_\Sigma$.

\begin{proposition}\label{prop_compl_monot4}
Let $\Omega\subset\R^2$,  $f\in L^p(\Omega)$ with  $p > 2$, $\Sigma\subset \overline{\Om}$  be a closed connected set, $x_0\in \Sigma$.
Suppose
$2r_1 \varepsilon<r_{0}<r_1 < \diam(\Sigma)/2$
and
\begin{eqnarray}
\beta_\Sigma(x_0,r_1,r_{0}) \leq \varepsilon \label{betaestimm2}
\end{eqnarray}
for some $\varepsilon \in (0,1/10]$.

Then for every  $r\in   [ r_{0}, r_1 ]$ one has that
\begin{eqnarray}
\int_{B_{r}(x_0)} |\nabla u_\Sigma|^2\, dx \leq
8\left(\frac{r}{r_1}\right)^{\frac{3}{2}} \int_{B_{r_1}(x_0)} |\nabla u_\Sigma |^2 \, dx +C r_1^{\frac{2}{p'}}  \label{amontrer}
\end{eqnarray}
for some constant $C=C(|\Omega|, p, \|f\|_p)>0$.
\end{proposition}

\begin{proof} In this proof, every ball is centered at $x_{0}$. Let $w$ be the harmonic replacement of $u_\Sigma$ in $B_{r_1}$, i.e.\ $w$ is a Dirichlet minimizer in $B_{r_1}$: it is harmonic in $B_{r_1}\setminus \Sigma$ with Dirichlet condition $w=0$ on $\Sigma \cap B_{r_1}$ and $w=u_\Sigma$ on $\partial B_{r_1}\setminus \Sigma$. This function is found by minimizing the Dirichlet energy $\int_{B_{r_1}}|\nabla w|^2 dx$ among all $w$ satisfying $w-u_\Sigma\in H^1_0(B_{r_1}\setminus \Sigma)$. Then Lemma~\ref{lm_compl_monot13} applies to $w$ and says
\begin{eqnarray}
\int_{B_{r}} |\nabla w|^2\, dx \leq
4\left(\frac{r}{r_1}\right)^{\frac{3}{2}} \int_{B_{r_1}} |\nabla w|^2 \, dx \leq 
4\left(\frac{r}{r_1}\right)^{\frac{3}{2}} \int_{B_{r_1}} |\nabla u_\Sigma|^2 \, dx. \label{dec40}
\end{eqnarray}
The last inequality comes from the fact that $w$ minimizes the energy in $B_{r_1}$ and $u_\Sigma$ is a competitor. Now $u_\Sigma$ minimizes $E$ so that, extending $w$ by  $u_\Sigma$ outside $B_{r_1}$ we can write
\[
\int_{\Omega} |\nabla u_\Sigma|^2 \, dx  - 2\int_{\Omega} u_{\Sigma} f \, dx=2 E(u_\Sigma)
\leq 2 E(w) = \int_{\Omega} |\nabla w|^2 \, dx - 2\int_{\Omega} w_{\Sigma} f \,dx,
\]
that is,
\begin{equation}\label{decccc}
\begin{aligned}
\int_{B_{r_1}} |\nabla u_\Sigma|^2 \, dx  -  \int_{B_{r_1}} |\nabla w|^2 \, dx \leq   2\int_{B_{r_1}} (u_\Sigma-w) f \, dx \leq C r_1^{\frac{2}{p'}},
\end{aligned}
\end{equation}
where $C=C(|\Omega|, \|f\|_p, p)$,   where in the last estimate
we used H\"{o}lder inequality together with the fact that $u_\Sigma$ and hence $w$ are bounded (Proposition~\ref{boundus}). On the other hand, for any $r \in [r_{0}, r_1 ]$, we have
\begin{eqnarray}
\int_{B_{r}} |\nabla u_\Sigma|^2 \, dx &\leq & 2  \int_{B_{r}} |\nabla w|^2 \, dx + 2 \int_{B_{r}} |\nabla (w-u_\Sigma)|^2 \, dx \notag \\
&\leq & 2 \int_{B_{r}} |\nabla w|^2 \, dx + 2 \int_{B_{r_1}} |\nabla (w-u_\Sigma)|^2 \, dx \notag \\
&= & 2 \int_{B_{r}} |\nabla w|^2 \, dx + 2 \int_{B_{r_1}} |\nabla u_\Sigma|^2 \, dx  - 2 \int_{B_{r_1}} |\nabla w |^2 \, dx \label{deccc}
\end{eqnarray}
 because $\nabla w$ and $\nabla (u_\Sigma - w)$ are orthogonal in $L^2(B_{r_1}\setminus\Sigma)$. Combining~\eqref{dec40},~\eqref{decccc} and~\eqref{deccc} we arrive at~\eqref{amontrer} thus concluding the proof.
\end{proof}

Let us now recall the following notation.  For any $\Sigma \in \mathcal{K}(\Omega)$ we denote by
\begin{eqnarray}
\omega_\Sigma(x,r) = \max_{\Sigma' \in \mathcal{K}(\Omega) ;  \Sigma' \Delta \Sigma \subset \overline{B}_r(x) } \left(\frac{1}{r} \int_{B_r(x)} |\nabla u_{\Sigma'} |^2 \, dx \right). \label{MAXMAX}
\end{eqnarray}

\begin{remark} Problem~\eqref{MAXMAX} has a solution. Indeed, if $\Sigma_k$ be a maximizing sequence, then up to a subsequence $\Sigma_k \to \Sigma_0$ for the Hausdorff distance and clearly $\Sigma_{0} \Delta \Sigma \subset \overline{B}_r(x)$. But then \v{S}ver\'ak Theorem~\ref{th_Sverak1} says that $\nabla u_{\Sigma_k}$ converges strongly to $\nabla u_{\Sigma_0}$ in $L^2(\Omega)$ and $\Sigma_{0}$ is therefore a maximizer.
\end{remark}

Notice that we also have, for $a \in (0,1)$,
\begin{equation}
 \omega_\Sigma(x,a r)\leq \frac{1}{a}\omega_{\Sigma}(x,r).\label{eq_omega_compar1}
\end{equation}

The next proposition states that $\omega_{\Sigma}(x,\cdot)$ decays at smaller scales, provided that $\beta_\Sigma(x,\cdot)$ is small.

\begin{proposition}\label{decayomegaA}
Let $\Omega\subset\R^2$,  $f\in L^p(\Omega)$ with  $p > 2$, $\Sigma\subset  \overline{\Om}$  be a closed connected set, $x_0\in \Sigma$ and assume 
\[\beta_\Sigma(x_{0},r_{1})\leq \varepsilon\]
for some $0< \varepsilon < 1/10$ and $r_{1}< \diam(\Sigma)/2$.
Then for any $r \in  ( 2\varepsilon r_{1}, r_{1} )$ we have
\[
\omega_\Sigma(x_0,r) \leq  
8\left(\frac{r}{r_{1}}\right)^{\frac{1}{2}}\omega_\Sigma(x_0,r_{1}) + Cr_{1}^{\frac{2}{p'}}\frac{1}{r}
\]
for some constant $C=C(|\Omega|, p, \|f\|_p)>0$.
\end{proposition}

\begin{proof} Every ball in this proof is centered at $x_0$. Let $\Sigma'$ be the maximizer for~\eqref{MAXMAX}. Since $\Sigma\Delta \Sigma' \subset  B_r$ we have that
\[
\beta_{\Sigma'}(x_0,r_{1},r)\leq \varepsilon
\]
and therefore Proposition~\ref{prop_compl_monot4} (applied with $r_0:=r$) gives
for every  $r\in ( 2\varepsilon r_{1}, r_{1} )$ the estimate
\begin{eqnarray}
\omega_\Sigma(x_0,r)=\frac{1}{r}\int_{B_{r}} |\nabla u_{\Sigma'}|^2\, dx &\leq& 
8\left(\frac{r}{r_{1}}\right)^{\frac{1}{2}}  \frac{1}{r_{1}}\int_{B_{r_{1}}} |\nabla u_{\Sigma'} |^2 \, dx +C r_{1}^{\frac{2}{p'}} \frac{1}{r}\notag \\
&\leq & 
8\left(\frac{r}{r_{1}}\right)^{\frac{1}{2}}  \omega_\Sigma(x_0,r_{1})   +C r_{1}^{\frac{2}{p'}} \frac{1}{r},\label{amontrer2}
\end{eqnarray}
showing the claim.
\end{proof}

We now state a first estimate on defect of minimality.

\begin{proposition}\label{lm_compl_locC}
Let $\Omega\subset\R^2$,  $f\in L^p(\Omega)$ with  $p > 2$, $\Sigma\subset \overline{\Om}$  be a closed connected set,  and assume that $x_0\in \Sigma$, and
\[\beta_\Sigma(x_0,r_1)\leq \varepsilon\]
for some $0< \varepsilon < 1/10$ and $r_1< \diam(\Sigma)/2$.
Then for any $r \in  ( 2\varepsilon r_1, r_1 /2]$ and for any closed and connected set $\Sigma' \subset \overline{\Om}$ satisfying $\Sigma \Delta \Sigma'\subset B_r(x_0)$ the estimate
\begin{equation}\label{eq_compl_estabove1main}
E(u_{\Sigma})- E(u_{\Sigma'}) \leq Cr\left(\frac{r}{r_1}\right)^{\frac{1}{2}}\omega_\Sigma(x_0,r_1) +C  r_1^{\frac{2}{p'}} \end{equation}
holds with a constant $C=C(|\Omega|, p, \|f\|_p)>0$.
\end{proposition}

\begin{proof} Every ball in this proof is centered at $x_0$.  Take an arbitrary $\varphi\in \Lip(\R^2)$
as in the statement of Lemma~\ref{lm_compl_estabove1a}, i.e.
such that $\|\varphi\|_\infty\leq 1$, $\varphi=1$ over $B_{2r}^c$, $\varphi=0$ over $B_{r}$ and
$\|\nabla\varphi\|_{\infty}\leq 1/r$.
From Lemma~\ref{lm_compl_loc1} we have
\begin{align*}
E(u_{\Sigma})- E(u_{\Sigma'})
\leq  & \int_{B_{2r}} u_{\Sigma'} f(1-\varphi)\, dx + \int_{B_{2r}} u_{\Sigma'}^2 |\nabla \varphi|^2\, dx \\
& + \int_{B_{2r}} u_{\Sigma'} \varphi\nabla u_{\Sigma'}\cdot\nabla\varphi\, dx.
\end{align*}
We obtain then the following chain of estimates (with the constant $C$ changing from line to line)
\begin{eqnarray}
E(u_{\Sigma})- E(u_{\Sigma'})
&\leq  & C \int_{B_{2r}} |\nabla u_{\Sigma'}|^2 \, dx + C r^{\frac{2}{p'}} \quad\text{by Lemma~\ref{lm_compl_estabove1a}}
\notag \\
&\leq & Cr \omega_{\Sigma}(x_0,2r) + C r^{\frac{2}{p'}} \quad\text{by definition of $\omega_\Sigma$} \notag \\
&\leq & Cr \left(\frac{r}{r_1}\right)^{\frac{1}{2}}\omega_\Sigma(x_0,r_1) + Cr_1^{\frac{2}{p'}}+ C r^{\frac{2}{p'}}
\quad\text{by Proposition~\ref{decayomegaA}}
\notag \\
&\leq & Cr \left(\frac{r}{r_1}\right)^{\frac{1}{2}}\omega_\Sigma(x_0,r_1) + Cr_1^{\frac{2}{p'}}
\quad\text{because $r\leq r_1$},\notag
\end{eqnarray}
concluding the proof.
\end{proof}

It is worth observing (though we will not use it in the sequel) that monotonicity holds also for $\omega_\Sigma$, as the following statement asserts.

\begin{corollary}\label{corYeah2}
Let  $\Sigma\subset\overline{\Om}$  be a  connected and  closed set, $f\in L^p(\Omega)$, where $p>2$,
and
$x_0\in\Sigma $.
Then \begin{align*}
 r \mapsto \omega_{\Sigma}(x_{0},r) + Cr^{\frac{2}{p'}-1},
\end{align*}
is nondecreasing on $(0,\diam(\Sigma)/2)$, where   $C=C(|\Omega|, p, \|f\|_p)>0$.
\end{corollary}

\begin{proof}
Balls in this proof are centered at $x_0$.
Let $0<r_0<r_1 <\diam(\Sigma)/2$. Since $\Sigma$ is connected,~\eqref{eq:geomass} is satisfied for $r\in(0,r_{1}]$. Let $\Sigma'$ be a maximizer in~\eqref{MAXMAX} for $\omega_{\Sigma}( x_{0},r_{0})$; then $\Sigma'$ satisfies~\eqref{eq:geomass} for $r\in[r_{0},r_{1}]$. Applying Remark~\ref{remYeah} to $u_{\Sigma'}$   we get
\begin{eqnarray*}
\omega_{\Sigma}(x_{0},r_0) + Cr_0^{\frac{2}{p'}-1}&= &  \frac{1}{r_0}\int_{B_{r}} |\nabla u_{\Sigma'}|^2\, dx + Cr_0^{\frac{2}{p'}-1}\\
&\leq &  \frac{1}{r_1}\int_{B_{r_1}} |\nabla u_{\Sigma'}|^2\, dx + Cr_1^{\frac{2}{p'}-1} \\
 &\leq &\omega_\Sigma(x_{0},r_1) + Cr_1^{\frac{2}{p'}-1},
 \end{eqnarray*}
 showing the claim.
\end{proof}



\subsection{Density  estimates}

%
%
%

We now establish our main estimates about the density of the minimizer $\Sigma$ at the point $x_0\in \Sigma$. The first one is weak in the sense that $\beta$ appears in the error terms, but  will be used to find some good radius $s$ for which $\Card \Sigma \cap \partial B_s(x_0)=2$, as stated in the second item. The third item provides a similar estimate without paying $\beta$, provided that  $\Card \Sigma \cap \partial B_s(x_0)=2$; this will be our main estimate that leads to regularity.

\begin{proposition}\label{TheProp}    Let $\Omega\subset\R^2$ be a $C^1$ domain,  $f\in L^p(\Omega)$ with  $p > 2$, $\Sigma\subset \overline{\Om}$  be a  minimizer for Problem~\ref{pb_compl_pen1},  $C_0>0$ its upper Ahlfors regularity constant. Then there exists an $ r_{\Omega} >0$ (depending on $\Omega$) such that the following holds.  Assume that $x_0\in \Sigma$ and $0<r_1<\min(r_\Omega, \diam(\Sigma)/2)$.
Then the following assertions hold.  
\begin{itemize}
\item[(i)] If
\begin{equation}\label{eq_bet104}
   \beta_\Sigma(x_0,2r_1)\leq  \frac{10^{-5}}{20+C_0},
\end{equation}
then for any $r \in  [r_1/2, r_1]$,
\begin{eqnarray}
\H(\Sigma \cap B_r(x_{0})) \leq   2r+ (20+C_0)r\beta_\Sigma(x_{0},2r)+ Cr\left(\frac{r}{r_1}\right)^{\frac{1}{2}}\omega_\Sigma(x_{0},r_1) +C  r_1^{\frac{2}{p'}} \label{densityestimateKey}
\end{eqnarray}
with the constant $C=C(|\Omega|, p, \|f\|_p)>0$
such that the estimate~\eqref{eq_compl_estabove1main} from Proposition~\ref{lm_compl_locC} holds.
\item[(ii)]
If in addition to~\eqref{eq_bet104} the estimate
\begin{equation}\label{eq_om500C}
\omega_\Sigma(x_{0},r_1)+r_1^{\frac{2}{p'}-1}\leq \frac{1}{500C}
\end{equation}
is valid with $C>0$ such that~\eqref{eq_compl_estabove1main} (Proposition~\ref{lm_compl_locC}) holds, then
there exists some  $s\in [r_1/2,r_1]$ such that $\Card \Sigma \cap \partial B_s(x_0)=2$.
\item[(iii)] If both~\eqref{eq_bet104} and~\eqref{eq_om500C} hold, and $r\in [r_1/2,r_1]$  is such that   $\Card \Sigma \cap \partial B_r(x_{0})=2$,
then
\begin{itemize}
\item[(iii-1)] the two points  of $\Sigma \cap \partial B_r(x_{0})$
belong to two different connected components of $\partial B_{r}(x_0) \cap \{{y} : \dist({y},P_0) \leq \beta_{\Sigma}(x_0,r) r\}$, where $P_0$ is a line that realizes the infimum in the definition of $\beta_{\Sigma}(x_0,r)$,
\item[(iii-2)] $\Sigma\cap \overline{ B}_{r}(x_{0})$ is connected,
\item[(iii-3)]  Assume moreover that $B_{r_1}(x_0)\subset \Omega$. Then denoting $\{a_1,a_2\}=\Sigma \cap \partial B_r(x_{0})$ we have
\[
\H(\Sigma \cap B_r(x_{0})) \leq  |a_2-a_1| + Cr\left(\frac{r}{r_1}\right)^{\frac{1}{2}}\omega_\Sigma(x_{0},r_1) +C  r_1^{\frac{2}{p'}} .
\]
\end{itemize}
\end{itemize}
\end{proposition}


\begin{remark}
It is easy to see from the proof that for an internal point $x_0\in \Omega$
one can take $r_\Omega:=+\infty$. Thus if one applies this result only to internal points of $\Omega$, then the requirement that $\Omega\subset\R^2$ be a $C^1$ domain is unnecessary.
\end{remark}

\begin{remark} \label{bothside} In the sequel, when the situation of item ($iii$-$1$) will occur we will say, in short, that the two points lie ``on both sides''.
\end{remark}
\begin{proof} Every ball in this proof is centered at $x_{0}$ unless otherwise
explicitly stated.
We first prove~($i$).
Let us construct a competitor in $B_r$ for any  fixed $r\in [r_1/2, r_1]$.
Let $x_1$ and $x_2$ be the two points of $ \partial B_r\cap P_0$ and let $W$ be the little wall  defined by
\[
W:=\partial B_r \cap \{x \colon  d(x,P_0)\leq 2r\beta_\Sigma(x_{0},r) \}.
\]

We consider two cases.

{\sc Case A}: $B_{r}\subset \Omega$. Then consider the competitor (see Figure~\ref{fig:wall})
\[\Sigma':=(\Sigma \setminus B_r ) \cup W \cup [x_1,x_2].\]
\begin{figure}[htbp]
\begin{center}
\includegraphics*[width=0.8\textwidth]{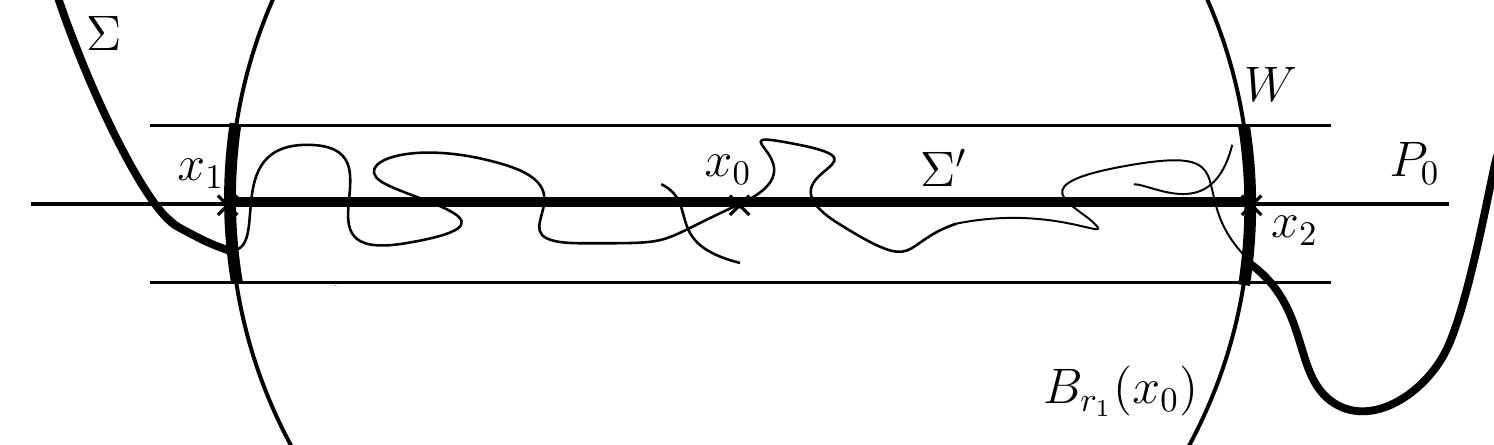}
\end{center}
\caption{The competitor $\Sigma'$ in the proof Proposition~\ref{TheProp}\,($i$) case one.}\label{fig:wall}
\end{figure}%

It is easily seen that $\Sigma'$ is a compact connected set satisfying $\Sigma\Delta \Sigma' \subset  \overline{B}_r$. Since $\Sigma$ is a minimizer, we have
\[\H(\Sigma)\leq \H(\Sigma') + E(u_\Sigma) - E(u_{\Sigma'}),\]
and thus
\begin{equation}\label{eq_PropHS1a}
\begin{aligned}
\H(\Sigma\cap B_r)& \leq 2r+ \H(W) + E(u_\Sigma) - E(u_{\Sigma'})\\
& \leq 2r+ \H(W) + Cr\left(\frac{r}{r_1}\right)^{\frac{1}{2}}\omega_\Sigma(x_0,r_1) +C  r_1^{\frac{2}{p'}}  \quad\text{by Proposition~\ref{lm_compl_locC}}.
\end{aligned}
\end{equation}
On the other hand,
\begin{equation}\label{eq_PropHS1b}
\H(W)\leq 
4r \arcsin(
\beta_{\Sigma}(x_0,r)) \leq 10r \beta_\Sigma(x_0,r),
\end{equation}
the latter inequality coming from the assumption $\beta_\Sigma(x_0,r)\leq 1/10$ minding that $|\arcsin'(z)|\leq 2$ for all $z\in [0,1/10]$.
The estimate~\eqref{eq_PropHS1a} combined with~\eqref{eq_PropHS1b} implies~($i$) in the case  $B_{r}\subset \Omega$.

{\sc Case B}: $\partial \Omega \cap B_{r} \not = \emptyset$.
Notice that, since $x_0$ may be very close to $\partial \Omega$,  the competitor $\Sigma'$  as before might not be contained in $\overline{\Omega}$   thus we need  to modify  it  in this situation. Letting $z\in \partial \Omega \cap B_{r}$ we have $B_{r}\subset B_{2r}(z)$. Moreover, since $\Omega$ is a $C^1$ domain, and in particular compact,
we can argue as in the proof of Lemma~\ref{decayboundary} to find $r_\Omega>0$ such that $\partial \Omega$ is very flat in all balls of radius less than $r_\Omega$, uniformly    on $\partial \Omega$. In other words such that $\beta_{\partial \Omega}(x,s)\leq 10^{-10}/2$  for all $x\in \partial \Omega$ and all $s\leq 2r_{\Omega}$. This means, since $2r\leq 2r_\Omega$, that   $\partial \Omega\cap B_{2r}(z)$ is localized in a very thin strip of height  $\delta:=4\cdot 10^{-10}r$ centered at $z$. Let us assume that this strip is oriented in the $\bf{e}_1$ direction and that $\Omega$ is situated below (i.e. touching the region $\{x_2<0\}$).
 Our aim is to translate locally $\Sigma'$ a little downwards, to insure that it lies in $\Omega$. For this purpose we construct a bi-Lipschitz mapping $\Phi$, equal to $\mathrm{Id}$ outside $B_{r(1+\delta)}$, and equal to $-\delta r{\bf e}_2$ in $B_{r_0}$ which will guarantee that  $\Phi(\Sigma')\subset \overline{\Omega}$ (see Figure~\ref{fig:wall2}).

More precisely, we let $\varphi :\R^+\to [0,\delta r_0]$  be a $1$-Lipschitz function, equal to $0$ on $[r(1+\delta),+\infty)$, and to $r\delta$ on $[0,r]$, and define $\Phi:\R^2\to \R^2$  as follows
$$
\Phi(x)= \mathrm{Id}- \varphi(|x-x_0|) {\bf e}_2.
$$
We notice that $\Phi$ is $2$-Lipschitz and maps $B_{r(1+\delta)}$ into itself. Next, we define as before the compact connected set
\[\Sigma':=(\Sigma \setminus B_r ) \cup (W\cap \overline{\Omega}) \cup [x_1,x_2],\]
and then the little translated one
$$\Sigma'':=\Phi(\Sigma')\subset \overline{\Omega}.$$
\begin{figure}[htbp]
\begin{center}
\includegraphics*[width=\textwidth]{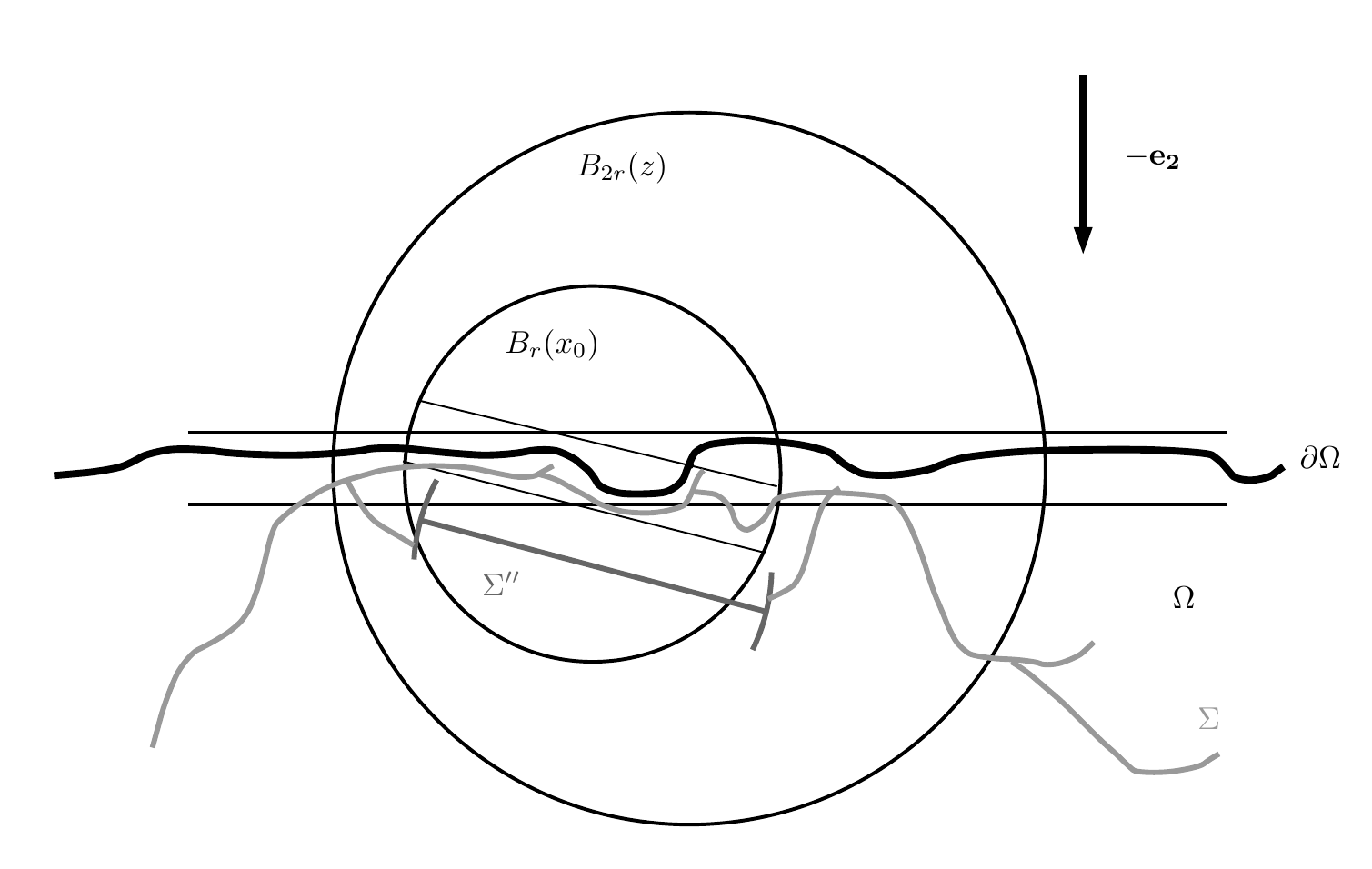}
\end{center}
\caption{The competitor $\Sigma''$ in the proof Proposition~\ref{TheProp}\,($i$), Case~B.}\label{fig:wall2}
\end{figure}%
The set $\Sigma''$ is still compact and connected and we observe
$$\Sigma''\setminus B_{r(1+\delta)} =\Sigma'\setminus B_{r(1+\delta)}=\Sigma \setminus B_{r(1+\delta)}.$$
Now we estimate $\mathcal{H}^1(\Sigma''\cap B_{r(1+\delta)})$. To this aim we decompose
$$B_{r(1+\delta)}=\Phi(\overline{B}_{r})\cup\Phi(B_{r(1+\delta)}\setminus \overline{B}_r).$$
On $B_r$, the mapping $\Phi$ is just a translation so that
$$
\mathcal{H}^1(\Sigma''\cap \Phi(B_r))=\mathcal{H}^1(\Phi(\Sigma'\cap B_r))=\mathcal{H}^1(\Sigma'\cap B_r),
$$
hence using ~\eqref{eq_PropHS1b} we get
$$
\mathcal{H}^1(\Sigma''\cap \Phi(B_r))\leq 2r+ 10r \beta_\Sigma(x_0,r).
$$
On the other hand, since $\Phi$ is $2$-Lipschitz,
\begin{eqnarray}
\mathcal{H}^1(\Sigma''\cap \Phi(B_{r(1+\delta)}\setminus \overline{B}_r))&=&\mathcal{H}^1(\Phi(\Sigma'\cap  B_{r(1+\delta)}\setminus \overline{B}_r)) \notag \\
&\leq& 2 \mathcal{H}^1(\Sigma'\cap  B_{r(1+\delta)}\setminus \overline{B}_r))   \notag \\
&=&2 \mathcal{H}^1(\Sigma \cap  B_{r(1+\delta)} \setminus \overline{B}_r))\notag \\
&\leq &  \mathcal{H}^1(\Sigma \cap  B_{r(1+\delta)} \setminus \overline{B}_r)) + 2rC_0\beta_{\Sigma}(x_0,2r), \notag
\end{eqnarray}
where $C_0$ is the upper Ahlfors regularity constant of $\Sigma$, because $\Sigma \cap  B_{r(1+\delta)}\setminus \overline{B}_r$   is localized in the union of two balls of radius bounded by $2r\beta_{\Sigma}(x_0,2r)$.

Finally,   thanks to~\eqref{eq_bet104} we still have $\beta_{\Sigma''}(x_0,r_0(1+\delta))\leq 1/10$,  so that Proposition~\ref{lm_compl_locC} applies.
 Since $\Sigma$ is a minimizer, we have
\[\H(\Sigma)\leq \H(\Sigma'') + E(u_\Sigma) - E(u_{\Sigma''}),\]
and thus using  Proposition~\ref{lm_compl_locC} we infer that
\begin{eqnarray*}
\H(\Sigma\cap B_{r})& \leq&  \H(\Sigma''\cap B_{r(1+\delta)})+ E(u_\Sigma) - E(u_{\Sigma''})\\
& \leq& 2r+ 10r \beta_\Sigma(x_0,r)+ rC_0\beta_{\Sigma}(x_0,2r) + Cr\left(\frac{r}{r_1}\right)^{\frac{1}{2}}\omega_\Sigma(x_0,r_1) +C  r_1^{\frac{2}{p'}} \\
&\leq & 2r+ (20+C_0)r \beta_\Sigma(x_0,2r) + Cr\left(\frac{r}{r_1}\right)^{\frac{1}{2}}\omega_\Sigma(x_0,r_1) +C  r_1^{\frac{2}{p'}},
\end{eqnarray*}
 where in the last inequality~\eqref{eq_beta_ar1} has been used.
This concludes the proof of~($i$) also in case~(B).


To prove~($ii$), we use Lemma~\ref{topologicalfact} (with $\varepsilon:= 10^{-4}$, and $r_0:=10^{-3}r_1$) which implies that,
\begin{eqnarray}\label{eq:doublepoints}
\H(\{s \in [10^{-3}r_1, r_1] \; : \;  \Card \Sigma \cap \partial B_{s}  < 2 \})& \leq& 2\cdot 10^{-4}r. \label{H1SS}
\end{eqnarray}
Next, applying~($i$) with $r:=r_1$, we get, thanks to our conditions on $r_1$ and $\omega_\Sigma(x,r_1)$,
\begin{eqnarray}
\H(\Sigma \cap B_{r_1}) \leq   (2+ 3/100)  r_1. \label{estimationon}
\end{eqnarray}
We shall use the following Eilenberg  inequality (\cite[2.10.25]{f}   or~\cite[Section~26, Lemma~1, p.~160]{d}),
\begin{eqnarray}
\H(\Sigma \cap B_{r_1}) \geq  \int_{0}^{ r_1} \Card \Sigma \cap \partial B_s\, ds. \label{HAA}
\end{eqnarray}
Let us define three sets
\begin{align*}
E_1& := \left\{s \in [0, r_1] \colon \Card \Sigma \cap \partial B_s  = 1 \right\},
E_2
:= \left\{s \in [0, r_1] \colon \Card \Sigma \cap \partial B_s  = 2 \right\},\\
E_3& := \left\{s \in [0, r_1] \colon\Card \Sigma \cap \partial B_s  \geq 3 \right\}.
\end{align*}
In particular,~\eqref{H1SS} says that $\H(E_1)\leq 10^{-3}r_1+2\cdot 10^{-4}r_1\leq 10^{-2} r_1$, and hence, using also~\eqref{estimationon} and~\eqref{HAA}   we deduce that
\begin{eqnarray}
 (2+ 3/100)  r_1  &\geq& \H(E_1) +2 \H(E_2) + 3\H(E_3) \notag \\
 &=  &  \H(E_1) +2 ( r_1 - \H(E_3)
 - \H(E_1) ) + 3\H(E_3) \notag \\
 &= & -\H(E_1) +2 r_1  + \H(E_3), \notag
\end{eqnarray}
from which we get
\[
\H(E_3)\leq \frac{3r_1}{100} + \H(E_1)\leq \frac{4r_1}{100}\leq \frac{r_1}{2},
\]
and therefore   $E_2 \cap [r_1/2,  r_1]\neq \emptyset$.

It remains to prove~($iii$). We argue in a similar way, but now since $\Sigma \cap \partial B_r=\{a_1,a_2\}$, i.e.\ $\Card \Sigma \cap \partial B_r=2$, the wall set $W$ is no more needed to make a competitor, so we get a better estimate by taking  as a competitor a replacement of $\Sigma$ by just a segment inside $B_r$ joining the two points $a_1$ and $a_2$, which leads to the estimate in~($iii$-3).
\begin{figure}[htbp]
\begin{center}
\includegraphics*[width=0.7\textwidth]{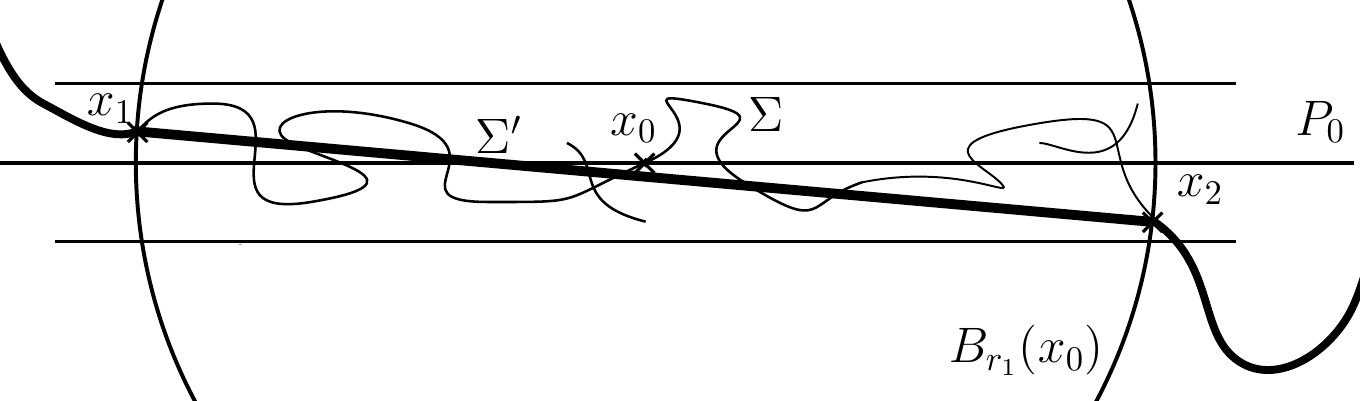}
\end{center}
\caption{The competitor $\Sigma'$ in the proof Proposition~\ref{TheProp}\,($iii$).}\label{fig:wall1}
\end{figure}%

Thus we only need to prove~($iii$-1) and~($iii$-2). Let us first prove~($iii$-1). Supposing the contrary,  we could take as a competitor the set
\[
\Sigma' := (\Sigma \setminus B_r) \cup  (W\cap \overline{\Omega})
\]
where $W$ is a little wall on one side of length less than $10r\beta_\Sigma(x,r_1)$. This would imply
\begin{eqnarray}
r\leq \H(\Sigma\cap B_r)&\leq&  10 r  \beta_\Sigma(x,r) + Cr\left(\frac{r}{r_1}\right)^{\frac{1}{2}}\omega_\Sigma(x,r_1) +C  r_1^{\frac{2}{p'}}  \notag \\
&\leq &  \frac{r}{100}   + \frac{r}{100} +\frac{r}{100}, \label{maindens}
\end{eqnarray}
hence a desired contradiction.

To prove~($iii$-2), we note that by Lemma~\ref{lm_SigmCapBall_conn1} below (with $D:=B_r\cap \Sigma$), if $\Sigma \cap \overline{B}_r$ is not connected, then  the set $\Sigma \setminus B_r$
has to be connected.
%
%
It follows that  $\Sigma \setminus B_r$ is a competitor, and as before, comparing the energy of $\Sigma$ and $\Sigma':=\Sigma\setminus B_r$ leads to
\begin{eqnarray}
r\leq \H(\Sigma\cap B_r)&\leq&    2r/100 ,
\end{eqnarray}
which is a desired contradiction.  Thus we have proven that~($iii$-2) holds true, so that the proof of~($iii$) is concluded.
\end{proof}

\begin{lemma}\label{lm_SigmCapBall_conn1}
If $\Sigma$ is an arcwise connected metric space, $D\subset \Sigma$ is its open subset, and
$\overline{D}$ is not arcwise connected, but $\Card \partial D=2$, then  the set $\Sigma \setminus D$ is arcwise connected.
\end{lemma}

\begin{proof}
Let $\partial  D:=\{a_1,a_2\}$, and consider an arbitrary couple of points $\{z_1, z_2\} \subset \Sigma \setminus D$
and an arc  $\Gamma\subset \Sigma$ that connects $z_1$ and $z_2$. Then either $\Gamma\cap \overline{D} =\emptyset$, in which case
$\Gamma\cap D =\emptyset$, or $\Card \Gamma\cap \overline{D} =1$,  in which case
again $\Gamma\cap D =\emptyset$ (because $\Gamma$ enters into $\overline{D}$ through, say, $a_1$, but it cannot enter $D$, since then
it must exit through the same point $a_1$, which is impossible by injectivity of $\Gamma$), or else $\Gamma\cap \overline{D} =\{a_1,a_2\}$, so that $\Gamma$ enters $\overline{D}$ at, say $a_1$ and exits at, say $a_2$. In this last case this means that
there exists a curve in $\overline{D}$ that connects $a_1$ to $a_2$, but then since other point in $\overline{D}$ is connected to  $z_1$ by some curve, that passes necessarily through either $a_1$ or $a_2$, this means that $\overline{D}$ is arcwise connected.
\end{proof}

\subsection{Flatness estimates}

We begin with a standard flatness estimate on curves coming from Pythagoras inequality.

\begin{lemma}\label{pythag} Let $\Gamma$ be an arc in $\overline{B}_r(x_0)$ satisfying   $\beta_\Gamma(x_0,r)\leq 1/10$,
and which connects two points  $x_1, x_2 \in \partial B_{r}(x_0)$ lying on both sides (as defined in Remark~\ref{bothside}). Then
\begin{eqnarray}
\dist(z,[x_1,x_2])^2 \leq 2r (\H(\Gamma)-|x_2-x_1|),  \quad \text{for all } z \in \Gamma. \label{estimgeom}
 \end{eqnarray}
\end{lemma}

\begin{proof}
 Assume that $z$ is the most distant point from the segment $[x_1,x_2]$ in $\Gamma$ and let $z'$ be the point making $(x_1,x_2,z')$ an isosceles triangle with same height  that we denote  $h=:\dist(z,[x_1,x_2])$, see Figure~\ref{figureB}.
 \begin{figure}[htbp]
\begin{center}
\includegraphics*[width=0.7\textwidth]{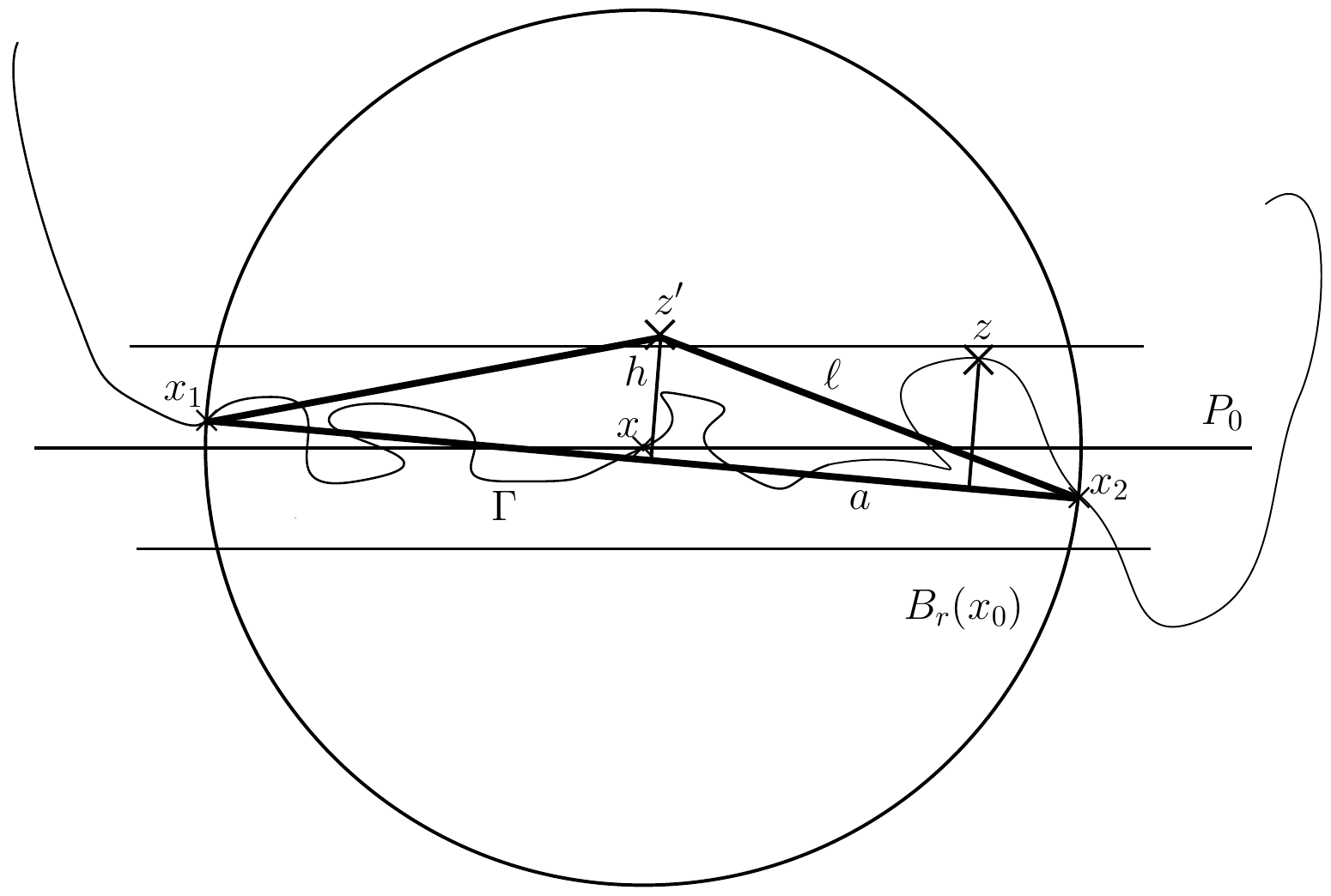}
\end{center}
\caption{Flatness estimate.}\label{figureB}
\end{figure}%
  Let $a:=|x_2-x_1|/2$ and $\ell := |z'-x_1|$. We have that
\[h^2 = \ell^2-a^2=(\ell -a) (\ell +a).\]
On the other hand $\H(\Gamma)\geq 2\ell$ so that
\[h^2\leq \frac{1}{2} (\H(\Gamma )-|x_1-x_2|) (\ell+a).\]
But now since $\beta_\Gamma(x_0,r)\leq 1/10$ it is easily seen that $\ell \leq \sqrt{101}r/10 \leq 2r$ and $a\leq r$, so follows the Lemma.
\end{proof}

\begin{remark} Notice that, contrary to the similar statements that sometimes can be found in the literature,  Lemma~\ref{pythag} would not be true replacing the curve  $\Gamma$ by an arbitrary closed connected set. Indeed, in an arbitrary connected set $\Sigma$, the curve from $x_1$ to $z$ and the curve from $z$ to $x_2$ may overlap, so that the sum of the  length of both curves may not be smaller than the length of $\Sigma$. Anyway, in the next proposition, we shall apply this lemma in an arbitrary connected set $\Sigma$ as follows: we first find an injective curve in $\Sigma$ from $x_1$ to $x_2$ and control the distance of that curve to the segment $[x_1,x_2]$. Then, this curve will have length at least $|x_1-x_2|$ and we will control the distance of the remaining parts of $\Sigma$ to $[x_1,x_2]$ by a density estimate which will say that the total length of the little forgotten pieces is very small (thus, by connectedness, very close in distance as well).
\end{remark}

We now can prove the existence of a threshold for which $\beta$ stays small at smaller scales as soon as it is small at one scale.

\begin{proposition}\label{firstdecay}  Let $\Omega\subset\R^2$,  $f\in L^p(\Omega)$ with  $p > 2$, $\Sigma\subset \overline{\Om}$  be a  minimizer, and $C_0$ be its
upper Ahlfors regularity constant.  Then there exist the numbers
$\tau_1$, $\tau_2$ with  $0<\tau_2 <\tau_1< \frac{10^{-5}}{20+C_0}$ and  $r_0>0$ such that whenever $x\in \Sigma$ and $0<r<r_0$ satisfy  $B_{r}(x) \subset \Omega$ and
\begin{eqnarray}
\beta_\Sigma(x,r)\leq \tau_1, \quad\quad \omega_\Sigma(x,r)\leq \tau_2, \label{toto1}
\end{eqnarray}
then
\begin{enumerate}
\item[(i)] ~\eqref{toto1} also holds with $r/16$ instead of $r$.
\item[(ii)] $$\omega_\Sigma(x,r/256) \leq  \frac{1}{2}\omega_\Sigma(x,r) + Cr^{\frac{2}{p'}-1}$$
\item[(iii)]  $$\beta_\Sigma(x,r/4)\leq  C (\omega_\Sigma(x,r))^{\frac{1}{2}} + Cr^{\frac{1}{p'}-\frac{1}{2}}$$
\end{enumerate}
with $C=C(|\Omega|, p, \|f\|_p)>0$.
\end{proposition}
\begin{proof} We first fix $r_0$ and $\tau_2$ small enough so that
\begin{eqnarray}
\tau_2+r_0^{\frac{2}{p'}-1}\leq \frac{1}{1000C}, \label{concondition}
\end{eqnarray}
 where $C$ appears in Proposition~\ref{lm_compl_locC}, and we also assume that
\begin{eqnarray}
 256Cr_0^{\frac{2}{p'}-1} \leq \frac{\tau_2}{2}.
\end{eqnarray}
Take an arbitrary $r\in (0,r_0)$ as in the statement. The control on $\omega_\Sigma$ will be achieved by use of Proposition~\ref{decayomegaA} which says that
$$\omega_\Sigma(x,r/256) \leq  8\left(\frac{1}{256}\right)^{\frac{1}{2}}\omega_\Sigma(x,r) + 256Cr^{\frac{2}{p'}-1}.$$
This proves $(ii)$. The proof of $(i)$ will be accomplished as soon as we prove $(iii)$ because
$$\beta_\Sigma(x,r/16)\leq 4\beta_{\Sigma}(x,r/4) \leq C \sqrt{\tau_2} +Cr_0^{\frac{1}{p'}-\frac{1}{2}}\leq \tau_1 ,$$
provided that we choose $\tau_2$ and $r_0$ small enough with respect to $\tau_1$.

Next, recalling~\eqref{concondition}, we apply Proposition~\ref{TheProp} ($ii$) in order to find some $s\in [r/4,r/2]$ such that ${\Card}\Sigma \cap \partial B_s =2$ (notice indeed that $\beta_\Sigma(x,r)\leq 10^{-5}/(20+C_0)$ and all the assumptions of Proposition~\ref{TheProp} are fulfilled with $x$, $r$ instead of $x_0$, $r_1$ respectively). Then assertion ($iii$) of Proposition~\ref{TheProp} says that the two points $a_1$ and $a_2$ of $\Sigma \cap \partial B_s(x)$ must lie on both sides, and $B_s(x)\cap \Sigma$ is connected. Moreover, Proposition~\ref{TheProp} ($iii$-3) says
\begin{eqnarray}
 \mathcal{H}^1(\Sigma \cap B_s(x))& \leq&  |a_2-a_1|  + Cs\left(\frac{s}{r}\right)^{\frac{1}{2}}\omega_\Sigma(x,r) +C  r^{\frac{2}{p'}}  \label{HYY} \\
 &=: & |a_2-a_1|  + R.\notag
 \end{eqnarray}
Let $\Gamma\subset \Sigma\cap \overline{B}_s(x)$ be an injective curve that connects $a_1$ and $a_2$.   Lemma~\ref{pythag} (with $s$ instead of $r$) says that
\begin{eqnarray}
\sup_{z\in \Gamma} \dist(z,[a_1,a_2])^2 &\leq& 2s (\HH^1(\Gamma)-|a_2-a_1|) \notag \\
&\leq& 2sR.
\end{eqnarray}
Next, since $\H(\Gamma)\geq |a_2-a_1|$ we also get from~\eqref{HYY} that
$$\H(\Sigma \cap B_s(x)\setminus \Gamma)\leq   R.$$
But this implies
$$
\sup_{z\in \Sigma \cap B_s(x)\setminus \Gamma} \dist(z,\Gamma) \leq R,
$$
and hence
\begin{eqnarray}
\sup_{z\in \Sigma \cap B_s(x) } \dist(z,[a_1,a_2]) \leq \sqrt{2r R}+R. \label{estimation1con}
\end{eqnarray}
From the connectedness of $\Sigma \cap B_s$ and~\eqref{estimation1con}, it is not difficult to deduce  that moreover
$$
\sup_{z\in [a_1,a_2] } \dist(z,\Sigma ) \leq \sqrt{2rR}+R,
$$
(because if it were not true we would have a too big hole in $\Sigma$) thus we have proved that
$$s\beta_\Sigma(x,s) \leq 2(\sqrt{2rR} +R )$$
(we have lost a further factor $2$ because $[a_1,a_2]$ may not be passing through $x$).
But then
$$\beta_\Sigma\left(x, \frac{r}{4}\right) \leq 2\beta_\Sigma(x,s) \leq \frac{4}{s}(\sqrt{2rR}+R).$$
Let us estimate the right hand side, with the constant $C>0$ possibly changing from line to line
\begin{eqnarray}
R&=&Cs\left(\frac{s}{r}\right)^{\frac{1}{2}}\omega_\Sigma(x,r) +C  r^{\frac{2}{p'}}\notag \\
&\leq & Cr\omega_\Sigma(x,r) + Cr^{\frac{2}{p'}} \label{etape} \quad \text{  (because } s\leq r/2\text{)}\\
&\leq&  Cr (\omega_\Sigma(x,r))^{\frac{1}{2}} + Cr^{\frac{1}{p'}+\frac{1}{2}}\notag
\end{eqnarray}
the latter inequality being due to $\omega_\Sigma(x,r)\leq 1$, $r\leq 1$ and $\frac{1}{p'}>\frac{1}{2}$. Moreover, from~\eqref{etape} we get
\begin{eqnarray}
\sqrt{2rR}&\leq & Cr (\omega_\Sigma(x,r))^{\frac{1}{2}} + Cr^{\frac{1}{p'}+\frac{1}{2}}. \notag
\end{eqnarray}
This concludes the proof of ($iii$), and so follows the proof of the Proposition.
\end{proof}

Next, we iterate the last proposition to obtain the following one.
\begin{proposition}\label{prop_compl_decay2}
Let $\Omega\subset\R^2$,  $f\in L^p(\Omega)$ with  $p > 2$, $\Sigma\subset \overline{\Om}$  be a  minimizer.  Then we can find  $0<\tau_2 <\tau_1$ and  $\bar r_0, C, \delta, \gamma>0$   such that whenever $x\in \Sigma$ and $r_0\leq \bar r_0$ are such that  $B_{r_0}(x) \subset \Omega$ and
\begin{eqnarray}
\beta_\Sigma(x,r_0)\leq \tau_1\;, \quad\quad \omega_\Sigma(x,r_0)\leq \tau_2 \;,\label{toto2}
\end{eqnarray}
then
\begin{eqnarray}
\omega_\Sigma(x,r)\leq  C \left(\frac{r}{r_0}\right)^{\frac{1}{4}} \tau_2+Cr^{\gamma}\quad  \text{for all } r \in(0,r_0/2)\label{decayome1}
\end{eqnarray}
and
\begin{eqnarray}
\beta_\Sigma(x,r)\leq  C\sqrt{\tau_2} \left(\frac{r}{r_0}\right)^{\frac{1}{8}} + Cr^{\delta}   \quad \text{for all } r\in(0,r_0/8).
\end{eqnarray}
for some $C$ depending on $|\Omega|, p, \|f\|_p$.
\end{proposition}
\begin{proof}  Let $\bar r_0$, $\tau_2$, $\tau_1$, $\gamma>0$, $C>0$ be the constants given by  Proposition~\ref{firstdecay}. Fix $a:=1/256$. By multiple application of  Proposition~\ref{firstdecay} we know that
$$\beta_\Sigma(x,a^nr_0)\leq \tau_1\quad \text{ and } \quad \omega_\Sigma(x,a^nr)\leq \tau_2 \quad \text{for all }n\geq 0. $$
We first prove by induction that for all $n\in \N$ we have, for some $\gamma>0$,
\begin{eqnarray}
\omega_\Sigma(x,a^n r_0) \leq  \frac{1}{2^n}\omega_\Sigma(x, r_0)+C(a^nr_0)^{\gamma}. \label{induction}
\end{eqnarray}
Let us fix $\alpha:=\frac{2}{p'}-1>0$ and define
\begin{align*}
\gamma & := \min\left( \frac{\alpha}{2},\; \frac{-\ln(3/4)}{\ln(256)}\right), 
\quad
t_0 
:=\left(\frac{1}{4}\right)^{\frac{1}{\gamma}}.
\end{align*}
This choice  of constants $\alpha$, $\gamma$, $t_0$ has been made to  guarantee that
\begin{eqnarray}
\frac{1}{2}t^{\gamma}+t^{\alpha} \leq (at )^{\gamma}\quad  \text{for all }t \in (0,t_0). \label{derniers}
 \end{eqnarray}
To check~\eqref{derniers}, we use that $\alpha\geq 2\gamma$ by definition and   $t\in (0,t_0)$ with $t_0\leq 1$ to estimate
$$\frac{1}{2}t^{\gamma}+t^\alpha\leq   \frac{1}{2}t^{\gamma}+t_0^\gamma t^{\gamma}\leq \frac{3}{4}t^{\gamma}\leq \left({\frac{1}{256}}\right)^{\gamma} t^{\gamma}=(at)^\gamma$$
because $\gamma\leq  -\frac{\ln(3/4)}{\ln(256)}$ implies $\frac{3}{4}\leq \left(\frac{1}{256}\right)^{\gamma}$.

It is clear that ~\eqref{induction} holds for $n=0$. Assume now that~\eqref{induction} holds for some $n$ and up to choose  $\bar r_0$ smaller assume  that $\bar r_0 \leq t_0$. Applying Proposition~\ref{firstdecay} ($ii$) in $B_{a^n r_0}(x)$ yields
\begin{eqnarray}
 \omega_\Sigma(x,a^{n+1}r_0)& \leq & \frac{1}{2}\omega_\Sigma(x,a^nr_0) + C(a^n r_0)^{\alpha} \notag \\
 &\leq &  \frac{1}{2}\left( \frac{1}{2^n}\omega_\Sigma(x,r_0)+C(a^nr_0)^{\gamma} \right)+ C(a^n r_0)^{\alpha} \notag \\
 &\leq &  \frac{1}{2^{n+1}} \omega_\Sigma(x,r_0)+ C(a^{n+1} r_0)^{\gamma} \notag,
\end{eqnarray}
concluding the proof of  ~\eqref{induction}.

Now let $r \in [0,r_0/2]$. Then $a^{n+1}r_0\leq r \leq a^n r_0$ for some $n$, and  using~\eqref{eq_omega_compar1},
\begin{eqnarray}
\omega_\Sigma(x,r)&\leq&  256 \omega_\Sigma(x,a^n r_0) \leq 256 (\frac{1}{2})^n \omega_\Sigma(x,r_0)+ 256C(a^nr_0)^{\gamma}\notag \\
&\leq & 256 \left(\frac{1}{256}\right)^{\frac{n}{4}} \omega_\Sigma(x,r_0)+256^{1+\gamma}C r^{\gamma} \leq 256^2 \left(\frac{r}{r_0}\right)^{\frac{1}{4}} \omega_\Sigma(x,r_0)+256^{1+\gamma}C r^{\gamma}. \label{decayome}
\end{eqnarray}
Now that $\omega_\Sigma$ decays, we can prove a similar decay for $\beta_\Sigma$. Let   $\alpha':=\frac{1}{p'}-\frac{1}{2}$. Then applying  Proposition~\ref{firstdecay} ($iii$) in $B_{a^n r_0}(x)$ and using~\eqref{decayome1} we get
\begin{eqnarray}
\beta_\Sigma(x,a^nr_0/4) &\leq&  C (\omega_\Sigma(x,a^{n}r_0))^{\frac{1}{2}} + C(a^nr_0)^{\alpha'} \notag \\
&\leq & C\Big( 256^2 \tau_2  a^{\frac{n}{4}}+C (a^nr_0)^{\gamma}\Big)^{\frac{1}{2}} + C(a^nr_0)^{\alpha'} \notag \\
&\leq & C \sqrt{ 256^2 \tau_2}  a^{\frac{n}{8}}+ \sqrt{C} (a^nr_0)^{\frac{\gamma}{2}} + C(a^nr_0)^{\alpha'} \notag
\end{eqnarray}
Now let $r\in [0,r_0/2]$ and choose $n$ such that $a^{n+1}r_0 \leq r \leq a^nr_0$. Then
\begin{eqnarray}
\beta_\Sigma(x,r/4) &\leq& 16 \beta_\Sigma(x,a^n r_0/4)\notag \\
&\leq & C\sqrt{\tau_2} a^{\frac{n}{8}} + Cr^{\frac{\gamma}{2}} +Cr^{\alpha'} \notag \\
&\leq &C\sqrt{\tau_2} a^{\frac{n}{8}} + Cr^{\frac{\gamma}{2}} +Cr^{\alpha'} \notag \\
&\leq &C\sqrt{\tau_2} \left(\frac{r}{r_0}\right)^{\frac{1}{8}} + Cr^{\frac{\gamma}{2}} +Cr^{\alpha'}.  \notag
 \end{eqnarray}
We finish the proof by setting $\delta = \max(\frac{\gamma}{2},\alpha')$.
\end{proof}



\begin{corollary}\label{TheCor} Let $\Omega\subset\R^2$,  $f\in L^p(\Omega)$ with  $p > 2$, $\Sigma\subset \overline{\Om}$  be a  minimizer.  Then we can find  $\varepsilon_0>0$ and $\bar r_0,C, \alpha>0$  such that whenever $x\in \Sigma$ and $r_0\leq \bar r_0$ are such that  $B_{r_0}(x) \subset \Omega$ and
\begin{eqnarray*}
\beta_\Sigma(x,r_0)+ \omega_\Sigma(x,r_0)\leq \varepsilon_0, \label{toto5}
\end{eqnarray*}
then there is an $\alpha\in (0,1)$ such that
\begin{eqnarray}
\beta_\Sigma(y,r)\leq  Cr^{\alpha}   \quad  \text{for all $y \in \Sigma\cap B_{r_0/2}(x)$ and  $r \in(0,r_0/16)$}. \label{regularityY}
\end{eqnarray}
In particular,
$\Sigma\cap B_{r_0/4}(x)$ 
is a
$C^{1,\alpha}$ 
regular curve.
\end{corollary}

\begin{proof} To get~\eqref{regularityY} fix $\varepsilon_0:=\tau_2/4$  where $\tau_2$ is defined on Proposition~\ref{prop_compl_decay2},  and apply the Proposition~\ref{prop_compl_decay2}  with $y \in \Sigma\cap B_{r_0/2}(x)$ and $r_0/2$ instead of $x_0$ and $r_0$ respectively, 
after noticing that $\beta_\Sigma(y,r_0/2) \leq 4\beta_\Sigma (x,r_0)$ and $\omega_\Sigma(y,r_0/2) \leq 4\omega_\Sigma (x,r_0)$.
Then it is quite standard to prove that~\eqref{regularityY} implies $C^{1,\alpha}$-regularity of $\Sigma$
(see, e.g.~proof of corollary~50.33 from~\cite{d}). 
\end{proof}

Note that at this point, using the uniform rectifiability of $\Sigma$ and a standard compactness argument, it would not be very difficult to  prove that $\Sigma$ is $C^{1,\alpha}$ regular outside a set of Hausdorff dimension $d<1$. But the blow-up analysis of the next section will actually prove much more, namely that $d=0$.

We finish this section by a last statement saying that Corollary~\ref{TheCor} still holds at the boundary, if the domain is convex.

\begin{proposition}~\label{TheCor2} Let $\Omega\subset\R^2$ be convex,  $f\in L^p(\Omega)$ with  $p > 2$, $\Sigma\subset \overline{\Om}$  be a  minimizer. Then we can find  $\varepsilon_0>0$ and positive constants $\bar r_0$, $C$, $\alpha$  such that whenever $x\in \Sigma \cap \partial \Omega$ and $r_0\leq \bar r_0$ are such that, $u_{\Sigma}$ being extended by zero outside $\Omega$,
\begin{eqnarray*}
\beta_\Sigma(x,r_0)+ \omega_\Sigma(x,r_0)\leq \varepsilon_0 \label{toto5bound}
\end{eqnarray*}
then there is an $\alpha\in (0,1)$ such that
\begin{eqnarray}
\beta_\Sigma(y,r)\leq  Cr^{\alpha}\quad
\text{for all $y \in \Sigma\cap B_{r_0/2}(x)$ and  $r \in(0,r_0/16)$}. \label{regularityY_bnd}
\end{eqnarray}
In particular,
$\Sigma\cap B_{r_0/4}(x)$ 
is a
$C^{1,\alpha}$ 
regular curve.
\end{proposition}

\begin{proof} If $\Omega$ is convex, then Problem~\ref{pb_compl_pen1} is exactly the same if we relax the class of competitors
\[
\{\Sigma \subset \overline{\Omega}, \text{ closed and connected} \}
\]
by
\[
\{\Sigma \subset \R^2, \text{ closed and connected} \}.
\]
Indeed, the projection $P_{\Omega}$ onto $\Omega$ is 1-Lipschitz thus minimizing on the second class would lead to the same minimizers since the projection of any competitor has a  lower value of $\mathcal{F}$.
This means that, in the case of convex domains, we may consider  $\Sigma$ as a subset of $\R^2$ with competitors in $\R^2$, and consider $u_{\Sigma}\in H^1_0(\Omega\setminus \Sigma)$ as a function of $H^1(\R^2)$ extending  it by zero outside $\Omega$, and then one can follow the whole Section 5 line by line and check that it works at the boundary straight away.



\end{proof}


\begin{remark} A consequence of Proposition~\ref{TheCor2} is that, if $\Omega$ is convex and  $C^1$ but nowhere $C^{1,\alpha}$, then $\Sigma$ will never want to stay inside $\partial \Omega$  for a while because it  would contradict the $C^{1,\alpha}$ regularity. In this case it can only touch $\partial \Omega$ pointwise.
\end{remark}


\section{Blow-up limits of minimizers and first consequences}

\subsection{Convergence of blow-up sequences for interior points}
 Throughout this section, let $\Sigma$ denote a minimizer for the Problem~\ref{pb_compl_pen1}, and $u:=u_{\Sigma}$.
 Let $\{x_n\}\subset \Sigma$ be a sequence of points and $r_n \to 0^+$ as $n\to\infty$. We define the blow-up sequence by
\begin{eqnarray}
\Sigma_n:=\frac{1}{r_n}(\Sigma-x_n), \quad \Omega_n:=\frac{1}{r_n}(\Omega-x_n),  \label{blowup1}\\
u_n(x):=r_{n}^{-\frac{1}{2}}u(r_nx+x_n) \in H^1_{0}(\Omega_n\setminus \Sigma_n). \label{blowup2}
\end{eqnarray}
Our scaling implies that
\[\int_{\Omega_n}|\nabla u_n|^2 \,dx +\H(\Sigma_n)=\frac{1}{r_n}\int_{\Omega}|\nabla u|^2 \, dx+\frac{1}{r_n}\H(\Sigma).\]
It follows that, for all $n$,  $(u_n,\Sigma_n)$ is a minimizer in $\Omega_n$ for the optimal compliance problem associated to the function
\[f_n(x):=r_{n}^{3/2} f(r_n x +x_n).\]

For a given closed set $\Sigma \subset \R^2$ we consider the subspace $H^1_{0,\Sigma,loc}(\R^2)$ of $H^1_{loc}(\R^2)$ consisting of functions vanishing on $\Sigma$ defined by
\begin{eqnarray}
H^1_{0,\Sigma,loc}(\R^2):=\left \{ u \in H^1_{loc}(\R^2) \colon \varphi_R u \in H^1_0(B_{R}(0)\setminus \Sigma) \text{ for all } R>0\right\}, \label{spaceH}
\end{eqnarray}
where $\varphi_R$ is the Lipschitz cut-off function $\varphi_R(x):=\max(R-|x|,0)$.
\begin{proposition} \label{blowuplimit} Assume that  $f\in L^p(\Om)$, $p>2$,
$x_n\to x_0\in \Sigma \cap \Omega$. Then there exist some closed set $\Sigma_0 \subset \R^2$, a function $u_0 \in H^1_{0,\Sigma_0,loc}(\R^2)$, and a subsequence $r_{n_k}\to 0$ such that $(u_{n_k},\Sigma_{n_k})$ converges to $(u_0,\Sigma_0)$ in the following way: for every ball $B\subset \R^2$,
\begin{itemize}
\item[(i)] $\Sigma_{n_k}  \to\Sigma_0 $ in the Kuratowski sense in $\R^2$,
and, moreover, $(\Sigma_{n_k}\cap B)\cup \partial B\to (\Sigma_0\cap B)\cup \partial B$ in the Hausdorff distance in $\R^2$,
\item[(ii)] $u_{n_k}\to u_0$ strongly in $H^1(B)$.
\end{itemize}
In addition, $u_0$ is harmonic in $\R^2\setminus \Sigma_0$.
\end{proposition}

\begin{proof}
For simplicity we will not relabel subsequences. We first extract a subsequence so that~($i$) holds. To this aim, let us stress that, in general, if a sequence of sets converge for  the Hausdorff distance, then their restrictions to subsets may not converge. This is why the notion of local Hausdorff converging sequence in $\R^2$ is delicate (see~\cite{d}) and we shall not try to make $\Sigma_n\cap B$ converging to $\Sigma_0\cap B$ for every ball for the  Hausdorff disance.
Instead, we start by  using Blaschke principle in the Alexandrov one-point compactification of $\R^2$, to extract a subsequence of $\Sigma_n$ converging to some set $\Sigma_0$ in the compactified space. This implies the convergence of $\Sigma_n$ in the Kuratowski sense in the compactified space, hence also in $\R^2$,  and it also implies that, for every ball $B\subset \R^2$, $(\Sigma_n\cap B)\cup \partial B\to (\Sigma_0\cap B)\cup \partial B$ for the classical Hausdorff distance (in fact, it clearly converges in Kuratowski sense, and hence in Hausdorff distance because $(\Sigma_n\cap B)\cup \partial B$ are all included in the same compact set $\overline{B}$).


We now claim that
\begin{itemize}
\item[(S)]
for every ball $B=B_R(0)\subset \R^2$ one may extract a further subsequence (depending on $B$) such that
for some $u_0\in H^1(B)$ harmonic in $B\setminus \Sigma_0$
and satisfying $\varphi u_0\in H_0^1(B\setminus\Sigma_0)$ for all $\varphi\in C_0^\infty(B)$,
 one has
$u_n\to u_0$ strongly in $L^2(B)$ and $\nabla u_n\to \nabla u_0$ strongly in $L^2(B_{R/2};\R^2)$.
\end{itemize}

To show claim~(S), use first the change of variables to get
\begin{align*}
\int_{B} |\nabla u_n|^2 \, dx=\frac{1}{r_n}\int_{B_{R r_n}(x_n)}|\nabla u_\Sigma|^2 \, dx \leq C,
\end{align*}
where $C>0$ depends only on $\|f\|_p$ and $R$,
the latter estimate being due to Lemma~\ref{lm_compl_monot1} applied with $\gamma=2\pi$ (see Remark~\ref{remYeah}).
Together with the Poincar\'e inequality (Lemma~\ref{lm_compl_poincare1}) this implies
\[
\|u_n\|_{H^1(B)}\leq C,
\]
where still  $C=C(\|f\|_p,R)>0$.
This means that, up to a subsequence, $u_n\rightharpoonup u_0$ weakly in $H^1(B)$ for some $u_0 \in H^1(B)$ (hence strongly in $L^2(B)$). This shows that $\varphi u_0\in H_0^1(B\setminus\Sigma_0)$ for any Lipschitz function $\varphi$ vanishing on $\partial B$ (see Remark~\ref{mosco}).

It is easy to verify that  $-\Delta u_0 = 0$ in $B\setminus \Sigma_0$ because, for all $\varphi \in C^\infty_0(B\setminus \Sigma_0)$, the Hausdorff convergence of $\Sigma_n$ implies $\varphi \in C^\infty_0(B\setminus \Sigma_n)$ for all sufficiently large $n\in \N$ and therefore
\begin{eqnarray}
\int_{\Omega} \nabla u_n \cdot \nabla  \varphi\, dx = \int_{\Omega} \varphi f_n\, dx. \label{test}
\end{eqnarray}
Passing to the limit in $n$ and recalling that $f_n\to 0$ strongly in $L^2(B)$, $\nabla u_n \to \nabla u$ weakly in $L^2(B)$, we infer that
\[
\int_{\Omega} \nabla u_0 \cdot \nabla \varphi = 0, \quad \text{for all  $\varphi \in C^\infty_0(B\setminus \Sigma_0)$}
\]
i.e.\ $u_0$ is harmonic in $\R^2\setminus\Sigma_0$.

Let $\varphi$
 be a $2$-Lipschitz cut-off function equal to $1$ on $B_{(3/4)R}$ and zero outside of $B$, and consider the truncated function $\varphi  u_0$, that is equal to $u_0$ on $B_{(3/4)R}$ and zero on $\partial B$.
Note that $\varphi u_0\in H_0^1(B\setminus\Sigma_0)$.
Since $\tilde{\Sigma}_n:=(\Sigma_n\cap B)\cup \partial B$ is connected for all $n\in \N$ and converges with respect to the Hausdorff distance in $\overline{B}$ to  $(\Sigma_0\cap B)\cup \partial B$, by
the
convergence of $H^1_0(B\setminus \tilde\Sigma_n)$ to $H^1_0(B\setminus \Sigma)$ in the sense of Mosco (in view of Remark~\ref{mosco}, namely, since for all
 $u\in H^1_0(B\setminus \Sigma)$ there is a sequence $\{u_n\}\subset H^1_0(B\setminus \tilde\Sigma_n)$ converging to $u$ in $H^1(B)$), we obtain the existence of a sequence $\{\tilde{u}_n\}\subset H^1_0(B\setminus \tilde{\Sigma}_n)$ converging to  $\varphi u_0$ strongly in $H^1(B)$.
 In particular, $\tilde{u}_n\to u_0$ strongly in $H^1(B_{(3/4)R})$.

To prove that $\nabla u_n\to \nabla u_0$ strongly in $L^2(B_{R/2};\R^2)$, it suffices, in view of the weak convergence of
$\nabla u_n\rightharpoonup \nabla u_0$ in $L^2(B_{R/2};\R^2)$, to prove
$\|\nabla u_n\|_{L^2(B_{R/2})}\to \|\nabla u_0\|_{L^2(B_{R/2};\R^2)}$ which in turn reduces to proving
\begin{eqnarray}
\int_{B_{R/2}}|\nabla u_0|^2 \, dx \geq  \limsup_{n}   \int_{B_{R/2}} |\nabla u_n|^2 \, dx \label{eq_limsupNablu1},
\end{eqnarray}
because the opposite inequality for $\liminf$ is automatic from weak convergence.
To prove~\eqref{eq_limsupNablu1},
consider also another cut-off function $\psi$ satisfying, for a parameter $0<s<1/2 $,
 \[
 \psi(x)=
 \left\{
 \begin{array}{l}
 1 \text{ on } B_{R/2},\\
 0 \text{ outside } (1+s) B_{R/2},\\
  \text{linear in } |x| \text{ otherwise},
 \end{array}
 \right.
 \]
and set
\[
w_n := \psi \tilde{u}_n + (1-\psi) u_n.
\]
since $u_n$ is a minimizer of
\[
u\in H_0^1(\Om_n\setminus\Sigma_n)\mapsto \int_{\Omega_n} |\nabla u|^2 \, dx -2\int_{\Omega_n} u f_n,
\]
and since $w_n$ is a competitor, we obtain
\[
\int_{\Omega_n} |\nabla u_n|^2 \, dx -2\int_{\Omega_n} u_n f_n \leq \int_{\Omega_n} |\nabla w_n|^2 \, dx -2\int_{\Omega_n} w_n f_n,
\]
thus, since $w_n=u_n$ outside $(1+s)B$,
\begin{eqnarray}
\int_{(1+s) B_{R/2}} |\nabla u_n|^2 \, dx -2\int_{(1+s) B_{R/2}} u_n f_n \leq \int_{(1+s) B_{R/2}} |\nabla w_n|^2 \, dx -2\int_{(1+s) B_{R/2}} w_n f_n. \label{corona}
\end{eqnarray}
Let us compute
\[
\nabla w_n =  \psi \nabla  \tilde{u}_n + (1-\psi) \nabla u_n + (\tilde{u}_n-u_n)\nabla \psi,
\]
which yields, denoting by $H_n:= \psi \nabla  \tilde{u}_n + (1-\psi) \nabla u_n $, the relationship
\[|\nabla w_n|^2 = |H_n|^2 +  |\tilde{u}_n-u_n|^2 |\nabla \psi|^2+2
H_n\cdot (\tilde{u}_n-u_n)\nabla \psi. \]
Integrating over $C_s:=(1+s) B_{R/2}\setminus B_{R/2}$ and recalling that $\nabla w_n=\nabla \tilde{u}_n$ in $B$, we can write~\eqref{corona}  in the form
\begin{eqnarray}
\int_{(1+s) B_{R/2}} |\nabla u_n|^2 \, dx  &\leq &  \int_{B_{R/2}} |\nabla \tilde{u}_n|^2 \, dx +   \int_{C_s}|H_n|^2 \,dx+  R_n, \label{inequ0}
\end{eqnarray}
 where
 \[R_n := 2\int_{(1+s) B_{R/2}} (u_n-w_n) f_n \,dx+ \int_{C_s} |\tilde{u}_n-u_n|^2 |\nabla \psi|^2+2 H_n\cdot (\tilde{u}_n-u_n)\nabla \psi \, dx.\]
By the convexity of $|\cdot|^2$ we get the inequality
 \[|H_n|^2\leq \psi |\nabla  \tilde{u}_n |^2+ (1-\psi) |\nabla u_n|^2,\]
 so that~\eqref{inequ0} becomes
 \begin{eqnarray*}
\int_{(1+s) B_{R/2}} |\nabla u_n|^2 \, dx  &\leq &  \int_{B_{R/2}} |\nabla \tilde{u}_n|^2 \, dx +R_n \\
&+&  \int_{C_s}\psi |\nabla  \tilde{u}_n |^2+ \int_{C_s} (1-\psi) |\nabla u_n|^2\,dx,
\end{eqnarray*}
 and in view of $\psi\leq 1$ we finally get
  \begin{equation}\label{eq_limsubNablu2}
\int_{B_{R/2}} |\nabla u_n|^2 \, dx  \leq   \int_{(1+s) B_{R/2}} |\nabla \tilde{u}_n|^2 \; dx +R_n.
\end{equation}
Notice now that $R_n\to 0$, because all the functions $\tilde u_n,  u_n, w_n$ converge strongly in $L^2((1+s)B)$ to the same function $u_0$,  the sequence $\{H_n\}$ is uniformly bounded in $L^2(C_s)\subset L^2(B)$, and $f_n\to 0$ in
 $L^2((1+s)B)\subset L^2(B)$. Therefore,
  passing to the limsup in~\eqref{eq_limsubNablu2} and using the strong convergence of $\tilde{u}_n$ to $u_0$ in $H^1((1+s)B_{R/2})\subset H^1(3/4 B)$, we get
\[
\limsup_{n} \int_{B_{R/2}} |\nabla u_n|^2 \, dx  \leq   \int_{(1+s)B_{R/2}} |\nabla u_0|^2 \, dx,
\]
which gives~\eqref{eq_limsupNablu1} by taking the limit in
$s\to 0^+$, hence completing the proof of strong convergence of $\nabla u_n$ to $\nabla u$ in $L^2(B_{R/2})$ and therefore
concluding the proof of claim~(S).

Finally, once the claim~(S) is proven, it suffices to choose for each $m\in \N$ a subsequence
$\{n(m,j)\}_j\subset \N$ such that $\{n(m+1,j)\}_j\subset \{n(m,j)\}_j$ and
each sequence $\{u_{n(m,j)}\}_j$ is convergent strongly some function in $L^2(B_{2m})$ with the sequence
of gradients $\{\nabla u_{n(m,j)}\}_j$ convergent strongly to the gradient of the same function in $L^2(B_{m}(x_0);\R^2)$, with
the limit function as in claim~(S) (with $B_{2m}$ instead of $B$). Taking the diagonal sequence
$\{u_{n(m,m)}\}_m$ we have that there is a $u_0\in H^1_{loc}(\R^2)$ harmonic in $\R^2\setminus \Sigma_0$
such that $\varphi u_0\in H_0^1(B_R\setminus\Sigma_0)$ for every $R>0$ and $\varphi$ vanishing over $\partial B_R$, with
$\{u_{n(m,m)}\}_m$ convergent strongly to $u$ in $L^2(B_R)$ and the sequence
of gradients $\{\nabla u_{n(m,m)}\}_m$  convergent strongly to $u$ in $L^2(B_R;\R^2)$. This completes the proof of~($ii$).
\end{proof}

\subsection{Two compactness estimates}

 The goal of  this section is to  prove that, as soon as $\beta_{\Sigma}(x,r)$
with $x\in \Sigma$
is small enough, then all the assumptions of the $C^1$ regularity result are satisfied. To this aim we need to control the energy $\omega_\Sigma$ of $u$ from the flatness $\beta_\Sigma$, and this is done in the following proposition via a compactness argument.

\begin{proposition}\label{gouter2}
Assume that $f\in L^p(\Om)$, $p>2$,
$\Sigma$ be a minimizer of  Problem~\ref{pb_compl_pen1} and let $x\in \Sigma\cap \Omega$ be such that  $\beta_\Sigma(x,r)\underset{r\to 0^+}{\longrightarrow}  0$. Then
\begin{eqnarray}
\frac{1}{r}\int_{B_r(x)}|\nabla u_\Sigma|^2 \, dx \underset{r\to 0^+}{\longrightarrow} 0. \label{limitingbeta}
\end{eqnarray}
\end{proposition}

\begin{proof}   The balls in this proof are all centered in $x$.
By Remark~\ref{remYeah}  we know that the limit in~\eqref{limitingbeta} exists, and is finite. Assume by contradiction that the limit is equal to some $C>0$. Then by considering the blow-up sequence $u_\Sigma(r_ny+x)/\sqrt{r_n}$ we know that  it converges as $r_n\to 0$ in $\R^2$, to some harmonic function $u_0$ in the complement of a line, with $u_0=0$ on that line. By the strong convergence in $H^1_{loc}(\R^2)$ we infer that $u_0$ has constant normalized Dirichlet energy, equal to $C$, in other words
\[
\int_{B_R} |\nabla u_n|^2\, dx = \frac{1}{r_n}  \int_{B_{Rr_n}} |\nabla u_\Sigma|^2\, dx \to RC,
\]
thus
\begin{eqnarray}
\frac{1}{R}\int_{B_R} |\nabla u_0|^2 \,dx=C \quad \text{ for all } R>0.  \label{constnormen}
\end{eqnarray}
But then by using a decomposition of   $u_0$ into spherical harmonics, it is easy to see that $u_0$ must be equal to zero. Indeed, considering  $u_0^+$  the restriction of $u_0$ on one side of the line, one can extend it as a harmonic function $\tilde u_0$ in  the whole $\R^2$ using a reflexion. This function still satisfies~\eqref{constnormen}. On the other hand it is a sum of harmonic polynomials whose gradients are orthogonal in $L^2$ (the so called spherical harmonic decomposition) which contradicts the estimate~\eqref{constnormen} (for more details see for instance~\cite[Theorem~15]{a17} for a similar argument in conical domains).
\end{proof}

Then we control $\omega_{\Sigma}(x,r)$  from  $\int_{B_r(x)}|\nabla u_\Sigma|^2$ by another compactness argument.

\begin{proposition}\label{gouter1}
Assume that $f\in L^p(\Om)$, $p>2$,
and let $C>0$ be the constant of Remark~\ref{remYeah} (depending only on $|\Omega|$, $\|f\|_p$, $p$). Then for every  $r_0\in (0,\diam(\Sigma)/2)$ there exists a $\rho \in (0,1)$ such that for any $x\in \Sigma$ we have
\[\omega_\Sigma(x,\rho r_0) \leq  \frac{2}{r_0}\int_{B_{r_0}(x)}|\nabla u_\Sigma|^2 \, dx+Cr_0^{\frac{2}{p'}-1}.\]
\end{proposition}

 \begin{proof}
If the claim is false, one can find an $r_0\in (0,\diam(\Sigma)/2)$ and two sequences $\rho_n\to 0$, $\{x_n\} \subset \Sigma$ such that
 \[\omega_\Sigma(x_n,\rho_n r_0) >  \frac{2}{r_0}\int_{B_{r_0}(x_n)}|\nabla u_\Sigma|^2 \, dx+Cr_0^{\frac{2}{p'}-1}.\]
 This implies in particular the existence of
a sequence of compact connected maximizers $\Sigma_n\subset\overline{\Om}$ for $\omega_\Sigma$  such that $\Sigma_n\Delta \Sigma \subset \overline{B}_{\rho_n r_0}(x_n)$ and
\begin{equation}\label{eq_contr_omsig1} \frac{1}{\rho_n r_0}\int_{B_{\rho_n r_0}(x_n)}|\nabla u_{\Sigma_n}|^2 \, dx >  \frac{2}{r_0}\int_{B_{r_0}(x_n)}|\nabla u_\Sigma|^2 \, dx+Cr_0^{\frac{2}{p'}-1}.
\end{equation}
From Remark~\ref{remYeah} applied with $\Sigma_n$ and $x_n$ instead of $\Sigma$ and $x_0$ respectively, we
get
\[
\frac{1}{r_0}\int_{B_ {r_0}(x_n)}|\nabla u_{\Sigma_n}|^2 \, dx  \geq
\frac{1}{\rho_n r_0}\int_{B_{\rho_n r_0}(x_n)}|\nabla u_{\Sigma_n}|^2 \, dx +C(\rho_n r_0)^{\frac{2}{p'}-1}-Cr_0^{\frac{2}{p'}-1},
\]
 which together with~\eqref{eq_contr_omsig1}
gives for an arbitrary $\varepsilon>0$ the estimate the estimate
\[
\frac{1}{r_0}\int_{B_ {r_0+\varepsilon}(x_0)}|\nabla u_{\Sigma_n}|^2 \, dx \geq
\frac{1}{r_0}\int_{B_ {r_0}(x_n)}|\nabla u_{\Sigma_n}|^2 \, dx  >  \frac{2}{r_0}\int_{B_{r_0}(x_n)}|\nabla u_\Sigma|^2 \, dx
+ C(\rho_n r_0)^{\frac{2}{p'}-1}
\]
for all sufficiently large $n$ (depending on $\varepsilon$).
Then passing to the limit (up to a subsequence) $x_n\to x_0$, $\Sigma_n \to \Sigma_0$ (which necessarily equals $\Sigma$) and applying \v{S}ver\'ak Theorem~\ref{th_Sverak1} we get
\[
\frac{1}{r_0}\int_{B_ {r_0+\varepsilon}(x_0)}|\nabla u_{\Sigma_n}|^2 \, dx \geq
\frac{2}{r_0}\int_{B_{r_0}(x_0)}|\nabla u_\Sigma|^2 \, dx,
\]
and then passing to a limit in $\varepsilon\to 0^+$ we arrive at
a contradiction.
\end{proof}

\subsection{Flat points are $C^1$-points}

A first consequence of the above compactness estimates is the following statement that we shall need later.

\begin{theorem}\label{C1reg} Assume that $f\in L^p(\Om)$, $p>2$, and let
$\Sigma$ be a minimizer of Problem~\ref{pb_compl_pen1}. There exist an $\varepsilon>0$ and an $a\in(0,1)$ such that, if  $x\in \Sigma$ is such that $B_r(x)\subset \Omega$ and
  \[\beta_\Sigma(x,r)\leq \varepsilon,\]
  then $\Sigma\cap B(x,ar)$ is a 
 $C^{1,\alpha}$ curve for some $\alpha\in (0,1)$. As a consequence, for any point $x\in \Sigma\cap \Omega$ which admits a line as blow-up limit, there exists $r>0$ such that  $\Sigma\cap B_r(x)$ is a 
 $C^{1,\alpha}$ curve.
\end{theorem}
\begin{proof} From  Proposition~\ref{gouter2} and Proposition~\ref{gouter1} we deduce  that, provided $\varepsilon$ is small enough, one can find $r_0$ such that both $\beta_\Sigma(x,r_0)$ and $\omega_{\Sigma}(x,r_0)$ are small enough to apply Corollary~\ref{TheCor}.
\end{proof}

For some technical reasons we shall also need the following more precise version of the above statement, which is just  a rephrasing of  some of the   previous results,
and that will be needed only in the proof of Proposition~\ref{chordarc}.


\begin{lemma} \label{reg}
Assume that $f\in L^p(\Om)$, $p>2$ and
let  $\Sigma$ be a minimizer of Problem~\ref{pb_compl_pen1}. There exist the numbers $\varepsilon_0 \in (0,1/100)$,  $\eta>0$ and  $r_0>0$ such that whenever $x\in \Sigma$ and $10r\leq \min(r_0,d(x,\partial \Omega ))$ are such that
\[
\beta_{\Sigma}(x,10r)\leq \varepsilon_0,
\]
then $B_{r}(x)\cap \Sigma$ is a $C^{1,\alpha}$ curve for some $\alpha\in (0,1)$
satisfying in addition
\begin{eqnarray}
\beta_{\Sigma}(x,t)&\leq& \varepsilon_0,  
\label{lemmetechnique1} \\
\int_{B_t(x)}|\nabla u_{\Sigma}|^2\,dx &\leq &  \left(\frac{t}{r}\right)^{1+\eta}\left( \int_{B_r(x)}|\nabla u_{\Sigma}|^2\,dx+Cr^{\frac{2}{p'}} \right) 
\label{lemmetechnique2}
\end{eqnarray}
for all $t<r$ and
for some $C>0$ depending on $|\Omega|$, $\|f\|_p$, $p$. Moreover,  for any $t<r$
there exists an $s\in [t/2,t]$ such that $\Card \Sigma\cap \partial B_{s}(x)=2$ with the points lying on both sides.
\end{lemma}
\begin{proof} The first part of the claim,~\eqref{lemmetechnique1} and the conclusion about the existence of $s\in [t/2,t]$ such that $\Card \Sigma\cap \partial B_{s}(x)=2$, are  directly coming from the proof of Theorem~\ref{C1reg} (and from the proofs of previous propositions used for the proof of the later). But then this allows us to apply the monotonicity Lemma~\ref{lm_compl_monot1}, which says that
\begin{eqnarray*}
\int_{B_t(x)}|\nabla u_{\Sigma}|^2\,dx &\leq &  \left(\frac{t}{r}\right)^{\alpha}\left( \int_{B_r(x)}|\nabla u_{\Sigma}|^2\,dx+Cr^{\frac{2}{p'}}\right)  \quad \text{for all } t <r
\end{eqnarray*}
with $\alpha=2\pi/\gamma_{\Sigma}(x,0,r)$. Indeed, we may assume that
$\frac{2}{p'}-\frac{2\pi}{\gamma_{\Sigma}(x,0,r)}>0$  provided that $\varepsilon_0$ is small enough (with respect to $p$). Thus~\eqref{lemmetechnique2} follows by
 setting $\eta:= 1-\frac{2\pi}{\gamma_{\Sigma}(x,0,r)}$.

\end{proof}

\begin{remark}Though it will not be used, under the assumptions of Lemma~\ref{reg} we also have by Proposition~\ref{prop_compl_decay2} that
$$\frac{1}{t}\int_{B_t(x)}|\nabla u_{\Sigma}|^2\; dx\leq \omega_\Sigma(x,t) \leq C\left(\frac{t}{r}\right)^{\frac{1}{4}}+Ct^{\delta} \quad \text{ for all } t <r,$$
but in the sequel we really need the slightly different version~\eqref{lemmetechnique2}.
\end{remark}


\subsection{Chord-arc estimate and connectedness of blow-up limits}

Let $d_{\Sigma}(x,y)$ be the geodesic distance in $\Sigma$. We would like to prove the following.

\begin{proposition} \label{chordarc}
Assume that $f\in L^p(\Om)$, $p>2$.
Let $\Omega$ be a $C^1$ domain and let $\Sigma\subset \overline{\Omega}$ be a minimizer of Problem~\ref{pb_compl_pen1}. Then there exists    $C>0$ (depending on $\Sigma$) such that
\begin{eqnarray}
d_{\Sigma}(x,y)\leq C |x-y| \quad \text{ for all } x,y \in \Sigma .\label{chordA}
\end{eqnarray}
\end{proposition}

\begin{remark}
As it will be clear from the proof, without any condition on the boundary
of $\Omega$ one would have that~\eqref{chordA} holds with $C$ depending on the distance between $\{x,y\}$ and $\partial \Omega$.
\end{remark}

The proof of Proposition~\ref{chordarc} uses the  uniform rectifiability of the minimizer $\Sigma$, and will be used later to prove that blow-up limits of minimizers are connected sets.  It  also needs some well known  facts about uniformly rectifiable sets. Here is the exact statement that we shall use.

\begin{proposition}\label{ahlfors} For any $C>0$ and $\varepsilon>0$ there exists   $c_0\in (0,1)$ depending only on $C$ and $\varepsilon$, such that the following holds. Let $K$ be a compact connected set which is Ahlfors-regular with constants  $C$, and let $\Gamma\subset K$ be a curve. Then for any $x \in \Gamma$ and $r < \min(1, \diam(\Gamma))$, there exists a ball $B_{s}(y) \subset B_{r}(x)$ such that $y \in \Gamma$, $s\geq c_0 r$ and
\[\beta_{K}(y,s)\leq \varepsilon.\]
\end{proposition}

\begin{proof}  The statement is classical in uniform rectifiability theory, but here we stress that the concluding ball $B_{s}(y)$ where $\beta_{K}(y,s)$ is small is centered on $y\in \Gamma$, a given sub curve of $K$. The proof follows by the same classical argument but we prefer to give the full details since it may not be as standard as when one wants $y \in K$.

Since  $K$  is compact, connected, and Ahlfors-regular, then it is a uniformly rectifiable set, in other words contained in an Ahlfors-regular curve~\cite[Theorem~31.5]{d}. Consequently, the  set $K$ satisfies the so-called BWGL (Bilateral Weak Geometric Lemma~\cite[Definition~2.2, p.~32]{DavSemm93}), which means the following. For any $\varepsilon >0$ let us consider the bad set
\[\mathcal{B}_\varepsilon := \{(x,t) \in K\times (0,1]  \colon \beta_{K}(x,t)> \varepsilon \}.\]
Then   there exists some $C_\varepsilon>0$ such that
\begin{eqnarray}
  \int_{K\cap B_{t}(x)}  \int_{0}^t {\bf 1}_{\mathcal{B}_\varepsilon} (y,s) \frac{ds}{s}  d\H(y)  \leq C_\varepsilon t\quad \quad \text{for all $(x,t)\in K\times (0,1]$}. \label{BWGL}
\end{eqnarray}
 Now  let us consider the set
\[
\mathcal{B}'_\varepsilon := \{(x,t) \in \Gamma \times (0,1]  \colon \beta_{K}(x,t)> \varepsilon \} \subset \mathcal{B}_\varepsilon.
\]
Of course ${\bf 1}_{\mathcal{B}'_\varepsilon}\leq {\bf 1}_{\mathcal{B}_\varepsilon}$  thus we readily have from~\eqref{BWGL}
\[ 
\int_{\Gamma \cap B_{t}(x)}  \int_{0}^t {\bf 1}_{\mathcal{B}'_\varepsilon} (y,s) \frac{ds}{s}  d\H(y)  \leq C_\varepsilon t\quad \text{for all $(x,t)\in \Gamma \times (0,1]$}.
\]
Now let $r_1=\min(1,\diam(\Gamma))$ we argue by contradiction and assume, for some given $(x,t) \in \Gamma \times (0,r_1)$, that  $\beta_K(y,s)>\varepsilon$ for all  $s \in (ct, t/2)$, and for all $y \in \Gamma\cap B_{t/2}(x)$, where  $c>0$ has to be chosen later.  Since $\Gamma$ is a curve, $x\in \Gamma$ and  since $t/2<\diam(\Gamma)$ we have
\[\H(\Gamma \cap B_{t/2}(x))\geq \frac{t}{2}  . \]
Therefore,
\begin{eqnarray*}
 C_\varepsilon t &\geq& \int_{\Gamma \cap B_{t}(x)}  \int_{0}^t {\bf 1}_{\mathcal{B}'_\varepsilon} (y,s) \frac{ds}{s}  d\H(y)   \geq   \int_{\Gamma \cap B_{t/2}(x)}  \int_{ct}^{t/2}  \frac{ds}{s}  d\H(y)  \geq   \frac{t}{2} \log\frac{1}{2c},
\end{eqnarray*}
which provides a contradiction if we chose $c>0$ small enough compared to $C_{\varepsilon}$.
\end{proof}

We can now prove Proposition~\ref{chordarc}.

\begin{proof}[Proof of Proposition~\ref{chordarc}]  We argue by contradiction. We already know that $\Sigma$ is Ahlfors regular with upper constant $C_0$.
Let $r_0$, $\varepsilon_0$, $\delta>0$ be the constants of  Proposition~\ref{reg} while  $c_0$ is the constant given by Proposition~\ref{ahlfors} with the choice $\varepsilon:=\varepsilon_0$. We also denote by $r_0'$ the constant of Lemma~\ref{decayboundary}.

Assume now that for some $\{x, y\} \subset \Sigma$ with $R:=|x-y|$ small enough,  any geodesic curve $\Gamma\subset \Sigma$ connecting them satisfies
 \begin{eqnarray}
\H(\Gamma ) > \Lambda  C_0 R,\label{estimateee}
\end{eqnarray}
where $\Lambda>0$ is a large constant that will  be fixed during the proof. Namely, $\Lambda$ will first  be chosen large enough, and then $R$ will be taken sufficiently small, depending in particular on this large $\Lambda$. Let us now proceed to the proof.

First we notice that  $\Gamma\setminus B_{\Lambda R}(x)$  is not empty,  because otherwise we would have
$\Gamma\subset B_{\Lambda R}(x)$ which would imply $\H(\Gamma) \leq  C_0 \Lambda R$ by Ahlfors regularity of $\Sigma$, a contradiction with~\eqref{estimateee}. Since $\Gamma$ is a curve containing $x$, we deduce that there exists a $z \in \Gamma$ such that
\begin{eqnarray}
\Lambda R/2 \leq |x-z|\leq  \Lambda R. \label{crainte1}
\end{eqnarray}

If $R$ is small enough we have that $\Lambda R/2 < 1$ and  the first inequality in~\eqref{crainte1} implies $\Lambda R/2 < \diam(\Gamma)$, which allows us to apply Proposition~\ref{ahlfors}  to find a ball $B_{10s_0}(z')\subset B_{\Lambda R/2}(z)$ with $s_0$ satisfying  $c_0\Lambda R/20\leq s_0 \leq \Lambda R/20$ and
\[
\beta_{\Sigma}(z',10s_0)\leq \varepsilon_0.
\]

Now  we distinguish two cases.

{\sc Case~A}: $B_{10s_0}(z')\subset \Omega$.  Then, assuming  $\Lambda R <r_0$, we can apply Proposition~\ref{reg} with $z'$ and $s_0$ in place of $x$ and $r$ respectively.  This implies
\[
\beta_{\Sigma}(z',t)\leq \varepsilon_0 \quad \text{ for all } t\leq s_0
\]
and
\begin{eqnarray}
 \int_{B_{t}(z')}|\nabla u_{ \Sigma}|^2 dx &\leq&  \left(\frac{t}{s_0}\right)^{1+\delta}\left( \int_{B_{s_0}(z')}|\nabla u_{\Sigma}|^2 dx+Cs_0^{\frac{2}{p'}}\right) \quad   \text{ for all } t\leq s_0 \label{decayPA}
\end{eqnarray}
with $C>0$ independent of $t$ and $s_0$. In the sequel we will still denote by $C$ a constant that may change from line to line. Then, we consider  $t$ satisfying $2R\leq t\leq 4R$ and  such that $\sharp \Sigma \cap \partial B_t(z')=2$ with the points lying on both sides (this is possible by the last conclusion of Proposition~\ref{reg}). Then we take the competitor
\[
\Sigma':=(\Sigma \setminus B_{t}(z'))  \cup [x,y],
\]
 it is easy to verify that $\Sigma'$ is connected. In addition, one has
\begin{eqnarray}
\H(\Sigma)-\H(\Sigma') \geq \H(\Sigma\cap B_{t}(z'))- |x-y|\geq t-R\geq R. \label{8R}
\end{eqnarray}

Now let  $\tilde \Sigma=\Sigma \setminus B_{t}(z')$, so that $\tilde \Sigma \triangle \Sigma \subset \overline{B}_{s/2}(z')$ and $\tilde \Sigma\subset \Sigma'$. Since  $ H^1_0(\Omega\setminus \Sigma')\subset H^1_0(\Omega\setminus \tilde \Sigma)$ and since $u_{\tilde \Sigma}$ minimizes $E$ over $H^1_0(\Omega\setminus \tilde \Sigma)$, we have $E(u_{\tilde{\Sigma}})\leq E(u_{\Sigma'})$ which in turn says that
$$\int_{\Omega}|\nabla u_{\Sigma'}|^2 dx \leq \int_{\Omega}|\nabla u_{\tilde \Sigma}|^2 dx.$$
Therefore, using that $\Sigma$ is a compliance minimizer and that $\Sigma'$ is a competitor, we can write
\begin{eqnarray}
\H(\Sigma) &\leq&  \H(\Sigma')+ \int_{\Omega} |\nabla u_{\Sigma'}|^2 dx -\int_{\Omega} |\nabla u_\Sigma |^2 dx \notag  \\
&\leq & \H(\Sigma')+ \int_{\Omega} |\nabla u_{\tilde \Sigma}|^2 dx -\int_{\Omega} |\nabla u_\Sigma |^2 dx  \notag \\
&\leq &\H(\Sigma') + C\int_{B_{2t}(z')}|\nabla u_{ \Sigma}|^2 dx + Ct^{\frac{2}{p'}}, \text{ by Lemmas~\ref{lm_compl_loc1} and ~\ref{lm_compl_estabove1a}},   \label{8R2}
\end{eqnarray}
with $C>0$ independent of $t$.

Plugging~\eqref{8R} into~\eqref{8R2} and using~\eqref{decayPA} we obtain
\begin{eqnarray}
R \leq \H(\Sigma) - \H(\Sigma')&\leq&  C \left(\frac{t}{s_0}\right)^{1+\delta}\left( \int_{B_{s_0}(z')}|\nabla u_{\Sigma}|^2 dx+Cs_0^{\frac{2}{p'}}\right) +Ct^{\frac{2}{p'}} \notag \\
&\leq&  C \left(\frac{1}{\Lambda}\right)^{1+\delta}\left( \int_{B_{s_0}(z')}|\nabla u_{\Sigma}|^2 dx+Cs_0^{\frac{2}{p'}}\right) +CR^{\frac{2}{p'}} . \label{obtained}
\end{eqnarray}
Next we apply Remark~\ref{remYeah} (monotonicity with $\gamma=2\pi$) to get the following estimate, denoting also by $d_0=\diam(\Sigma)/2$,
\begin{eqnarray}
\int_{B_{s_0}(z')}|\nabla u_{\Sigma}|^2 dx&\leq& s_0\left(\frac{1}{d_0}\int_{B_{d_0}(z')}|\nabla u_{\Sigma}|^2 dx +Cd_0^{\frac{2}{p'}-1}\right)\\
&\leq & s_0\left(\frac{1}{d_0}\int_{\Omega}|\nabla u_{\Sigma}|^2 dx +Cd_0^{\frac{2}{p'}-1}\right)=:Cs_0.
\end{eqnarray}
Returning to~\eqref{obtained} and using that $s_0\leq \Lambda R/20$ we have now obtained
$$R \leq  C \left(\frac{1}{\Lambda}\right)^{1+\delta}\left( s_0 +s_0^{\frac{2}{p'}}\right) +CR^{\frac{2}{p'}}\leq  C \left(\frac{1}{\Lambda}\right)^{1+\delta}\left( \Lambda R +(\Lambda R)^{\frac{2}{p'}}\right) +CR^{\frac{2}{p'}}.$$
By choosing now $\Lambda$ sufficiently large so that $C\frac{1}{\Lambda^{\delta}}\leq \frac{1}{2}$ we get
$$R/2 \leq  C (\Lambda R)^{\frac{2}{p'}}.$$
Finally, recalling that $2/p'>1$ we obtain a contraction for $R$  sufficiently small.

{\sc Case B}: there exists  a  $z'' \in B_{10s_0}(z') \cap \partial  \Omega$. Then $B_{s_0}(z')\subset B_{20s_0}(z'')$  and assuming that $\Lambda R\leq r_0'$ we can apply Lemma~\ref{decayboundary} which implies
\begin{eqnarray}
 \int_{B_{s_0}(z')} |\nabla u_{\Sigma}|^2\, dx \leq  \int_{B_{20s_0}(z'')} |\nabla u_{\Sigma}|^2\, dx  \leq  C\left(\frac{R}{r_0'}\right)^{1+\nu} + CR^{\frac{2}{p'}}, \label{decaydecaydecay}
 \end{eqnarray}
where $C>0$ is independent of $R$ and $r_0':=r(\partial\Omega,p)$ in the notation of Lemma~\ref{decayboundary}.
We can also assume that $\varepsilon_0 \leq 10^{-5}/(20+C_0)$  and that $R$ is small enough so that the assumptions~\eqref{eq_bet104} and~\eqref{eq_om500C} of Proposition~\ref{TheProp} are satisfied, thanks to~\eqref{decaydecaydecay}. This allows us to find $s\in [s_0/4,s_0/2]$ such that $\sharp \Sigma \cap \partial B_s(z')=2$ with the points lying on both sides. We then continue as in Case~A, namely, we take the competitor
\[
\Sigma':=(\Sigma \setminus B_{s}(z'))  \cup \Gamma_{x,y},\]
where $\Gamma_{x,y}$ is the geodesic curve in $\overline{\Omega}$ that connects $x$ and $y$. Since $\Omega$ is a $C^1$ domain, we have
$$\mathcal{H}^1(\Gamma_{x,y})\leq R+o(R)$$
as $R\to 0^+$.
Therefore, if $R$ is small enough,
 \begin{eqnarray}
\H(\Sigma)-\H(\Sigma') \geq \H(\Sigma\cap B_{s}(z'))- R-o(R)\geq \Lambda R/40-2R\geq R. \label{8RP}
\end{eqnarray}

Now arguing exactly as for~\eqref{8R2} with $s$ instead of $t$ we get the estimate
\begin{eqnarray}
\H(\Sigma) &\leq&  \H(\Sigma')+ \int_{\Omega} |\nabla u_{\Sigma'}|^2 dx -\int_{\Omega} |\nabla u_\Sigma |^2 dx \notag  \\
&\leq & \H(\Sigma')+ \int_{\Omega} |\nabla u_{ \Sigma\setminus B_{s}(z')}|^2 dx -\int_{\Omega} |\nabla u_\Sigma |^2 dx  \notag \\
&\leq &\H(\Sigma') + C\int_{B_{2s}(z')}|\nabla u_{ \Sigma}|^2 dx + Cs^{\frac{2}{p'}}, \text{ by Lemmas~\ref{lm_compl_loc1} and ~\ref{lm_compl_estabove1a}}.  \label{8R2P}
\end{eqnarray}
Then using~\eqref{decaydecaydecay},~\eqref{8RP} and~\eqref{8R2P} we get

$$R\leq C\int_{B_{s_0}(z'')}|\nabla u_{ \Sigma}|^2 dx + Cs_0^{\frac{2}{p'}}\leq  C\left(\frac{R}{r_0'}\right)^{1+\nu} + CR^{\frac{2}{p'}},$$
which yields again a contradiction for $R$ small enough.

 This means that for some $R_0>0$, once $x,y\in  \Sigma$
satisfy $|x-y|\leq  R_0$, then the geodesic curve $\Gamma\subset \Sigma$ connecting them
does not satisfy~\eqref{estimateee}, or, in other words,
that~\eqref{chordA} holds for any such couple of points.
But if $|x-y|\geq R_0$ we obviously have that
\[d_{\Sigma}(x,y)\leq \H(\Sigma)\leq \frac{\H(\Sigma)}{R_0}|x-y|,\]
which concludes the proof.

 \end{proof}


One of the main purposes of proving the chord-arc estimate is to obtain the following consequence.

\begin{proposition}\label{connected} The set $\Sigma_0$ given by Proposition~\ref{blowuplimit} is an unbounded  arcwise connected set, and any connected component of $\R^2\setminus \Sigma_0$ is simply connected.
\end{proposition}

\begin{proof} Let $\Sigma$ be
a minimizer for  Problem~\ref{pb_compl_pen1}.
We first prove that $\Sigma_0$ is connected.
Let $x$ and $y$ be two arbitrary points in $\Sigma_0$, and set
$R:=|x|+2|x-y|$. Since $\Sigma_n \to \Sigma$ for the Hausdorff distance in $B_R(0)$, we can find two sequences  $x_n\to x$ and $y_n \to y$ satisfying $\{x_n,y_n\} \subset \Sigma_n$ for all $n\in \N$.
Let $\Gamma_n\subset \Sigma_n$ be a geodesic curve connecting $x_n$ to $y_n$ in $\Sigma_n$.
By Proposition~\ref{chordarc} there exists a constant $C>0$ (depending only on $\Sigma$) such that
\[
\H(\Gamma_n)\leq C |x-y|,
\]
and hence $\Gamma_n\subset \overline{B}_{CR}(0)$. We can therefore extract a subsequence such that $\Gamma_n\to \Gamma$ in Hausdorff distance
 for some compact connected $\Gamma\subset \overline{B}_{CR}(0)$.
 But then $\Gamma$ is a connected set contained in $\Sigma_0$ which necessarily contains $x$ and $y$.  Since $x$ and $y$ were arbitrary, we have that $\Sigma_0$ is a connected set.

It is also quite clear that $\Sigma_0$ must be  unbounded because  $\diam(\Sigma_n)\to +\infty $ and $\Sigma_n$ converges to $\Sigma_0$ in the sense of Kuratowski.  It remains to prove that any connected component of $\R^2\setminus \Sigma_0$ is simply connected.
 Let $U$ be a connected component of $\R^2\setminus \Sigma_0$, and let $\Gamma \subset U$ be a simple closed curve. Without loss of generality, we may assume that this curve is polygonal. Since $\Gamma$ is a Jordan curve,  $\R^2\setminus \Gamma$ has two connected components, $A^-$ and $A^+$.  Let $A^-$ be the bounded one.
 Since $\Sigma_0$ is unbounded, we must have $\Sigma_0\cap A^+ \neq  \emptyset$. And since  $\Sigma_0$ is connected we must have $\Sigma_0 \subset A^+$, because otherwise we could write  $\Sigma_0= (\Sigma_0\cap A^+)\cup (\Sigma_0\cap A^-)$,
 a union of two relatively closed disjoint sets, which would give a contradiction. But now it is clear that,  $\Gamma$ being polygonal, it can easily be retracted to a point in $A^-$, proving the claim.
\end{proof}


\section{Dual formulation and classification of blow-up limits }
\label{dual}

Blow-up limits are not easy to handle in the original  formulation of the problem because it is of min-max type, so that the functional is not easily localizable.  To overcome this difficulty  we use a dual formulation of the problem to transform the min-max into a min-min, and then in many aspects we follow the strategy pursued by Bonnet~\cite{b} for Mumford-Shah minimizers.

More precisely, we prove, using a duality argument, that  the blow-up limits converges to the following notion of global minimizers. In analogy with the Mumford-Shah problem we define the following notion of a \emph{compliance global minimizer}.

\begin{definition} \label{global} A (compliance) global minimizer in $\R^2$ is a pair $(u,\Sigma)$ with $\Sigma\subset \R^2$ closed, and  $u \in H^1_{0,\Sigma,loc}(\R^2)$   such that for every ball $B\subset \R^2$ we have
\[
\frac{1}{2}\int_{B}|\nabla u|^2 \,dx + \H(\Sigma \cap B)\leq \frac{1}{2}\int_{B}|\sigma|^2 \,dx + \H(\Sigma' \cap B),
\]
 for all  vector fields  $\sigma \in L^2_{loc}(\R^2;\R^2)$ satisfying $\sigma=\nabla u$ in $\R^2\setminus B$ and ${\rm div } \sigma =0$ in $\R^2\setminus \Sigma'$,  where $\Sigma'$ is any closed set  $\Sigma'\subset \R^2$ satisfying $\Sigma \cap B = \Sigma'\cap B$ and $\Sigma' \cap\overline{B}$ is connected.
\end{definition}

Assuming furthermore that $\Sigma$ is connected (which actually occurs in our case), our notion of global minimizer $(u,\Sigma)$ turns out to be exactly the dual one of Mumford-Shah global minimizers, thus $\Sigma$ must be of the same type of the ones from the Mumford-Shah functional.

\subsection{Dual formulation}

We introduce the following dual formulation of our problem.
\begin{proposition}\label{duality1} Let $f\in L^2(\Om)$. Then
Problem~\ref{pb_compl_pen1} is equivalent  to the minimization problem
\begin{eqnarray}
\min_{(\sigma, \Sigma)\in \mathcal{B}} \frac{1}{2}\int_{\Omega }|\sigma|^2 \, dx +  \H(\Sigma) \label{dualcompliance}
\end{eqnarray}
where
\[
\mathcal{B}:=\{(\sigma,\Sigma) \; :\; \Sigma\subset \overline{\Omega} \text{ compact connected and }  \sigma \in L^2(\Omega;\R^2) \text{ satisfies } {\rm div}\, \sigma = f \text{ in  } \mathcal{D}'(\Omega \setminus \Sigma) \}
\]
in the sense the minimum value of the latter is equal to that of the original Problem~\ref{pb_compl_pen1}, and
once $(\sigma,\Sigma)\in \mathcal{B}$ is a minimizer for~\eqref{dualcompliance}, then $\Sigma$ solves the original Problem~\ref{pb_compl_pen1}.

In addition, for a given compact connected $\Sigma\subset \overline{\Om}$ fixed, the choice $\sigma := \nabla u_{\Sigma}$  solves
\[\min \left\{ \frac{1}{2}\int_{\Omega}|\sigma|^2 dx \colon {\rm div}\, \sigma = f \text{ in  } \mathcal{D}'(\Omega \setminus \Sigma) \right\}.\]
\end{proposition}

\begin{proof}  For any  given $u\in H^1_0(\Omega \setminus \Sigma)$ and $\sigma \in L^2(\Omega; \R^2)$, we can write
\begin{eqnarray*}
\int_{\Omega} |\nabla u|^2\, dx &=& \int_{\Omega} |\nabla u-\sigma|^2\, dx+2 \int_{\Omega } \nabla u \cdot \sigma \, dx-\int_{\Omega} | \sigma|^2\, dx \\
&\geq & 2 \int_{\Omega }  \nabla u \cdot  \sigma \rangle\, dx-\int_{\Omega} |  \sigma |^2\, dx,
\end{eqnarray*}
and since we have equality for $\sigma=\nabla u$, we then have just proven  the standard Legendre transform identity
\begin{eqnarray}
\int_{\Omega} |\nabla u|^2\, dx =\max_{\sigma \in L^2(\Om;\R^2)} \left( 2 \int_{\Omega } \nabla u \cdot \sigma \, dx-\int_{\Omega} | \sigma |^2\, dx\right). \label{legendre}
\end{eqnarray}
Therefore, 
\begin{align*}
-\C(\Sigma) & = \min_{u\in H^1_{0}(\Omega \setminus \Sigma)} \left( \max_{\sigma \in L^2(\Om;\R^2)} \left(  \int_{\Omega } \nabla u \cdot \sigma \, dx-\frac{1}{2}\int_{\Omega} | \sigma |^2\, dx\right) -\int_{\Omega} u f \, dx\right) \\
& = \min_{u\in H^1_{0}(\Omega \setminus \Sigma)} \max_{\sigma \in L^2(\Om;\R^2)} F(u,\sigma),\quad\text{where}\\
F(u,\sigma)& :=   \int_{\Omega } \nabla u \cdot \sigma \, dx-\frac{1}{2}\int_{\Omega} | \sigma |^2\, dx
-\int_{\Omega} u f \, dx.
\end{align*}
Now we want to exchange $\min$ and $\max$ in the above formula. 
Although it follows from quite standard results of convex analysis,
in this particular setting the proof is elementary: it is clear that
we always have 
\begin{eqnarray*}
 \min_{u\in H^1_{0}(\Omega \setminus \Sigma)}  \max_{\sigma \in L^2(\Om;\R^2)}F(u,\sigma) \geq \sup_{\sigma \in L^2(\Om;\R^2)}  \inf_{u\in H^1_{0}(\Omega \setminus \Sigma)}  F(u,\sigma)
= \sup_{\sigma \in D} \left(-\frac{1}{2}\int_{\Omega}|\sigma|^2 \, dx\right),
\end{eqnarray*}
where $D$ stands for the space of $\sigma \in L^2(\Omega;\R^2)$ such that
\begin{eqnarray}
\int_{\Omega} \sigma \cdot \nabla \varphi \,dx= \int_{\Omega}f \varphi \, dx \quad \text{for all } \varphi \in H^1_0(\Omega\setminus \Sigma),\label{divu=f}
\end{eqnarray}
otherwise the infimum in $u$ would be $-\infty$.

To verify  the reverse inequality, we consider the minimizer $u_\Sigma$ for the problem
\[\min_{u\in H^1_0(\Omega \setminus \Sigma)} \frac{1}{2} \int_{\Omega}|\nabla u|^2 \, dx - \int_{\Omega} uf \, dx .\]
We observe  that the optimality conditions on $u_\Sigma$ yields $\nabla u_\Sigma \in D$. Thus, recalling that the maximum in~\eqref{legendre} is attained at $\sigma:= \nabla u$ for $u\in H_0^1(\Om\setminus\Sigma)$, we get
\begin{align*}
 \min_{u\in H^1_{0}(\Omega \setminus \Sigma)}  \max_{\sigma \in L^2(\Om;\R^2)}F(u,\sigma) &=\min_{u\in H^1_0(\Omega \setminus \Sigma)} \left(\frac{1}{2} \int_{\Omega}|\nabla u|^2 \, dx - \int_{\Omega} uf \, dx \right)\\
 & =-\frac{1}{2}\int_{\Omega}|\nabla u_\Sigma|^2 \, dx  \leq  \sup_{\sigma \in D} \left(- \frac 1 2 \int_{\Omega}|\sigma|^2 \, dx\right).
\end{align*}
In conclusion, we have proven that
\[-\frac 1 2 \int_{\Omega}|\nabla u_\Sigma|^2 \, dx = \min_{u\in H^1_0(\Omega \setminus \Sigma)} \left(\frac{1}{2} \int_{\Omega}|\nabla u|^2 \, dx - \int_{\Omega} uf \, dx \right)= \max_{\sigma \in D} \left(- \frac 1 2 \int_{\Omega}|\sigma|^2 \, dx\right),\]
and $\sigma = \nabla u_\Sigma$ is a maximizer. It follows that
\[\min_{\Sigma} \frac 1 2 \int_{\Omega} |\nabla u_\Sigma|^2 +\H(\Sigma) =\min_{\sigma \in \mathcal{D}, \Sigma\in \mathcal{K}(\Omega)} \frac 1 2 \int_{\Omega} |\sigma |^2 +\H(\Sigma),\]
and  the  proposition follows from the uniqueness of $u_\Sigma$ and the minimizer $\sigma$.
\end{proof}

\begin{remark} The dual formulation yields another proof for the Ahlfors regularity of minimizers. Indeed, if $B:=B_{r}(x)$,
$x\in \Sigma$, recalling  that  $(\nabla u_\Sigma,\Sigma)$ is a minimizer for the Problem~\eqref{dualcompliance},  we can consider the competitor $(\sigma',\Sigma')$ defined by
\begin{align*}
\Sigma' &:=(\Sigma \setminus B) \cup (\partial B\cap\Om),\\
\sigma' &:=\left\{
\begin{array}{l}
\nabla u_\Sigma \text{ in } \Omega \setminus B, \\
\nabla u \text{ in } B\cap \Om, \text{ where } u \in H^1_{0}(B) \text{ solves } -\Delta u = f.
\end{array}
\right.
\end{align*}
Then
 \begin{equation}\label{eq_estlendual1}
 \frac 1 2 \int_{B\cap \Om}|\nabla u_\Sigma|^2 \, dx+\H(\Sigma \cap B) \leq 2\pi r + \frac 1 2 \int_{B\cap\Om}|\nabla u|^2 \, dx,
 \end{equation}
and it remains to estimate $\int_{B}|\nabla u|^2 \, dx$. For this we observe that for every $\varphi \in H^1_0(B\cap\Om)$ one has
\[
\int_{B\cap \Om}\nabla u \cdot \nabla \varphi \,dx = \int_{B\cap \Om} \varphi f \,dx,
\]
so that taking $\varphi := u$ and using the Poincar\'e inequality we obtain
\begin{equation*}
\int_{B\cap\Om}|\nabla u |^2 \,dx = \int_{B\cap\Om} u f \,dx \leq \|u\|_{L^2(B)}\|f\|_{L^2(B\cap\Om)}\leq Cr \|\nabla u\|_{L^2(B\cap\Om)}\|f\|_2,
\end{equation*}
and thus
\[
\int_{B\cap\Om}|\nabla u |^2 \,dx \leq  C r^2 \|f\|_2,
\]
where $C>0$ is universal. We conclude then from~\eqref{eq_estlendual1} the estimate
\begin{eqnarray}
\H(\Sigma\cap B)\leq \frac 1 2 \int_{B\cap\Om}|\nabla u_\Sigma|^2 \, dx+\H(\Sigma \cap B) \leq 2\pi r +  Cr^2 \|f\|_2^2, \label{Thebound}
\end{eqnarray}
which in particular implies Ahlfors regularity of $\Sigma$.
\end{remark}

\subsection{Minimization problem for the blow-up limit}

The aim of this section is to prove the following assertion.

\begin{proposition} \label{global1}The limit  $(u_0,\Sigma_0)$ of Proposition~\ref{blowuplimit} is a compliance global minimizer in the sense of Definition~\ref{global}. In addition, for any ball $B\subset \R^2$ one has that
\begin{eqnarray}
 \H(\Sigma_{n_k}\cap B)\to \H(\Sigma_0 \cap B), \label{limiteForte}
\end{eqnarray}
where $\Sigma_{n_k}$ is  as in Proposition~\ref{blowuplimit}.
\end{proposition}

\begin{proof}  Let    $\Sigma_{n_k}$, $u_{n_k}$, $\Sigma_0$ and $u_0$  be as in Proposition~\ref{blowuplimit}, and we further write $n$ instead
   of $n_k$ for brevity. We will use the dual formulation~\eqref{dualcompliance} which says that for all $n\in\N$, the pair $(\nabla u_n,\Sigma_n)$ solves
\begin{equation}\label{minprobN}
\begin{split}
\min_{(\sigma, \Sigma)\in \mathcal{B}_n} &\frac{1}{2}\int_{\Omega_n }|\sigma|^2 \, dx +  \H(\Sigma),
\quad\text{where}\\
\mathcal{B}_n:= &\left\{(\sigma,\Sigma) \colon  \Sigma\in\mathcal{K}(\Omega_n) \text{ 
and }  \sigma \in L^2(\Omega_n,\R^2)
\text{ satisfies } {\rm div}\, \sigma = f_n \text{ in  } \mathcal{D}'(\Omega_n \setminus \Sigma) \right\}.
\end{split}
\end{equation}

Now let  $B \subset \R^2$ be a ball.
We may assume without loss of generality that $B=B_1(0)$. All the balls in this proof will by default be centered at the origin.
Let the pair $(\sigma, \Sigma')$ be a competitor to $(u_0, \Sigma_0)$ in $B$
  in the sense of Definition~\ref{global}, i.e. satisfying $\Sigma'\setminus B =  \Sigma_0\setminus B$ and $\Sigma'\cap \overline{B}$ is connected.

Let $\varepsilon \in (0,1)$ be a small parameter. We can choose an $s \in (1,1+\varepsilon)$ such that
\[N:=\Card\Sigma_0 \cap \partial B_s<+\infty\]
(notice that $\Sigma_0$ has finite length due to Go\l ab's theorem  and the Ahlfors regularity of $\Sigma_n$ with same constant). Let $\delta >0$ be another small parameter, and let us define the set
\[S_\delta:= \{x\in \partial B_s \colon \dist(x,\Sigma_0)\leq \delta\}.\]
Then since $\Sigma_n$ converges to $\Sigma_0$ in the Kuratowski sense in $\R^2$, it follows that for $n$ large enough, $\Sigma_n \cap B_2$ stays inside  a $\delta$-neighborhood  of $\Sigma_0\cap B_2$, thus it is easily seen  that  the set
\begin{equation*}
\Sigma'_n :=( \Sigma_n \setminus B) \cup (\Sigma'\cap B_{s})\cup S_\delta
\end{equation*}
is arcwise connected.

\begin{figure}[htbp]
\begin{center}
\includegraphics*[width=0.8\textwidth]{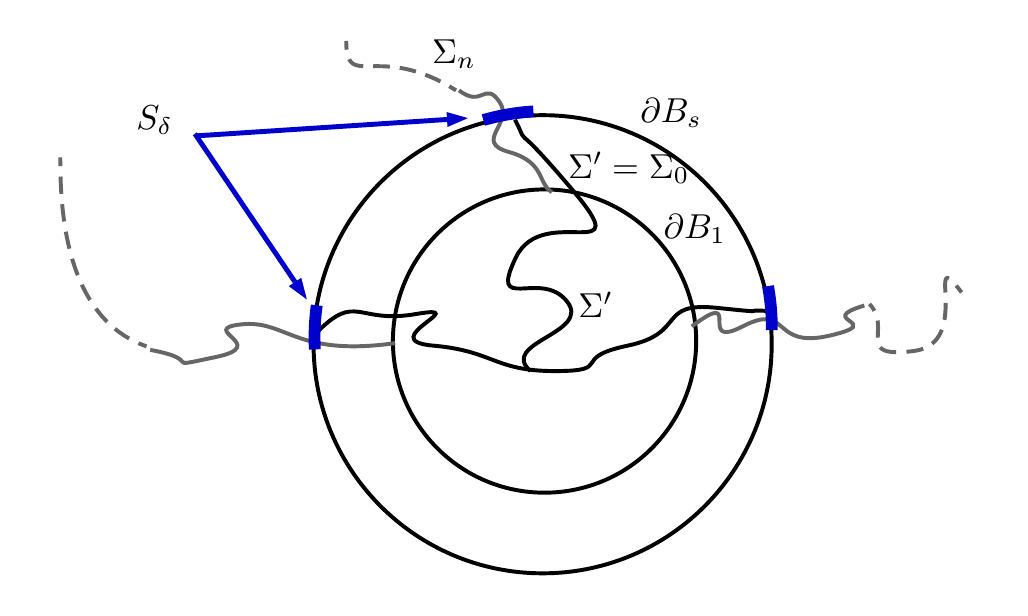}
\end{center}
\caption{Construction of the competitor $\Sigma'_n$ in Proposition~\ref{global1}.}\label{fig:Sigma'n}
\end{figure}%

Let now $\sigma \in L^2_{loc}(\R^2; \R^2)$ be as above, a competing vector field associated to $\Sigma'$
in the sense of Definition~\ref{global}, i.e.\ $\sigma=\nabla u_0$ in $\R^2 \setminus B$ and ${\rm div} \sigma =0$ in $\R^2\setminus \Sigma'$. We want to prove that
\begin{eqnarray}
\frac{1}{2}\int_{B}|\nabla u_0|^2 \,dx + \H(\Sigma_0 \cap B)\leq \frac{1}{2}\int_{B}|\sigma|^2 \,dx + \H(\Sigma' \cap B). \label{minglob}
\end{eqnarray}
For this purpose we modify $\sigma$ to make it admissible as a competitor for the minimization problem~\eqref{minprobN}, i.e.\ construct a vector field $\tilde\sigma_n$ such that $(\tilde\sigma_n, \Sigma_n')\in \mathcal{B}_n$.
We start by considering  a $2/(s-1)$-Lipschitz cut-off function $\varphi$ equal to $1$ on $B_1$ and to $0$ outside $B_{1+(s-1)/2}$, and we let
\[\sigma_n:=\varphi \sigma  + (1-\varphi)\nabla u_n.\]
Notice that
\[{\rm div } \sigma_n =( \sigma - \nabla u_n) \cdot \nabla \varphi +(1-\varphi) f_n  =:g_n\quad \text{ in } \mathcal{D}'(\Omega_n \setminus \Sigma'_n),\]
and that
\[g_n \to 0 \text{ strongly in } L^2_{loc}(\R^2),\]
because $\sigma =\nabla u_0$ a.e. in $\R^2\setminus  B$, $\nabla u_n \to \nabla u_0$ strongly in $L^2_{loc}(\R^2;\R^2)$
and $f_n \to 0$ strongly in $L^2_{loc}(\R^2)$ (recall that $u_0$ is harmonic by Proposition~\ref{blowuplimit}).
We now add a correction because we would like the divergence to be exactly equal to $f_n$. For this purpose we denote, for convenience, $r_0:=1+\varepsilon$ and let $v_n$ be the solution for the problem
\[\min_{v \in H^1_{0,\Sigma'_n}(B_{r_0}\setminus \Sigma'_n)} \int_{B_{r_0}} |\nabla v|^2 - 2\int_{B_{r_0}}(f_n- g_n) v.\]
In other words we are solving the problem  $-\Delta v_n= f_n-g_n$ in $B_{r_0}\setminus \Sigma'_n$, $v_n=0$ on $\Sigma'_n\cap B_{r_0}$ and
$\nabla v_n \cdot \nu = 0$ on $\partial B_{r_0}\setminus \Sigma_n'$. It follows that
\[- {\rm div} ({\bf 1}_{B_{r_0}}\nabla v_n)=(f_n- g_n){\bf 1}_{B_{r_0}} \text{ in }\mathcal{D}'(\R^2 \setminus \Sigma_n').\]

In addition one has the estimate
\begin{eqnarray}
\int_{B_{r_0}}|\nabla v_n|^2 \, dx = \int_{B_{r_0}} (f_n-g_n) v_n \leq \|f_n-g_n\|_{L^2(B_{r_0})} \|v_n\|_{L^2(B_{r_0})}. \label{ineqq}
\end{eqnarray}
Since $v_n$ vanishes on the substantially large piece of set $\Sigma_0\cap B_{1+\varepsilon}\setminus B_1$, by Poincar\'e inequality (Lemma~\ref{lm_compl_poincare2}) we infer that
\[\int_{B_{r_0}}|v_n|^2 \, dx \leq C(\varepsilon) \int_{B_{r_0}}|\nabla v_n|^2 \, dx, \]
thus~\eqref{ineqq} gives
\begin{eqnarray}
\int_{B_{r_0}} |\nabla v_n|^2\leq C(\varepsilon) \|f_n-g_n\|_{L^2(B_{r_0})}^2\to 0. \label{estimate5}
\end{eqnarray}
Consider now  the function
\[\tilde \sigma_n := \sigma_n - {\bf 1}_{B_{r_0}} \nabla v_n,\]
which, by construction, satisfies
\[{\rm div} \tilde \sigma_n = f_n \text{ in } \Omega_n \setminus \Sigma'_n,\]
i.e.\ $(\tilde\sigma_n, \Sigma_n')\in \mathcal{B}_n$.
Since $(\nabla u_n, \Sigma_n)$ is a minimizer for  the problem~\eqref{minprobN}, it follows that
\[\int_{\Omega_n}|\nabla u_n|^2 +  \H(\Sigma_n) \leq  \int_{\Omega_n}|\tilde \sigma_n|^2 +  \H(\Sigma'_n).\]
Since both $\tilde \sigma_n = \nabla u_n$ and $\Sigma'_n=\Sigma_n$ outside $B_{r_0}$, we get
\begin{eqnarray}
\int_{B_{r_0}}|\nabla u_n|^2 +  \H(\Sigma_n\cap B_{r_0}) &\leq&  \int_{B_{r_0}}|\tilde \sigma_n|^2 +  \H(\Sigma'_n\cap B_{r_0}). \notag \\
&\leq  &  \int_{B_{r_0}}|\tilde \sigma_n|^2 +  \H(\Sigma'\cap B_{r_0}) +\H(\Sigma_n\cap B_{r_0}\setminus B_1) + cN\delta,  \notag
\end{eqnarray}
where we have estimated  $\H(S_\delta)\leq cN\delta$, with $c>0$ being a universal constant. In other words,
\begin{eqnarray}
\int_{B_{r_0}}|\nabla u_n|^2 +  \H(\Sigma_n\cap B_{1})\leq
    \int_{B_{r_0}}|\tilde \sigma_n|^2 +  \H(\Sigma'\cap B_{r_0}) + cN\delta.  \label{apasser}
\end{eqnarray}

Notice that
\begin{eqnarray}
\H(\Sigma_0 \cap B) \leq \liminf_{n} \H(\Sigma_n \cap B), \label{Golabb}
\end{eqnarray}
thanks to Go\l ab's Theorem (one can just use the classical Go\l ab theorem applied to the sequence of compact connected sets $(\Sigma_n\cap B)\cup \partial B$ which converges to $(\Sigma\cap B) \cup \partial B$ for the Hausdorff distance by Proposition~\ref{blowuplimit} ($i$)).

Taking the liminf in $n$ of~\eqref{apasser}, using that $\tilde \sigma_n\to \sigma$ strongly in $L^2(B_{r_0};\R^2)$ we obtain that
\[
\int_{B_{r_0}}|\nabla u_0|^2 +  \H(\Sigma_0\cap B_{r_0})\leq
    \int_{B_{r_0}}|\sigma|^2 +  \H(\Sigma'\cap B_{r_0}) + cN\delta.
\]
We can now let $\delta \to 0^+$ and then $\varepsilon \to 0^+$ to get
\[
\int_{B_1 }|\nabla u_0|^2 +  \H(\Sigma_0\cap B_1)\leq
    \int_{B_1}|\sigma|^2 +  \H(\Sigma'\cap B_1) ,
\]
which concludes the proof of the fact that $(u_0, \Sigma_0)$ is a global minimizer.

It remains to prove that
\[\H(\Sigma_{n}\cap B)\to \H(\Sigma \cap B).\]
Inequality~\eqref{Golabb} gives the first half. To prove the reverse inequality, we pass to the limsup in $n$ in~\eqref{apasser}  with the special choice $\Sigma'=\Sigma_0$ and $\sigma=\nabla u_0$. Using $\nabla u_n \to \nabla u_0$ we obtain
\[ \limsup_n \H(\Sigma_n \cap B)\leq \H(\Sigma_{0}\cap B)+cN\delta\]
and letting $\delta \to 0^+$ we conclude the proof.
\end{proof}





\subsection{Another formulation and characterization of global minimizers}

In this section we  give another formulation of global minimizers and prove that it is equivalent to Mumford-Shah minimizers.

\begin{definition} \label{BonnetMin} A pair $(u,K)$ with $K\subset \R^2$ closed is called global Mumford-Shah minimizer, if $u\in L^2_{loc}(\R^2\setminus K)$ and $\nabla u \in L^2_{loc} (\R^2)$ satisfies
\[\int_{B}|\nabla u|^2 \, dx + \H(K\cap B) \leq \int_{B}|\nabla v|^2 \, dx + \H(L\cap B),\]
for every ball $B\subset \R^2$ and for all $(v,L)$ satisfying $v=u$  in $\R^2\setminus B$ and $K\setminus B=L\setminus B$, and such that $L\cap \overline{B}$ is connected.
\end{definition}

\begin{remark} The original definition of Bonnet~\cite{b} was slightly different, with a less restrictive topological condition of competitors, namely keeping $\{x, y\}\subset \R^2\setminus B$ separated by $L$ as soon as they were separated by $K$. We cannot use here exactly the same definition of competitors because we  need our competitors to be connected. However,  our definition is stronger, in the sense that any of our compliance global minimizers is automatically a global minimizer in the sense of Bonnet, which is enough to get the classification of global minimizers.
\end{remark}

\begin{definition} If $\Omega\subset \R^2$ is open and  $u$ is harmonic in $\Omega$, we call harmonic conjugate for $u$  in $\Omega$ a harmonic function $v$ in $\Omega$ such that $\nabla v =\nabla^T u$, where the notation $\nabla^Tu$ is used for the vector field $(-\partial_y u,\partial_x u)$.
\end{definition}

\begin{proposition}\label{bonnetM} Let $(u_0,\Sigma_0)$ be a global minimizer     in the sense of Definition~\ref{global}. Assume  moreover that every connected component of $\R^2\setminus \Sigma_0$ is simply connected.  Then, $u_0$ admits a harmonic conjugate   $u$  in $\R^2 \setminus \Sigma_0$, and $(u,\Sigma_0)$ is a Mumford-Shah global minimizer in the sense of Definition~\ref{BonnetMin}.
\end{proposition}

\begin{proof} Since each component of $\R^2\setminus \Sigma_0$ is simply connected,
$u_0$ admits a harmonic conjugate. Namely, for any $U\subset \R^2\setminus \Sigma_0$
connected component, since     ${\rm div } \nabla u_0 =0$ in $\mathcal{D}'(U)$,
by De Rham's theorem we get the existence of  $u$ harmonic in $U$ such that
$\nabla^T u_0=\nabla u$ in $\R^2 \setminus \Sigma_0$. In particular, $\|\nabla u_0\|_{L^2(K)}=\|\nabla u\|_{L^2(K)}$ for all $K\subset \R^2$, compact.

Now let $(v,\Sigma')$ be a competitor in some ball $B\subset \R^2$, i.e. $v=u$ in $B\setminus \R^2$, $\Sigma'\setminus B=\Sigma_0\setminus B$  and $\Sigma'\cap \overline{B}$ is connected. Then we define $\sigma:=\nabla^T v$. In particular, ${\rm div} \sigma = 0$ in $\R^2 \setminus \Sigma'$ and $\sigma = \nabla^T u = \nabla u_0$ in $\R^2\setminus B$. By the minimality of $(u_0,\Sigma_0)$ in the sense  of Definition~\ref{global}
we have
\[\int_{B}|\nabla u_0|^2 \, dx +\H(\Sigma_0\cap B)\leq\int_{B}|\sigma|^2 \, dx +\H(\Sigma' \cap B),\]
which implies
\[\int_{B}|\nabla u|^2 \, dx +\H(\Sigma_0\cap B)\leq\int_{B}|\nabla v|^2 \, dx +\H(\Sigma' \cap B),\]
thus $(u,\Sigma_0)$ is a global Mumford-Shah minimizer.
\end{proof}

As a direct consequence we get the following.

\begin{proposition} \label{corr} Any blow-up limit $(u_0,\Sigma_0)$   given by Proposition~\ref{blowuplimit}  is one of the following list, up to a translation, rotation,  dilatation, or adding a constant for $u_0$:
\begin{enumerate}
\item[(i)]  $\Sigma_0$ is a line  and $u_0$ is a constant on each side of it;
\item[(ii)]  $\Sigma_0$ is a propeller (three half lines meeting  in a single point by number of 3 at  120 degree angles) and $u_0$ is a constant in each of the sectors  formed by these half lines;
\item[(iii)]  $\Sigma_0$ is a half-line and $u_0$ is the ``Dirichlet-craktip'' function $(r,\theta)\mapsto \sqrt{r/2\pi}\cos(\theta/2)$ written in polar coordinates $(r, \theta)\in [0,+\infty)\times (-\pi,\pi)$.
\end{enumerate}
 \end{proposition}

\begin{proof}  Proposition~\ref{connected}, Proposition~\ref{global1} and Proposition~\ref{bonnetM} imply that $\Sigma_0$ is a global Mumford-Shah minimizer in the sense of Definition~\ref{BonnetMin}. To obtain that $\Sigma_0$ is the singular set of a global Mumford-Shah minimizer in the sense of Bonnet, it is enough to prove that it satisfyies the topological condition of Bonnet, namely, that each pair of points $x$ and $y$ in  $\Sigma_0\setminus B$ that are separated by $\Sigma_0$, are still separated by any competitor $L$. But this is clear due to the fact that $L\cap \overline{B}$ is connected.   It follows from Bonnet~\cite{b} that  it must be one of the list described above (notice that the blow-up limit is never empty since we blow-up a connected set).
\end{proof}


\subsection{Blow-up limit at a boundary point}

The purpose of this section   is to state results analogous to those of the previous sections,    in the particular  case of a boundary point, namely when  $x_n\to x_0 \in \partial \Omega$. Since the proofs are very similar, we shall only highlight the main differences. We shall put together the results  in the following unified statement. We first define the type of minimizing problem that arises at the limit.

\begin{definition}\label{definPropH} Let $H\subset \R^2$ be a closed half-plane. Then, a closed and arcwise connected set $\Sigma \subset H$ is a boundary compliance global minimizer in $H$ if it satisfies the following minimizing property: for every open ball $B\subset \R^2$ and every competitor $L\subset H$,  satisfying $L\setminus  B=\Sigma \setminus B$  and such that $L\cap \overline{B}$ is connected, one has
\begin{eqnarray}
\H(\Sigma\cap B) \leq \H(L\cap B). \label{minimal0}
\end{eqnarray}
\end{definition}

\begin{remark} A direct consequence of~\eqref{minimal0} is that $\Sigma$ has locally finite Hausdorff measure (take $(\Sigma \setminus B)  \cup\partial (B\cup H)$ as a competitor).
\end{remark}

\begin{remark} Our class of boundary compliance global minimizers looks similar to the one dimensional sliding minimal sets studied in~\cite{FANG,davAgain}. Notice however that our class is quite different (and, actually, is simpler) because we do not impose any sliding condition: in our situation, the competitors are not obliged to preserve the points of $\Sigma\cap \partial H$ on the boundary, they are free to be detached and move everywhere including inside the domain.   We indeed will arrive to a different classification than the one of~\cite[Lemma 4.3]{FANG}.
\end{remark}

\begin{theorem}  \label{global2} Assume that $f\in L^p(\Om)$, $p>2$, where
$\Omega$ is a $C^1$ domain. Let $(u_n,\Sigma_n)$ be the blow-up sequence in $\Omega_n$ defined in~\eqref{blowup1} and~\eqref{blowup2} with $x_n \to  x_0 \in \partial \Omega$  and assume that
$d(x_{n_k}, \partial \Omega)\leq  r_n$ for a subsequence (otherwise the blow-up analysis is exactly the same as the interior case). Then there exists a further subsequence (still indexed by $n_k$) such that $\Omega_{n_k}$ converges to a half-plane and
\begin{enumerate}
\item[(i)]   $\Sigma_{n_k}  \to\Sigma_0 $ in the Kuratowski sense in $\R^2$,
\item[(ii)] $u_{n_k}\to 0$ strongly in $H^1_{loc}(\R^2)$,
\item[(iii)] $\Sigma_0$ is a boundary compliance global minimizer in the sense of Definition~\ref{definPropH},
\item[(iv)]  $\mathcal{H}^1(\Sigma_{n_k} \cap B)\to \mathcal{H}^1(\Sigma_0 \cap B)$ for every ball $B$.
\end{enumerate}
\end{theorem}

\begin{proof} Let us index the  sequence by $n$ instead of $n_k$. The proof is the same as in the case of interior points. The only difference is that now $\Omega_{n}$ blows-up to a half-plane instead of  the whole $\R^2$. Indeed, the assumption $d(x_{n}, \partial \Omega)\leq r_n$ says that for all $n$, there exists some $z_n\in \partial \Omega$ with $d(z_n,x_n) \leq r_n$ and in particular $d(z_n,x_0)\to 0$.
 Since $\partial \Omega$ is $C^1$, it is clear that the blow-up of $\Omega$ at $z_n$ is the tangent line of $\partial \Omega$ at this point, and converges to the tangent line at point $x_0$. Comparing the blow-up at $z_n$ and the one at $x_n$ (for instance by noticing $B_{r_n}(z_n)\subset B_{2r_n}(x_n)\subset B_{4r_n}(z_n)$), we deduce that $\Omega_n$ blows up to a half-plane.

Let us assume without loss of generality that this half-plane is
\[H=\{(x,y) \in \R^2 \; : \; x\leq 0\}.\]
The proofs of Proposition~\ref{blowuplimit}  and Proposition~\ref{connected} can be followed line by line without any change even at the boundary, leading to the fact that $\Sigma_n  \to\Sigma_0 $ in the Kuratowski sense in $\R^2$,  $\Sigma_0$ is arcwise connected and $\R^2\setminus \Sigma_0$ is simply connected,  and finally $u_n$ converges strongly in $H^1_{loc}(\R^2)$ to a harmonic function $u_0$ in $H\setminus \Sigma_0$.

We prove now 
that $\nabla u_0=0$. Indeed, Lemma~\ref{decayboundary} implies
\begin{align*}
 \int_{B_{r}(z_n)} |\nabla u_{\Sigma}|^2\, dx  \leq  C_1\left(\frac{r}{r_0}\right)^{1+\nu} + C_2r^{\frac{2}{p'}},
\end{align*}
 where  
 $C_1$ and $C_2$ depend only on $|\Omega|$, $p$, $\|f\|_p$. This estimate, since $B_{r_n}(z_n)$ and $B_{r_n}(x_n)$ are comparable
 in the sense that $B_{r_n}(z_n)\subset B_{2r_n}(x_n)\subset B_{4r_n}(z_n)$, implies
\begin{align*}
\frac{1}{r_n} \int_{B_{r_n}(x_n)} |\nabla u_{\Sigma}|^2 dx \underset{n\to +\infty}{\longrightarrow} 0,
\end{align*}
and this yields $\nabla u_0=0$ as claimed.

Now comes the identification of the minimization problem arising at the limit. For this purpose we follow the proof of Proposition~\ref{global1}. Everything works the same way except that now the competitor $\Sigma'_n$ may not be admissible, since it may not be inside $\Omega$. We then modify the proof as follows: by the $C^1$ regularity of $\partial \Omega$ , we know that there exists some $\eta_n \to 0$ such that  $\partial \Omega \cap B_{r_n}(z_n)\subset C_{\eta_n}$, where $C_\eta$ is a cone of aperture $ \eta$, namely
\[C_\eta:=\{(x,y)\in \R^2 \colon  |x|\leq \eta |y|    \}.\]
For every $\eta$, let $\Phi_\eta:H\to H\setminus C_{\eta}$ be the following  $(1+\eta)$-Lisphchitz mapping
\[\Phi_\eta(x,y)=(x-\eta|y|,y).\]
Let now $\Sigma'$ be a competitor for $\Sigma_0$ in some ball $B\subset \R^2$. Then $\Sigma'':=\Phi_{\eta_n}(\Sigma_0)\subset \Omega_n$ and
\[\H(\Sigma'')\leq (1+\eta_n)\H(\Sigma').\]
The rest of the proof then follows the same way with  $\Sigma''$ in place of $\Sigma'$.
\end{proof}

\begin{proposition}\label{classification1} Let $\Sigma$ be a compliance global
minimizer in the sense of Definition~\ref{definPropH} with $\Sigma \cap \partial H \neq \emptyset$. Then
$\Sigma=\partial H$.
\end{proposition}

\begin{proof}
Let $x_0 \in \Sigma \cap \partial H$, $S$ be a connected component of
$\Sigma\cap \overline{B}_R$ containing $x_0$ and
$N:=S\cap \partial B_R$, where $R>0$ is arbitrary, the balls here and below being centered at $x_0$.
Then one has, in the notation of Lemma~\ref{lm_SteinerConvex0}, that
$S\in \St(N)$,
since otherwise for an arbitrary $K\subset \St(N)$ taking
$L:=(\Sigma\setminus S)\cup \overline{K}$, one gets, recalling that
$K\subset \overline{B}_R$ by Lemma~\ref{lm_SteinerConvex0} (see below), the estimate
\[
\H(L\cap \overline{B}_R) = \H((\Sigma\setminus S)\cap \overline{B}_R) + \H(K)
<\H((\Sigma\setminus S)\cap \overline{B}_R) + \H(S\cap \overline{B}_R)
= \H(\Sigma\cap \overline{B}_R)
\]
contradicting the assumption that $\Sigma$ be a compliance global minimizer.
But then by Lemma~\ref{lm_SteinerConvex0} the set
 $S$ may not contain $x_0$
unless $\{A_R^+, A_R^-\}\subset N$, where
$\{A_R^+, A_R^-\}:= \partial B_R\cap \partial H$.
In fact, if either of the points $A_R^\pm$ does not belong to $N$, then
$x_0$ does not belong to the closed convex envelope of $N$.
This shows $\{A_R^+, A_R^-\}\subset N$, and hence, since
$R>0$ is arbitrary, then $\partial H\subset \Sigma$.

To show that in fact $\Sigma=\partial H$, we assume the contrary,
and let now $R>0$ be such that
$\#(\Sigma\cap \partial B_R)$ is finite
(by coarea inequality, a.e.\ $R>0$ would suit for this purpose
because $\Sigma$ has locally
finite length)
and
a connected component $S$
of $\Sigma\cap B_R$ containing the line segment $l:=\partial H\cap B_R$
does not coincide with the latter.
Then this component belongs to $\St(N'\cup \bar l)$,
with $N'\subset \Sigma\cap \partial B_R$
and hence by~\cite[Theorem~7.4]{PaoSte13-steiner}
is a finite embedded graph consisting of line segments
with exactly one endpoint (we denote it by $A$) over
$\bar l$. But this cannot happen
since $S\in\St(N)$, where $N:=S\cap \partial B_R$; in fact, then there is a
line segment $(AB)\in S$ such that the angle between $(AB)$ and $\bar l$ is less than 120 degrees, which is impossible for Steiner sets connecting a finite number of points. This contradiction
shows that the connected component of $\Sigma \cap B_R$ containing $l$ must be $l$ itself.

Finally it is easy to see that there exist no other connected components of
$\Sigma \cap B_R(x_0)$ different from  $l \cap B_R(x_0)$.  Indeed, assuming the contrary, and letting $x$ be a point belonging to the other component, using that $\Sigma$ is arcwise connected, one can find a curve in $\Sigma$ connecting $x$ to $x_0$. This curve has to branch on $\partial H$ at some point $y_0$. Then reasoning as before with a ball centered at $y_0$ instead of $x_0$, we get a contradiction  that concludes the proof.
\end{proof}

\begin{lemma}\label{lm_SteinerConvex0}
For every $N\subset \R^n$ denote by $\St(N)$ the set of minimizers of the Steiner problem
\[
\min \{ \H(K) \; :  K\subset \R^n,\, K\cup N \text{ connected}\}.
\]
Then for every $K\in \St(N)$ of finite length
one has that $\overline{K}$ belongs
to the closed convex envelope of $N$.
\end{lemma}

\begin{proof}
If
$\overline{K}$ does not belong
to the closed convex envelope of $N$, then its projection to the latter still connects $N$ and has strictly lower length, contradicting the optimality of $K$.
\end{proof}

\begin{remark}
A standard strategy to prove Proposition~\ref{classification1} would be to follow the usual
classification of minimal cones. We used here a different and more elementary approach
based on Lemma~\ref{lm_SteinerConvex0}, but for the sake of completeness  we describe the standard one:
it is not difficult to see that if $\Sigma$ is a cone centered at some point $x_0\in \partial H$,
then $\Sigma$ can only be a half-line or a line (the latter being true only if $\Sigma=\partial H$).
It is not difficult to exclude the half-line by a competitor which would cut the ``corner'' near $x_0$,
showing that the only minimal cone is $\partial H$.  Next, for an arbitrary
compliance global minimizer $\Sigma$ and $x_0\in \partial H\cap \Sigma$,
comparing as usual $\Sigma$ with a cone over its trace on the boundary of a sphere,
one shows that the density   $\H(\Sigma\cap B_r)/r$ is monotone in $r$.
Therefore the limiting densities at $r\to 0^+$ and $r\to +\infty$ do exists.
Considering the blow-in and blow-up limits,  due to the strong convergence of $\H$ along those
sequences one can  see that those limit densities can only be equal to two, because they are densities of a minimal cone.
But then by monotonicity, the density of $\Sigma$ itself is constant  and equal to two so that $\Sigma$ must be a line,
leading to the conclusion $\Sigma =\partial H$.
\end{remark}

\begin{remark} Without any attempt to make it more precise here,   one could also consider global minimizers in angular sectors instead of half-planes leading to some regularity issues in Lipschitz domains instead of $C^1$ domains. For instance, it is not difficult to see that  in any convex angular sector, the only possible global minimizer is the empty set. This means, for instance, if $\Omega$ is a convex polygone, a minimizer $\Sigma$ of Problem~\ref{pb_compl_pen1} will never go through a corner of $\partial \Omega$.
\end{remark}


\section{Conclusion and full regularity}

In this section, we prove our main result on characterization of minimizers. In particular we are interested in the following statement.

\begin{theorem} \label{main} Assume that $f\in L^p(\Om)$, $p>2$,
where  $\Omega$ is a $C^1$ domain. Then every minimizer $\Sigma$ for Problem~\ref{pb_compl_pen1} consists of a    finite  number of embedded curves whihc are locally $C^{1,\alpha}$ inside $\Omega$ for some $\alpha\in (0,1)$, meeting only by number of three at 120 degree angles. In particular, $\Sigma$ has finite number of endpoints, finite number of branching points (which are all triple points where smooth curves meet at
$120$ degree angles), and at all the other points $\Sigma$ is locally $C^{1,\alpha}$-smooth.
\end{theorem}

\begin{remark}\label{rem_regloc1}
From the proof of  Theorem~\ref{main} we deduce that for an arbitrary bounded open $\Omega\subset \R^2$, without any smoothness condition of the boundary, one has
that  the result analogous to Theorem~\ref{main} holds locally inside $\Omega$, namely, for every $\Omega'\Subset \Omega$
and every minimizer $\Sigma$ for Problem~\ref{pb_compl_pen1} one has
 that $\Sigma\cap \Omega'$ has finite number of branching points (which are all triple points where smooth curves meet at
$120$ degree angles), and at all the other points $\Sigma$ is locally $C^{1,\alpha}$-smooth for some $\alpha\in (0,1)$.
\end{remark}

\begin{remark}\label{rem_regconvOm1}
If under conditions of Theorem~\ref{main} one has additionally that $\overline{\Omega}$ is convex, then every minimizer $\Sigma$ for Problem~\ref{pb_compl_pen1} consists of a    finite  number of $C^{1,\alpha}$ embedded curves in $\overline{\Omega}$ for some $\alpha\in (0,1)$, meeting only by number of three at 120 degree angles.
The fact that the curves of $\Sigma$ are $C^{1,\alpha}$ in $\overline{\Omega}$
comes from Proposition~\ref{TheCor2} which says that $\Sigma$ is
$C^{1,\alpha}$ locally around any ``flat point'' with low energy $\omega_\Sigma$,
and this holds true in $\overline{\Omega}$ provided that $\Omega$ is convex (Proposition~\ref{TheCor2}).
The conclusion then follows similarly to that of Theorem~\ref{main},
in particular any point which is not an endpoint nor a triple point
is a flat point and  Lemma~\ref{decayboundary} (for the boundary case)
or Proposition~\ref{gouter2}
(for the interior case) says  that any ``flat point'' has low energy $\omega_\Sigma$
so that Proposition~\ref{TheCor2} applies.
\end{remark}

The rest of the section is dedicated to the proof of the above Theorem~\ref{main}.

\subsection{Finite number of curves}

To prove the assertion on the finite number of curves,  we follow the approach of Bonnet~\cite{b}
 and start with the following observation.

\begin{proposition} \label{typeS}
Assume that  $f\in L^p(\Om)$, $p>2$,
where  $\Omega$ is a $C^1$ domain, $\Sigma\subset \overline{\Omega}$ be a minimizer for the Problem~\ref{pb_compl_pen1} and let $x\in \Sigma$. Then  all the possible blow-up limits at point $x$  are  of same type.
The same result holds without any condition on the boundary of $\partial\Omega$, once $x\in \Sigma\cap\Omega$.
\end{proposition}

\begin{proof} By Remark~\ref{remYeah} we know that, for any $x\in \Sigma$ the limit
\[e(x):=\lim_{r\to 0^+}\frac{1}{r}\int_{B_r(x)} |\nabla u_{\Sigma} |^2 \, dx\]
exists, and is finite (due to~\eqref{Thebound}). By the strong convergence of the blow-ups of $u_\Sigma$ in $H^1_{loc}(\R^2)$ to their limit,   and
 recalling Proposition~\ref{corr},
we deduce that $e(x)=0$ in the case when there exists a blow-up limit which is a line or a propeller, and $e(x)>0$ only if all the blow-up limits are  half-lines.

On the other hand, if a blow-up limit at point $x \in \Sigma \cap \Omega$ is a line, by local $C^1$ regularity Theorem~\ref{C1reg} we know that all other blow-up limits must also be a line (and if $x \in \partial \Omega$ we also know that all blow-up sequences converge to the tangent lines to $\partial \Omega$).  We then easily conclude that all the blow-up limits have same type: indeed, either    $e(x)>0$ and all blow-ups must be a half-line, or $e(x)=0$. In the latter case, either there exists a blow-up which is a line, and then all other blow-ups must be also lines, or there is no lines and then all blow-ups are propellers.
\end{proof}

Proposition~\ref{typeS} motivates the following terminology.

\begin{definition} We define the type of  $x \in \Sigma$ as follows.
\begin{itemize}
\item[(i)] $x$ is called a regular point if all blow-up limits at point $x$ are lines.
\item[(ii)] $x$ is called a triple point if all blow-up limits at point $x$ are propellers.
\item[(iii)] $x$ is called an endpoint if all blow-up limits at point $x$ are half-lines.
\end{itemize}
\end{definition}

The proof of Theorem~\ref{main} now directly follows from Theorem~\ref{C1reg} together with  the following result.

\begin{theorem}\label{finite}
Assume that  $f\in L^p(\Om)$, $p>2$,
where  $\Omega$ is a $C^1$ domain, and $\Sigma$ be a minimizer for the Problem~\ref{pb_compl_pen1}. Then the set of triple points and endpoints is a finite set.
\end{theorem}

\begin{remark}\label{rem_regloc2}
The proof of  Theorem~\ref{finite}
in the case of an arbitrary bounded open $\Omega\subset \R^2$ without any smoothness condition of the boundary gives that for every $\Omega'\Subset \Omega$
and every minimizer $\Sigma$ for Problem~\ref{pb_compl_pen1} one has
 that $\Sigma\cap \Omega'$  has finite number of endpoints and
finite number of branching points (which are all triple points where smooth curves meet at
$120$ degree angles).
\end{remark}

\begin{proof} We have now all the ingredients at hand to follow the    same proof as~\cite[Theorem 5.1]{b}. The proof is exactly the same as the one in~\cite{b} but let us write the full details here for the convenience of the reader. The strategy is to argue by contradiction. Assuming that there exists a sequence $T_n$ of distinct triple points  in $\Sigma$,   we consider the  blow-up limit
\[u_{\varepsilon_n}(x)=\varepsilon_n^{-\frac{1}{2}}u(T_n+\varepsilon_n R_n(x))\]
where $R_n$ is the rotation that maps ${\bf e}_1$ on $T_n-T_{n+1}$ and $\varepsilon_n=|T_n-T_{n+1}|$. For this transformation, the rescaled set $\Sigma_{\varepsilon_n}=\{x \in \R^2 : T_n+\varepsilon_n R_n(x) \in \Sigma\}$ has at least two triple points, one at $(0,0)$ and another one at $(1,0)$.  Along a subsequence (not relabelled), $\Sigma_{\varepsilon_n}$ converges to some global minimizer $\Sigma_0$ in the plane containing $(0,0)$ and $(0,1)$, and by the classification of blow-up limits, at least one of those two points is not a triple point for $\Sigma_0$. Let us assume that $(0,0)$ is not a triple point. Then it is a cracktip or a regular point, thus, thanks to Proposition~\ref{global1} and Theorem~\ref{global2} respectively,
\begin{eqnarray}
\mathcal{H}^1(\Sigma_{\varepsilon_n}\cap B_r )\underset{n\to +\infty}{\longrightarrow} r \text{ or } 2r \text{, uniformly in } r\in (0,1), \label{conconv}
\end{eqnarray}
because the functions $r\mapsto \mathcal{H}^1(\Sigma_{\varepsilon_n}\cap B_r )$ are all nondecreasing (the balls here and below are centered in the origin).

Now since $T_n$ is a triple point for $\Sigma_{\varepsilon_n}$, we also have, for any $n$ fixed,
$$\lim_{r\to 0^+} \frac{\mathcal{H}^1(\Sigma_{\varepsilon_n}\cap B_r)}{r}=3.$$
Let us define
$$\rho_n=\inf\{r >0 \colon \mathcal{H}^1(\Sigma_{\varepsilon_n}\cap B_r)\leq 2.5 r\}.$$
Because of~\eqref{conconv}, we deduce that $\rho_n\to 0$. Now let $\eta_n:=\rho_n \varepsilon_n$ and let us consider the blow-up sequence $u_{\eta_n}=u(T_n+\eta_n R_n(x))/\sqrt{\eta_n}$, $\Sigma_{\eta_n}=\{x \in \R^2 :  T_n+\eta_n R_n(x) \in \Sigma\}$.
By definition of $\eta_n$ as an infimum we deduce that
$$
2.5 s \leq \mathcal{H}^1(\Sigma_{\varepsilon_n}\cap B_s)\quad \text{for all } s < \rho_n,
$$
which implies, in view of the relation $\mathcal{H}^1(\Sigma_{\eta_n}\cap B_t)=\frac{1}{\rho_n}\mathcal{H}^1(\Sigma_{\varepsilon_n}\cap B_{t\rho_n})$, the estimate
$$
2.5 t \leq \mathcal{H}^1(\Sigma_{\eta_n}\cap B_t)\quad \text{for all } t < 1 .$$
By taking $s=\rho_n$ we also get $ \mathcal{H}^1(\Sigma_{\varepsilon_n}\cap B_{\rho_n}) \leq 2.5$ which in total yields
\begin{eqnarray}
2.5 t \leq \mathcal{H}^1(\Sigma_{\eta_n}\cap B_t) \leq \mathcal{H}^1(\Sigma_{\eta_n}\cap B_1)\leq 2.5 \quad \text{for all } t < 1. \label{densityO}
\end{eqnarray}
By taking now  a subsequence of $\Sigma_{\eta_n}$ converging to a global minimizer, we get a contradiction because none of the list verifies such a density estimate as the one in~\eqref{densityO}
\end{proof}

\subsection{Further regularity}
\label{further}

In this section we derive some regularity of higher order. It relies on the classical elliptic regularity theory and the Euler-Lagrange equation associated to our problem.

\begin{proposition}\label{prop_furtherReg1}
Let $\Omega$ be an open set, $\lambda\in(0,\infty)$ and $f\in H^1(\Om)\cap L^p(\Om)$   with $p>2$.
Let $\Sigma$ be a minimizer for Problem~\ref{pb_compl_pen1}, and $x\in\Sigma\cap\Omega$, $r>0$, $\alpha_{0}\in(0,1)$ such that $\gamma:=\Sigma\cap B_{r}(x)$ is $C^{1,\alpha_{0}}$. Then $\gamma$ is $C^{2,\alpha}$ for $\alpha=1-\frac{2}{p}$, and
\begin{equation}\label{eq:EL}
\left(\frac{\partial u^+}{\partial\nu}\right)^2-\left(\frac{\partial u^-}{\partial\nu}\right)^2+\lambda H_{\gamma}=0\;\;\;\;\;\textrm{ over }\;\gamma,
\end{equation}
where $u^+,u^-$ are the restrictions of $u$ on the two connected components of $B_{r}(x)\setminus\gamma$ oriented in a suitable way, and $H_{\gamma}$ denote the mean curvature. \\
If moreover $f$ is $C^{k,\beta}$ in $B_{r}(x)$, for $k\in\N$ and $\beta\in(0,1)$ then $\gamma$ is $C^{k+3,\beta}$.
\end{proposition}

\begin{proof}
Equality~\eqref{eq:EL} is proven in~\cite{ButMaiSte09}, where $H_{\gamma}$ is understood in a weak sense.
From~\cite[Theorem 8.34 and the remark at the end of section 8.11]{gt}, $u$ is $C^{2,\alpha}$ up to $\gamma$ from each side, where $\alpha=1-\frac{n}{p}$. The regularity theory of the mean curvature equation gives the regularity of $\gamma$. The case $f\in C^{k,\alpha}$ follows from a bootstrap argument.
\end{proof}





\appendix

\section{Auxiliary results}

We introduce the following notation.
For a relatively open set $S\subset \partial B_r(0)\subset\R^n$
$\lambda_1(S)$ the first eigenvalue of the Laplace-Beltrami operator over $H_0^1(S)$.
For instance $n=2$ and $S$ is an arc of a circumference, 
then it is easy to compute
\[
 \lambda_1(S)=\frac{\pi^2}{\H(S)^2}.
\]
This gives a possibility to formulate the following immediate version of the
Poincar\'{e} inequality.

\begin{lemma}\label{lm_compl_poincare1}
For a function $u\in H^1(B_R(0))\cap H_0^1 (K^c)$  for some
closed $K\subset \R^n$.
Then for a.e.\ $r\in (0,R)$ such that $K\cap \partial B_r(x_0)$ has positive harmonic capacity
in $\partial B_r(x_0)$, letting
\[
C:=\sup\left\{\frac{1}{\lambda_1(S)}\colon S\mbox{ connected component of } \partial B_r(x)\setminus K\right\},
\]
one has
\begin{equation}\label{eq_compl_poinc1}
\int_{\partial B_r(x)} u^2\, d\HH^{n-1}\leq C \int_{\partial B_r(x)} |\nabla_\tau u|^2\, d\HH^{n-1}.
\end{equation}
In particular, for $n=2$ one has~\eqref{eq_compl_poinc1} for
a.e.\ $r\in (0,R)$ such that $K\cap \partial B_r(x_0)\neq\emptyset$
with
\begin{equation}\label{eq_compl_poinc2}
C:=\sup\left\{ \left(\frac{\H(S)}{\pi}\right)^2  \colon S\mbox{ connected component of } \partial B_r(x)\setminus K\right\}.
\end{equation}
\end{lemma}

\begin{proof}
One has
\begin{equation}\label{eq_compl_poinc1a}
\int_S u^2\, d\HH^{n-1}\leq \frac{1}{\lambda_1(S)} \int_S |\nabla_\tau u|^2\, d\HH^{n-1}\leq C
\int_S |\nabla_\tau u|^2\, d\HH^{n-1}
\end{equation}
for every connected component $S$ of $\partial B_r(x)\setminus K$
so that  summing~\eqref{eq_compl_poinc1a} over all connected components of $\partial B_r(x)\setminus K$,
we get~\eqref{eq_compl_poinc1}.
\end{proof}

A corollary is the following Poincar\'{e} inequality on the annulus.

\begin{lemma}\label{lm_compl_poincare2}
Let $\Sigma\subset \R^2$ be a closed connected set, 
$\Sigma\cap \overline{B}_r(x_0)\neq\emptyset$ and $\Sigma\setminus B_R(x_0)\neq \emptyset$.
For a function $u\in H^1(B_R(x_0))\cap H_0^1(\Sigma^c)$,
one has
\begin{equation}\label{eq_compl_poinc3a}
\int_{B_R(x_0)\setminus B_r(x_0)} u^2\, dx\leq 4R^2\int_{(B_R(x_0)\setminus B_r(x_0))\setminus\Sigma} |\nabla u|^2\, dx
\end{equation}
for every $r\in (0,R)$.
\end{lemma}

\begin{proof}
For every $\rho\in (r,R)$ one has $\Sigma\cap \partial B_\rho(x_0)\neq\emptyset$, and hence, by Lemma~\ref{lm_compl_poincare1},
observing that in view of~\eqref{eq_compl_poinc2} the Poincar\'{e} constant over each sphere
$\partial B_\rho(x_0)$ with $\rho\in (r,R)$ satisfies
 $C\leq 4\rho^2\leq 4 R^2$,
we get
\[
\int_{\partial B_\rho(x_0)} u^2\, d\H\leq 4R^2 \int_{\partial B_\rho(x_0)\setminus\Sigma} |\nabla_\tau u|^2\, d\H\leq
4R^2 \int_{\partial B_\rho(x_0)\setminus\Sigma} |\nabla u|^2\, d\H.
\]
Integrating the latter equation in $d\rho$ over $\rho\in (r,R)$, one gets~\eqref{eq_compl_poinc3a}.
\end{proof}

We now prove the integration by parts formula.

\begin{lemma}\label{lm_compl_intparts1}
Let $\Sigma$  be a closed set and $f\in L^2(\Omega)$. If $u\in H_0^1(\Omega\setminus\Sigma)$ is a minimizer of $E$ over
$H_0^1(\Omega\setminus\Sigma)$, then for every 
$x_0\in \R^n$ and
a.e. $r>0$ 
one has
\begin{align*}
\int_{B_{r}(x_0)\setminus\Sigma} |\nabla u|^2\, dx = \int_{\partial B_{r}(x_0)\setminus\Sigma} u\frac{\partial u}{\partial\nu}\,d\HH^{n-1} +  \int_{B_{r}(x_0)} f u\, dx,
\end{align*}
where $\frac{\partial u}{\partial\nu}:=\nabla u\cdot \nu$, $\nu$ standing for the normal to $\partial B_r(x_0)$.
\end{lemma}

\begin{proof}
From the relationship
\[
\int_{\Omega\setminus \Sigma} \nabla u\cdot\nabla\varphi\, dx =  \int_{\Omega} f \varphi\, dx
\]
for every $\varphi\in H_0^1(\Omega\setminus\Sigma)$, taking $\varphi(x):=g_\varepsilon(|x-x_0|)u(x)$, where
$g_\varepsilon\in \Lip(\R^+)$ be defined as
\[
g_\varepsilon(x) :=
\left\{
\begin{array}{rl}
1, & x\leq r-\varepsilon, \\
0, & x\geq r,\\
-\frac{1}{\varepsilon}(x-r), & r-\varepsilon<x<r
\end{array}
\right. ,
\]
and $\varepsilon\in (0,r)$ is arbitrary, we get
\begin{equation}\label{eq_compl_parts1}
\begin{aligned}
\int_{\Omega\setminus \Sigma} |\nabla u|^2 g_\varepsilon(|x-x_0|)\, dx  & +
\int_{\Omega\setminus \Sigma} u\nabla u\cdot g'_\varepsilon(|x-x_0|)\frac{x-x_0}{|x-x_0|}\, dx
\\
& =  \int_{\Omega} f u g_\varepsilon(|x-x_0|)\, dx.
\end{aligned}
\end{equation}
Letting $\varepsilon\to 0^+$, we have
\begin{equation}\label{eq_compl_parts2}
\begin{aligned}
\int_{\Omega\setminus \Sigma} |\nabla u|^2 g_\varepsilon(|x-x_0|)\, dx & \to
\int_{B_r(x_0)\cap\Omega\setminus \Sigma} |\nabla u|^2 \, dx\\
\int_{\Omega} f u g_\varepsilon(|x-x_0|)\, dx & \to \int_{\Omega} f u dx,
\end{aligned}
\end{equation}
and
\begin{equation}\label{eq_compl_parts3}
\begin{aligned}
\int_{\Omega\setminus \Sigma} u\nabla u\cdot g'_\varepsilon(|x-x_0|)\frac{x-x_0}{|x-x_0|}\, dx & =
-\frac{1}{\varepsilon}\int_{(B_r(x_0)\setminus
\overline{B}_{r-\varepsilon}(x_0))\cap\Omega \setminus \Sigma} u\nabla u \cdot \frac{x-x_0}{|x-x_0|}\, dx \\
&=
-\frac{1}{\varepsilon}\int_{r-\varepsilon}^r \,d\rho \int_{\partial B_\rho(x_0)\cap\Omega \setminus \Sigma} u\nabla u \cdot \frac{x-x_0}{\rho}\, d\HH^{n-1} \\
& \qquad\qquad\qquad \mbox{ by coarea formula }\\
& \to \int_{\partial B_{r}(x_0)\cap\Omega \setminus\Sigma} u\frac{\partial u}{\partial\nu}\,d\HH^{n-1}(x),
\end{aligned}
\end{equation}
for a.e. $r>0$,
because
the function
\[
\rho\in (0, r)\mapsto
\int_{\partial B_\rho(x_0) \setminus \Sigma} u\nabla u \cdot \frac{x-x_0}{\rho}\, d\HH^{n-1}(x)
\]
defined for a.e.\ $\rho\in (0,r)$
belongs to $L^1(0,r)$ for every $r>0$.

Plugging~\eqref{eq_compl_parts2} and~\eqref{eq_compl_parts3} into~\eqref{eq_compl_parts1}, we get the desired assertion.
\end{proof}

\bibliographystyle{plain}

\end{document}